\documentclass[a4paper,12pt]{article}
\usepackage[margin=1.1in]{geometry}  

\usepackage{graphicx}              
\usepackage{amsmath}
\usepackage{amscd}               
\usepackage{amsfonts} 
\usepackage{amssymb}             
\usepackage{amsthm}                
\usepackage{mathrsfs}
\usepackage{amsbsy}
\usepackage{tikz}
\usepackage{booktabs}
\usetikzlibrary{decorations.pathmorphing}
\usepackage{hyperref}

\usepackage{biblatex}
\hypersetup{pdfstartview=}

\numberwithin{equation}{section}

\newtheorem{thm}{Theorem}[section]
\newtheorem{defn}[thm]{Definition}
\newtheorem{lem}[thm]{Lemma}
\newtheorem{prop}[thm]{Proposition}
\newtheorem{cor}[thm]{Corollary}

\newtheorem{assumption}[thm]{Assumption}
\DeclareMathOperator{\id}{id}

\def\XXint#1#2#3{{\setbox0=\hbox{$#1{#2#3}{\int}$ }
\vcenter{\hbox{$#2#3$ }}\kern-.6\wd0}}

\makeatletter
\def\th@newremark{\th@remark\thm@headfont{\bfseries}}
\makeatletter

\theoremstyle{newremark}
\newtheorem{rmk}[thm]{Remark}

\newcommand{\m}{\mathbf{m}}

\newcommand{\aA}{\mathcal{A}}

\newcommand{\cC}{\mathcal{C}}
\newcommand{\dD}{\mathcal{D}}

\newcommand{\fF}{\mathcal{F}}

\newcommand{\iI}{\mathcal{I}}
\newcommand{\jJ}{\mathcal{J}}
\newcommand{\kK}{\mathcal{K}}
\newcommand{\lL}{\mathcal{L}}
\newcommand{\mM}{\mathcal{M}}

\newcommand{\rR}{\mathcal{R}}
\newcommand{\sS}{\mathcal{S}}
\newcommand{\tT}{\mathcal{T}}

\newcommand{\vV}{\mathcal{V}}
\newcommand{\wW}{\mathcal{W}}
\newcommand{\xX}{\mathcal{X}}
\newcommand{\yY}{\mathcal{Y}}
\newcommand{\zZ}{\mathcal{Z}}

\newcommand{\E}{\mathbf{E}}
\newcommand{\N}{\mathbf{N}}
\renewcommand{\P}{\mathbf{P}}
\newcommand{\R}{\mathbf{R}}
\newcommand{\T}{\mathbf{T}}
\newcommand{\Z}{\mathbf{Z}}

\renewcommand{\id}{\mathrm{id}}
\newcommand{\eps}{\varepsilon}
\renewcommand{\d}{\partial}
\renewcommand{\div}{\mathrm{div }}
\newcommand{\supp}{\mathrm{supp }}

\newcommand{\fs}{\mathfrak{s}}

\newcommand{\sT}{\mathsf{T}}

\definecolor{darkgreen}{rgb}{0.1,0.7,0.1}
\definecolor{darkred}{rgb}{0.7,0.1,0.1}
\definecolor{darkblue}{rgb}{0,0,0.7}
\def\bigscal#1{\big\langle#1\big\rangle}
\def\Bigscal#1{\Big\langle#1\Big\rangle}

\colorlet{symbols}{blue!90!black}
\colorlet{testcolor}{green!60!black}

\def\drawx{\draw[-,solid] (-3pt,-3pt) -- (3pt,3pt);\draw[-,solid] (-3pt,3pt) -- (3pt,-3pt);}
\tikzset{
	root/.style={circle,fill=testcolor,inner sep=0pt, minimum size=2mm},
	dot/.style={circle,fill=black,inner sep=0pt, minimum size=1mm},
	var/.style={circle,fill=black!10,draw=black,inner sep=0pt, minimum size=
	2mm},
	dotred/.style={circle,fill=black!50,inner sep=0pt, minimum size=2mm},
	generic/.style={semithick,shorten >=1pt,shorten <=1pt},
	dist/.style={ultra thick,draw=testcolor,shorten >=1pt,shorten <=1pt},
	testfcn/.style={ultra thick,testcolor,shorten >=1pt,shorten <=1pt,<-},
	testfcnx/.style={ultra thick,testcolor,shorten >=1pt,shorten <=1pt,<-,
		postaction={decorate,decoration={markings,mark=at position 0.6 with {\drawx}}}},
	kepsilon/.style={semithick,shorten >=1pt,shorten <=1pt,densely dashed,->},
	kprimex/.style={semithick,shorten >=1pt,shorten <=1pt,densely dashed,->,
		postaction={decorate,decoration={markings,mark=at position 0.4 with {\drawx}}}},
	kernel/.style={semithick,shorten >=1pt,shorten <=1pt,->},
	multx/.style={shorten >=1pt,shorten <=1pt,
		postaction={decorate,decoration={markings,mark=at position 0.5 with {\drawx}}}},
	kernelx/.style={semithick,shorten >=1pt,shorten <=1pt,->,
		postaction={decorate,decoration={markings,mark=at position 0.4 with {\drawx}}}},
	kernel1/.style={->,semithick,shorten >=1pt,shorten <=1pt,postaction={decorate,decoration={markings,mark=at position 0.45 with {\draw[-] (0,-0.1) -- (0,0.1);}}}},
	kernel2/.style={->,semithick,shorten >=1pt,shorten <=1pt,postaction={decorate,decoration={markings,mark=at position 0.45 with {\draw[-] (0.05,-0.1) -- (0.05,0.1);\draw[-] (-0.05,-0.1) -- (-0.05,0.1);}}}},
	kernelBig/.style={semithick,shorten >=1pt,shorten <=1pt,decorate, decoration={zigzag,amplitude=1.5pt,segment length = 3pt,pre length=2pt,post length=2pt}},
	gepsilon/.style={dotted,semithick,shorten >=1pt,shorten <=1pt},
	renorm/.style={shape=circle,fill=white,inner sep=1pt},
	labl/.style={shape=rectangle,fill=white,inner sep=1pt},
	xi/.style={circle,fill=symbols!10,draw=symbols,inner sep=0pt,minimum size=1.2mm},
	xix/.style={crosscircle,fill=symbols!10,draw=symbols,inner sep=0pt,minimum size=1.2mm},
	xib/.style={circle,fill=symbols!10,draw=symbols,inner sep=0pt,minimum size=1.6mm},
	xibx/.style={crosscircle,fill=symbols!10,draw=symbols,inner sep=0pt,minimum size=1.6mm},
	not/.style={circle,fill=symbols,draw=symbols,inner sep=0pt,minimum size=0.5mm},
	>=stealth,
	}

\makeatletter
\def\DeclareSymbol#1#2#3{\expandafter\gdef\csname MH@symb@#1\endcsname{\tikz[baseline=#2,scale=0.15,draw=symbols]{#3}}\expandafter\gdef\csname MH@symb@#1s\endcsname{\scalebox{0.7}{\tikz[baseline=#2,scale=0.15,draw=symbols]{#3}}}}
\def\<#1>{\csname MH@symb@#1\endcsname}
\makeatother

\DeclareSymbol{Xi22}{0.5}{\draw (0,0) node[xi] {} -- (-1,1) node[not] {} -- (0,2) node[xi] {}; }

\begin{document}

\title{Periodic homogenisation for two dimensional generalised parabolic Anderson model}

\author{
Yilin Chen\footnote{Email: yilin\_chen@pku.edu.cn}\\
\textit{Peking University}\\
\and
Benjamin Fehrman\footnote{Email: fehrman@lsu.edu}\\
\textit{Louisiana State University}\\
\and
Weijun Xu\footnote{Corresponding author. Email: xuweijun@westlake.edu.cn}\\
\textit{Westlake University}
}

\maketitle

\abstract{We consider the periodic homogenisation problem for the generalised parabolic Anderson model on the two dimensional torus. We show that, for the renormalisation that respects Wick ordering, the homogenisation and renormalisation procedures commute. The main novelty is to identify a suitable solution ansatz beyond the usual para-controlled ansatz to set up a fixed point problem uniform in the homogenisation parameter. After that, one further utilises cancellations and resonances from the homogenisation oscillations to show convergences of both the solution and flux to the right limits. 

At a technical level, we frequently use integration by parts as well as ``completing the products" to circumvent the incompatibility between para-products and variable coefficients. As a byproduct, we also show that the standard two dimensional generalised parabolic Anderson model can be constructed with para-controlled calculus without using commutator estimates.}

\tableofcontents

\section{Introduction}

The aim of this article is to study the $\eps \rightarrow 0$ limit (along $\eps = \eps_N = \frac{1}{N}$) of the periodic homogenisation problem for the generalised parabolic Anderson model (gPAM) on the two dimensional torus $\T^2 = (\R / \Z)^2$, formally written as
\begin{equation} \label{e:gPAM_formal}
    \d_t u_\eps = \lL_\eps u_\eps  + g(u_\eps) \big( \xi - \infty_\eps \, g'(u_\eps) \big)\;, \quad (t,x) \in \R^{+} \times \T^2.
\end{equation}
Here, $\lL_\eps = \div \big( a( \cdot /\eps) \nabla \big)$ with a $1$-periodic, elliptic and three times differentiable $2 \times 2$ matrix $a$, and $\eps = \eps_N = \frac{1}{N}$ is the oscillation parameter along the inverse of the integers. The term $\xi$ on the right hand side is a spatial white noise on $\T^2$, a Gaussian generalised function with formal correlation $\E \xi(x) \xi(y) = \delta (x-y)$. The term ``$\infty_\eps$'' is formally the $\eps$-dependent quantity
\begin{equation} \label{e:formal_infinity}
    \infty_\eps ``=" \, \E \, \big( \nabla X_\eps \cdot a_\eps \nabla X_\eps \big)\;,
\end{equation}
where $X_\eps$ is the time independent solution to the linear equation
\begin{equation} \label{e:linear_solution}
    - \lL_\eps X_\eps = \Pi_0^\perp \xi := \xi - \int_{\T^2} \xi(y) \, {\rm d}y\;.
\end{equation}
Here, the average of $\xi$ is subtracted to ensure \eqref{e:linear_solution} is well-posed. Before we make the above formal expressions and equation precise, we first give our assumptions on $a$ and $g$. 

\begin{assumption} \label{a: assumption for a and g}
    The coefficient $a: \T^2 \rightarrow \R^{2 \times 2}$ is $1$-periodic, uniformly elliptic, and has three continuous derivatives. The function $g: \R \rightarrow \R$ has three bounded derivatives. 
\end{assumption}

We can now be more precise about the meaning of \eqref{e:gPAM_formal} and \eqref{e:formal_infinity}. Let $\xi^{(\delta)}$ be a sequence of smooth approximations to $\xi$, and $u_\eps^{(\delta)}$ be the solution to
\begin{equation} \label{e:gPAM_regularised}
    \d_t u_\eps^{(\delta)} = \lL_\eps u_\eps^{(\delta)} + g(u_\eps^{(\delta)}) \big( \xi^{(\delta)} - C_{\eps}^{(\delta)} g'(u_\eps^{(\delta)}) \big)\;.
\end{equation}
Then, it is well known (\cite{para_control_SPDE, rs_theory, Hairer-Labbe, SPDE_variable_nontrans}) that for every $\eps \in \N^{-1}$ and for the choice
\begin{equation} \label{e:renormalisation}
    C_\eps^{(\delta)}= \E \big( \nabla X_\eps^{(\delta)} \cdot a_\eps \nabla X_\eps^{(\delta)} \big)\;,
\end{equation}
where $X_\eps^{(\delta)}$ is the solution to the linear equation \eqref{e:linear_solution} with $\xi$ replaced by $\xi^{(\delta)}$, the solution $u_\eps^{(\delta)}$ to \eqref{e:gPAM_regularised} converges at fixed times in $\cC^{\alpha}$ for every $\alpha < 1$ to a unique (random) limit $u_\eps$ as $\delta \rightarrow 0$. Furthermore, the limit $u_\eps$ is independent of the choice of approximation $\{\xi^{(\delta)}\}$\footnote{That the limit with the particular choice of renormalisation \eqref{e:renormalisation} does not depend on the choice of approximation may not seem obvious at this stage. It will be a consequence of the proof later that all objects in the limiting system \eqref{e:fixed_pt_system_limit} are universal, and do not depend on the choice of regularisation.}. We call this limit $u_\eps$ the solution to the generalised parabolic Anderson model \eqref{e:gPAM_formal}. For fixed $\eps>0$, the renormalisation function $C_\eps^{(\delta)}$ explodes logarithmically in $\delta$, hence suggesting the formal notation \eqref{e:formal_infinity}. Note that $C_\eps^{(\delta)}$ is independent of time, but depends on the spatial variable $x$ since $a$ is not homogeneous. When $a$ is a constant matrix, it also becomes a constant, and agrees with the choice in \cite{Hairer-Labbe}. 

On the other hand, starting from \eqref{e:gPAM_regularised}, one can also fix $\delta > 0$ and take the $\eps \rightarrow 0$ limit, which corresponds to periodic homogenisation with a smooth right hand side. It is standard that the operator $\lL_\eps$ homogenises to
\begin{equation*}
    \lL_0 := \div \big(\bar{a} \nabla \big)\;
\end{equation*}
in the sense of convergence of the associated Green's functions, where $\bar{a}$ is the homogenised matrix of $a$, a constant $2 \times 2$ elliptic matrix which is in general \textit{different} from the average of $a$ on the torus. A precise expression of $\bar{a}$ is given in \eqref{e:homogenised_matrix}. 

Back to \eqref{e:gPAM_regularised}, the above classical homogenisation results imply that for every fixed $\delta>0$, along $\eps = \eps_N = \frac{1}{N} \rightarrow 0$, $u_\eps^{(\delta)}$ converges to a unique limit $u_0^{(\delta)}$, which solves the constant coefficient parabolic random PDE
\begin{equation} \label{e:gPAM_constant}
    \d_t u_0^{(\delta)} = \lL_0 u_0^{(\delta)} + g \big( u_0^{(\delta)} \big) \big( \xi^{(\delta)} - C^{(\delta)} g'(u_0^{(\delta)}) \big)\;,
\end{equation}
where
\begin{equation} \label{e:renormalisation_constant}
    C^{(\delta)} = \E \big( \nabla X_0^{(\delta)} \cdot\bar{a} \nabla X_0^{(\delta)} \big)\;,
\end{equation}
is the $\eps \rightarrow 0$ (weak) limit of $C_\eps^{(\delta)}$. Here, $X_0^{(\delta)}$ is the solution to \eqref{e:linear_solution} with $\eps = 0$ (that is, defined by $\mathcal{L}_0$). Now, \eqref{e:gPAM_constant} is the standard constant coefficient $2$D gPAM, and hence we know $u_0^{(\delta)}$ converges in $\cC^{\alpha}$ for $\alpha<1$ (at fixed times) to a unique limit $u_0$. 

The main result of this article is that the solution $u_\eps$ to \eqref{e:gPAM_formal}, defined as the $\delta \rightarrow 0$ limit of $u_\eps^{(\delta)}$ above, converges to the same $u_0$ as $\eps \rightarrow 0$ up to the time where $u_0$ is defined. In other words, the procedures of renormalisation and periodic homogenisation commute in this case. 

In the two recent advances \cite{2d_gpam_global_cfw, 2d_gpam_global_szz}, the authors showed that the solution $u_0$ as the $\delta \rightarrow 0$ limit from \eqref{e:gPAM_constant} exists globally in time. When combined with our main convergence result, this implies that the convergence of $u_\eps$ to $u_0$ holds up to arbitrary time. The main theorem is stated as follows.

\begin{thm} \label{thm:main}
    Let $\alpha \in (0,1)$. For every $\eps = \eps_N = \frac{1}{N}$, let $u_\eps$ be the solution to \eqref{e:gPAM_formal} as specified above with initial data $u_\eps (0) \in \cC^{\alpha}$. Suppose there exists $\alpha'>\alpha$ and $u_0(0) \in \cC^{\alpha'}$ such that $u_\eps (0) \rightarrow u_0(0)$ in $\cC^{\alpha}$. Let $u_0$ be the solution to \eqref{e:gPAM_formal} with initial data $u_0(0)$ and $\lL_\eps$ replaced by $\lL_0$ described above. 
    
    Then for every $T>0$ and almost every realisation of $\xi$, $\{u_\eps\}$ is uniformly bounded in $\cC_\fs^\alpha ([0,T] \times \T^2)$ along the sequence $\eps_N = \frac{1}{N}$, and $u_\eps \rightarrow u_0$ in $\cC_\fs^{\alpha}([0,T] \times \T^2)$ in probability. The space $\cC_\fs^{\alpha}$ is defined in \eqref{e:spacetime}. 
    
    Furthermore, the flux $a_\eps \nabla u_\eps$ also converges to $\bar{a} \nabla u_0$ in a proper space as specified in Theorem~\ref{thm:flux_convergence} below. 
\end{thm}
\begin{proof}
    We prove it by assuming Theorems~\ref{thm:fixed_pt_thm}, ~\ref{thm:flux_convergence} and~\ref{thm:stochastic}. The verifications of these assumed statements are provided in corresponding sections below. 
    
    We first note that the stochastic objects in \eqref{e:noises_limit} are precisely the same as the ones from the two dimensional generalised parabolic Anderson model. Hence, the solution $u_0$ constructed from the fixed point problem in Theorem~\ref{thm:fixed_pt_thm} coincides with the solution to the standard $2$D gPAM (as the limit of \eqref{e:gPAM_constant}). Hence, by \cite{2d_gpam_global_cfw, 2d_gpam_global_szz}, for every $T>0$ and every realisation of $\xi$, we have the bound
    \begin{equation*}
        \|u_0\|_{\cC_\fs^\alpha([0,T] \times \T^2)} + \|u_0\|_{L_{T}^\infty \cC_x^{\alpha'}} \lesssim 1 + \|u_0\|_{\cC^{\alpha+\delta}}\;.
    \end{equation*}
    The convergence of $u_\eps$ to $u_0$ on $\cC_\fs^\alpha ([0,T] \times \T^2)$ then follows from Theorem~\ref{thm:fixed_pt_thm} and the convergence of the stochastic objects in Theorem~\ref{thm:stochastic}. The convergence of the flux $a_\eps \nabla u_\eps$ to $\bar{a} \nabla u_0$ is the content of Theorem~\ref{thm:flux_convergence} together with Theorem~\ref{thm:stochastic}. 
\end{proof}

In order to emphasise the main new inputs and keep the overall presentation concise, we have made a number of restrictions and assumptions, which we briefly describe below. We also indicate how one could remove or improve them. 
\begin{enumerate}
    \item We take limits along $\eps = \eps_N = \frac{1}{N}$ to keep the problem on the torus. One could study the problem \eqref{e:gPAM_formal} on a bounded domain in $\R^2$ with a nice boundary. For fixed $\eps$, the solution theory for \eqref{e:gPAM_formal} on bounded domains can be obtained by combining the results in \cite{SPDE_boundary} and \cite{SPDE_variable_nontrans}. To obtain uniform boundedness and convergence as $\eps \rightarrow 0$, one needs to use the Dirichlet correctors (\cite{Geng2023GaussianBA}) instead of the period correctors $\chi$ as in the current article. We expect it should in principle follow from the methods in this article, but will be much more technically involved. 

    \item We assume $a$ being time-independent for technical simplicity. One can work with space-time periodic coefficients $a$ (as in \cite{st_homo_phi42}). We assume $g$ has three bounded derivatives, which could be relaxed to having derivatives with polynomial growth. We also assume that for each fixed $\eps$, the initial data $u_\eps(0)$ has the same regularity as the solution, and that $u_0(0)$ has slightly higher regularity. This is mainly for conciseness of the presentation --- one does not need to add weight to the solution space for $u_\eps$ at $t = 0$, and that the convergence $u_\eps \rightarrow u_0$ can be in the same space where $u_\eps$ lives. 
\end{enumerate}

Periodic homogenisation has a long history. An overview of the theory of periodic homogenisation for divergence form elliptic equations can be found in \cite{BenLioPap2011, JikKozOle1994}. The book \cite{Shen2018} is also comprehensive and contains more recent developments. Homogenisation problems with singular stochastic forcing have also been studied recently. We mention two interesting works \cite{SPDE_multiscale} and \cite{add_functional_rp}, where the authors discovered that homogenisation has a nontrivial effect on other parts of the limiting equations/objects (beyond that the homogenised coefficient being different from the average). 

Periodic homogenisation for nonlinear singular stochastic PDEs that require renormalisations are also natural questions to investigate, as they involve understanding the interactions between two singular limiting procedures -- homogenisation and renormalisation. The case of dynamical $\phi^4_2$ has been treated in \cite{homo_pphi2, st_homo_phi42}. In that model, the ansatz is relatively simple via standard Da Prato-Debussche method. This is not the case for the generalised parabolic Anderson model in 2D, where even for \eqref{e:gPAM_constant} with a constant coefficient Laplacian, one needs a nontrivial ansatz provided by regularity structures \cite{rs_theory} or para-controlled distributions \cite{para_control_SPDE}. The regularity structures framework also extends to situations of variable coefficients with sufficient regularity (see \cite{rs_manifolds, SPDE_variable_nontrans}), in particular covering the solution theory for \eqref{e:gPAM_formal} for every fixed $\eps$. 

However, the solution theories so far do not give uniform-in-$\eps$ control of the solutions, mainly because there is no uniform-in-$\eps$ Schauder estimate for the Green's function of $\d_t - \lL_\eps$, and because the ansatzes from homogenisation (two-scale expansion) and singular stochastic PDEs (Taylor-like expansions from regularity structures or para-controlled distributions) are not easily compatible with each other. The key to our proof of Theorem~\ref{thm:main} is to identify a suitable ansatz to complete a fixed point problem for \eqref{e:gPAM_formal} uniformly in the homogenisation parameter $\eps$. This ansatz could be viewed as a higher order refinement of the standard para-controlled ansatz in \cite{para_control_SPDE} adapted to the operator $\d_t - \lL_\eps$. We expect similar techniques can be applied to other models such as dynamical $\phi^4_3$, but technically much more involved, since Fourier methods based on trigonometric functions are not conveniently compatible with variable coefficients. 

The two very recent works \cite{para_operator_homo, homo_spde_general} aim at developing more unified and systematic procedures for such questions. In \cite{para_operator_homo}, we built a variant of para-controlled calculus from eigen-functions of the oscillatory operator $\lL_\eps$ (instead of those of the Laplacian), which enables us to treat periodic homogenisation for singular SPDEs in a relatively unified way. On the other hand, our framework is currently restricted to time-independent and symmetric coefficient matrices $a$, since spectral theory for self-adjoint operators is heavily used. In \cite{homo_spde_general}, based on the previous work \cite{st_homo_phi42}, Hairer and Singh successfully implemented periodic homogenisation operators into regularity structures, thus giving a general and systematic framework for periodic homogenisation for a broad class of singular SPDEs. It also allows spacetime oscillatory coefficients (and without assumption on symmetry). We refer the readers to \cite{para_operator_homo, homo_spde_general} for more details of the developments and the precise examples covered there. 

Many of the techniques developed in this paper are based on fine studies of the asymptotic behaviour of the parabolic Green's function related to the periodic operator $\mathcal{L}_\varepsilon$, which are obtained in \cite{GengShen2020,GengShen2015,Geng2020}. These in turn are based on techniques in early works \cite{FLinCompactness,FLinCompactness1989}, which introduced compactness methods to establish uniform regularity and Green's function estimates for elliptic systems. In the context of stochastic homogenisation, the asymptotic behaviour of Green's functions is very different and in general much subtler. Hence, we expect problems combining singular stochastic PDEs and stochastic homogenisation will be a much greater challenge, and will be worth investigation in the long term. We refer to \cite{Kozlov1980,Papanicolaou1981,PapVar1982} for early works on qualitative study of stochastic homogenisation, \cite{Yurinskii1986} for a first quantitative result, and \cite{AKM_homo, Armstrong2016, GloNeuOtt2020, GloriaOtto2015, GloriaOtto2012, OttoMarahrens2015} for recent breakthroughs in quantitative stochastic homogenisation. We also refer to \cite{AK_homo_book, AKM_homo_book} for overviews of this subject.

\subsection*{Remarks on homogenisation and renormalisation}

The statement of Theorem~\ref{thm:main} provides a solution theory to \eqref{e:gPAM_formal} for every fixed $\eps$, and gives convergence $u_\eps \rightarrow u_0$. Combined with the discussions above the theorem, it also implies that with the choice of $C_\eps^{(\delta)}$ in \eqref{e:renormalisation}, the homogenisation and renormalisation procedures commute for the solution $u_\eps^{(\delta)}$ to \eqref{e:gPAM_regularised}. 

To see this, we first note that, for fixed $\delta>0$ the limiting procedure $u_\eps^{(\delta)} \rightarrow u_0^{(\delta)}$ in $L_T^\infty \cC_x^{\alpha}$ for $\alpha<1$ is standard periodic homogenisation. In particular, the homogenised limit $u_0^{(\delta)}$ solves \eqref{e:gPAM_constant} with renormalisation constant $C^{(\delta)}$ given in \eqref{e:renormalisation_constant}, which is also the weak limit of $C_\eps^{(\delta)}$ as $\eps \rightarrow 0$. The convergence of $u_0^{(\delta)} \rightarrow \bar{u}_0$ in $L_T^\infty \cC_x^\alpha$ is then standard. 

For the other direction, for fixed $\eps>0$, it follows from the continuity of the solution map $(\Upsilon_\eps, u_\eps(0)) \mapsto (u_\eps, v_\eps, w_\eps)$ (as part of Theorem~\ref{thm:main}) that $u_\eps^{(\delta)} \rightarrow u_\eps$ in $L_T^\infty \cC_x^{\alpha}$ as $\delta \rightarrow 0$, where $\Upsilon_\eps$ is the collection of the enhanced stochastic objects given in \eqref{e:noises_eps}. Then, the most important ingredient is Theorem~\ref{thm:main}, giving the convergence $u_\eps \rightarrow u_0$. Finally, note that the solution $\bar{u}_0$ to the standard constant coefficient gPAM can be described exactly by the fixed point system \eqref{e:fixed_pt_system_limit}. This shows that $\bar{u}_0 = u_0$, concluding that homogenisation and renormalisation procedures do commute in this situation. 

We now give some remarks on the choice of the renormalisation function $C_\eps^{(\delta)}$ in \eqref{e:renormalisation}. The choice $C_\eps^{(\delta)}$ is time independent but depends on the spatial variable $x \in \T^2$ if $a$ is non-constant, and this inhomogeneity in general gives an infinite degrees of freedom in choice of renormalisation. In \cite{SPDE_variable_nontrans}, Singh proposed a natural way to produce finite dimensional family of counter terms when the coefficient $a$ has sufficient regularity. But this is not the case for homogenisation problems, where uniformity in the oscillation parameter is required. 

In the specific situation \eqref{e:gPAM_regularised}, the choice $\widetilde{C}_\eps^{(\delta)}$ in \cite{SPDE_variable_nontrans}, as a function of $x \in \T^2$, is proportional to $\big( \det \big( a(x/\eps) \big) \big)^{-\frac{1}{2}}$. The main theorem in \cite{SPDE_variable_nontrans} implies that for each fixed $\eps$, the difference $C_\eps^{(\delta)} - \widetilde{C}_\eps^{(\delta)}$ converges to a continuous function on $\T^2$ as $\delta \rightarrow 0$. But as discussed in \cite{homo_pphi2}, there exists $\lambda(\delta) \sim |\log \delta|$ such that
\begin{equation*}
    \int_{\T^2} C_\eps^{(\delta)}(x) {\rm d}x \rightarrow \frac{\lambda(\delta)}{\sqrt{\det \bar{a}}}\;, \qquad \int_{\T^2} \widetilde{C}_\eps^{(\delta)}(x) {\rm d}x \rightarrow \lambda(\delta) \int_{\T^2} \big( \det ( a(y) ) \big)^{-\frac{1}{2}} {\rm d}y
\end{equation*}
as $\eps \rightarrow 0$. Since the two coefficients multiplying $\lambda (\delta)$ are not equal if $a$ is not constant matrix, the two renormalisations diverge far apart as $\delta \rightarrow 0$ if one wants uniformity in $\eps$. Similar as in \cite{st_homo_phi42}, one can add further $\eps$-dependent counter terms to $\widetilde{C}_\eps^{(\delta)}$ in a systematic way to obtain convergence also in $\eps$. The choice of $C_\eps^{(\delta)}$ in \eqref{e:renormalisation_constant} is also natural in the sense that it respects the natural Wick ordering, and agrees with the one in \cite{Hairer-Labbe} when $a$ is a constant matrix. 

In \cite{homo_spde_general}, Hairer and Singh gave a systematic way to analyse renormalisation functions, allowing to understand general behaviour of the counter-terms in different regimes of homogenisation and regularisation parameters.

\subsection*{Notations and definitions}

\begin{flushleft}
    \textbf{Basic notations}
\end{flushleft}
Let $\T^2 = (\R / \Z)^2$ be the two dimensional torus of size $1$. For $f: \T^2 \rightarrow \R^m$, $\nabla f$ is the usual $2 \times m$ gradient vector. When $f$ is scalar valued, we write $\nabla^\sT f := (\nabla f)^{\sT}$ for transpose of the vector $\nabla f$. 

For functions/distributions $f$ on $\T^2$ (scalar, vector, or matrix valued), let $\Pi_0 f$ be the projection of $f$ onto its $0$-th Fourier mode, and $\Pi_0^\perp = \id - \Pi_0$. More precisely, we write
\begin{equation*}
    \Pi_0 f := \int_{\T^2} f(y) {\rm d}y\;, \qquad \Pi_0^\perp f = f - \Pi_0 f\;.
\end{equation*}
For $j \geq -1$, $\Delta_j$ denotes the $j$-th Littlewood Paley block (defined in Appendix~\ref{sec:para-products}). For $\m \geq 1$, define operators $\P_\m$ and $\P_\m^\perp$ by
\begin{equation} \label{e:Littlewood_Paley_cutoff}
    \P_\m f := \sum_{j \, \leq \, \log_2 \m} \Delta_j f\;, \qquad \P_\m^\perp f := f - \P_\m f\;.
\end{equation}
The definitions of the para-products $\prec, \succ$ and $\circ$ are given in Appendix~\ref{sec:para-products}. To simplify notations, we also write
\begin{equation*}
    f \preceq g = f \prec g + f \circ g\;, \qquad f \succeq g = f \succ g + f \circ g\;.
\end{equation*}
For a function space $\zZ$, $T>0$ and $p \in [1,+\infty]$, we write $L_T^p \zZ = L^p([0,T]; \zZ)$. For $\sigma>0$, we define the weighted norm $L_{T}^{\infty,\sigma} \zZ$ by
\begin{equation} \label{e:space_weight}
    \|f\|_{L_T^{\infty,\sigma} \zZ} := \sup_{t \in [0,T]} \Big( t^{\frac{\sigma}{2}} \|f(t)\|_{\zZ} \Big)\;.
\end{equation}
In practice, when a spacetime function is involved, we often write $\zZ_x$ instead of $\zZ$ to emphasise that the norm is taken with respect to the spatial variable (at fixed time); for example $\|f(t)\|_{\zZ_x}$ or $\|f\|_{L_T^p \zZ_x}$. 

For $\gamma \in (0,1]$, we use $\cC_\fs^{\gamma}$ to denote the norm
\begin{equation} \label{e:spacetime}
    \|f\|_{\cC_\fs^\gamma([0,T] \times \T^2)} := \|f\|_{L_{T,x}^\infty} + \sup_{\stackrel{0 < s < t < T}{x,y \in \T^2}} \bigg( \frac{|f(t,x) - f(s,y)|}{\big( \sqrt{|t-s|} + |x-y| \big)^\gamma} \bigg)\;.
\end{equation}
Note that $\cC_\fs^\gamma$ is equivalent to $L_T^\infty \cC_x^\gamma \cap \cC_T^{\gamma/2} L_x^\infty$. Finally, we use $\diamond$ to denote Wick product between two Gaussians.

\begin{flushleft}
    \textbf{Homogenisation quantities}
\end{flushleft}
Recall the coefficient matrix $a: \T^2 \rightarrow \R^{2 \times 2}$ in Assumption~\ref{a: assumption for a and g}. We write $a_\eps = a(\cdot/\eps)$. Let $\chi = (\chi_1, \chi_2)^{\sT}: \T^2 \rightarrow\R^2$ be the corrector for $a$, which is the unique solution to
\begin{equation}\label{eq: homo_corrector}
    \div \big( a (\id + \nabla \chi) \big) = 0
\end{equation}
subject to $\int_{\T^2} \chi = 0$. Let $\chi^*$ be the corrector for the transposed matrix $a^{\sT}$. Let
\begin{equation} \label{e:homogenised_matrix}
   \bar{a} := \int_{\T^2} a (\id + \nabla \chi) {\rm d}x
\end{equation}
be the homogenised matrix. Define $\Phi_\eps, \Phi_\eps^*: \R^2 \rightarrow \R^{2 \times 2}$ by
\begin{equation} \label{e:Phi}
    \Phi_\eps := \id + \nabla \chi(\cdot/\eps)\;, \qquad \Phi_\eps^* := \id + \nabla \chi^*(\cdot/\eps)\;.
\end{equation}
We also write $\Phi = \id + \nabla \chi$ and $\Phi^* = \id + \nabla \chi^*$. Finallly, we write $\lL_\eps := \div (a_\eps \nabla)$ and $\lL_0 := \div (\bar{a} \nabla)$. 

\begin{flushleft}
    \textbf{Green functions}
\end{flushleft}
Let $Q_\eps$ and $Q_0$ denote the parabolic Green's functions associated to $\d_t - \lL_\eps$ and $\d_t - \lL_0$ respectively, in the sense that
\begin{equation*}
    \big((\d_t - \lL_\eps)^{-1} f \big)(t,x) = \int_{0}^{t} \int_{\T^2} Q_\eps (t-r,x,y) f(r,y) {\rm d}y {\rm d}r\;.
\end{equation*}
We write $\iI_\eps = (\d_t - \lL_\eps)^{-1}$ and $\iI_0 = (\d_t - \lL_0)^{-1}$. Let $\nabla_x Q_\eps$ and $\nabla_y Q_\eps$ denote the gradient of $Q_\eps$ with respect to its first and second spatial variables respectively, and $\nabla \iI_\eps$ and $\nabla \iI_0$ denote convolutions with $\nabla_x Q_\eps$ and $\nabla_x Q_0$.  

Recall $\Phi_\eps$ and $\Phi_\eps^*$ from \eqref{e:Phi}. Define the vector valued kernels $R_\eps$ and $R_\eps^*$ by
\begin{equation} \label{e:difference_first_order}
    \begin{split}
    R_\eps (t,x,y) &:= (\nabla_x Q_\eps)(t,x,y) - \Phi_\eps (x) \cdot (\nabla_x Q_0)(t,x,y)\;,\\
    R_\eps^* (t,x,y) &:= (\nabla_y Q_\eps)(t,x,y) - \Phi_\eps^* (y) \cdot (\nabla_y Q_0)(t,x,y)\;.
    \end{split}
\end{equation}
Let $\rR_\eps$ be operation given by
\begin{equation*}
    (\rR_\eps f)(t,x) := \int_{0}^{t} \int_{\T^2} R_\eps (t-r,x,y) f(r,y) {\rm d}y {\rm d}r\;,
\end{equation*}
and $\rR_\eps^*$ be that given by the kernel $R_\eps^*$ in the same way. Let $\widetilde{R}_\eps$ be the $2 \times 2$ matrix valued kernel given by
\begin{equation} \label{e:difference_second_order}
    \widetilde{R}_\eps (t,x,y) := (\nabla_{x,y}^{2} Q_\eps)(t,x,y) - \Phi_\eps (x) \, (\nabla_{x,y}^2 Q_0)(t,x,y) \, \big( \Phi_\eps^* (y) \big)^{\sT}\;,
\end{equation}
and write
\begin{equation*}
    (\widetilde{\rR}_\eps f)(t,x) := \int_{0}^{t} \int_{\T^2} \widetilde{R}_\eps (t-r,x,y) f(r,y) \, {\rm d}y \, {\rm d}r
\end{equation*}
for the convolution of $f$ with $\widetilde{R}_\eps$. Here, $\nabla_{x,y}^{2} = \nabla_x \nabla_y^\sT$, where $\sT$ denotes the transpose. Note that $Q_0(t,x,y) = Q_0 (t,x-y)$, hence we have
\begin{equation*}
    \begin{split}
    &(\nabla_x Q_0)(t,x,y) = (\nabla Q_0)(t,x-y)\;, \qquad (\nabla_y Q_0)(t,x,y) = - (\nabla Q_0)(t,x-y)\;,\\
    &(\nabla_{x,y}^2 Q_0)(t,x,y) = - (\nabla^2 Q_0)(t,x-y)\;,
    \end{split}
\end{equation*}
where $\nabla$ and $\nabla^2$ denote the gradient and Hessian of $Q_0$, considered as a function of a single spatial variable. We still distinguish the roles of $x$ and $y$ in $Q_0$ as in \eqref{e:difference_first_order} and \eqref{e:difference_second_order} since this expression (up to a negative sign) is consistent with the definition of the kernels $R_\eps$, $R_\eps^*$ and $\widetilde{R}_\eps$.

\begin{flushleft}
    \textbf{Quantities from the equation}
\end{flushleft}
For $\eps \in \N^{-1}$, define $X_\eps := - \lL_\eps^{-1} \Pi_0^\perp \xi$. For integer $\m \in \N$, define $\Lambda_\eps = \Lambda_{\eps,\m}$ by
\begin{equation} \label{e:functional_Lambda}
    \Lambda_\eps (u) := \nabla \big( g(u) \big) \prec \P_\m^\perp X_\eps - g(u) \succeq \nabla \P_\m^\perp X_\eps - g(u)\nabla\P_{\m}X_{\eps}\;.
\end{equation}
Since $\m$ is fixed, we simply write $\Lambda_{\eps}$ and drop the dependence of $\m$ in notation. For $\eps=0$, $\Lambda_0$ is the same as above except one replaces $X_\eps$ by $X_0$. These two quantities will appear throughout the article. Now, define stochastic objects
\begin{equation*}
    \fF_\eps := a_\eps \nabla X_\eps\;, \quad \fF_0 := \bar{a} \nabla X_0\;,
\end{equation*}
for associated fluxes. They are a priori well defined for every $\eps>0$ since $a$ is H\"older continuous. The uniformity and convergence in $\eps$ will be justified in Section~\ref{sec:stochastic}.

\subsection*{Organisation of the article}

The rest of the article is organised as follows. In Section~\ref{sec:overview}, we give an overview of the strategy, including a heuristic derivation of the ansatz. In Section~\ref{sec:solution_convergence}, we list the collection of enhanced stochastic objects and set up a fixed point problem based on the ansatz. We then prove uniform boundedness of the solution as well as the convergence to its homogenised limit, assuming convergence of the stochastic objects. The convergence of the flux is proved in Section~\ref{sec:flux_convergence}. In Section~\ref{sec:stochastic}, we give detailed analysis on the stochastic objects and the convergences to their homogenised limits. In the appendices, we provide some background and the preliminary results that are needed in the main text. 

\subsection*{Acknowledgements}

We thank Jun Geng and Jinping Zhuge for helpful discussions. We also thank two reviewers for their careful reading of the first version of the article, and providing many sharp and helpful comments. 

B. Fehrman is supported by NSF DMS-Probability Grant 2348650, the Simons Foundation Grant MPS-TSM-00007753, and the Louisiana BoR RCS Grant 20130014386. W. Xu was supported by Ministry of Science and Technology via the National Key R\&D Program of China (no.2023YFA1010100) and National Science Foundation China via Standard Project Grant (no.12171008) and Key Project Grant (no.12595280, 12595281). Part of the work was done when the first and third authors were visiting NYU Shanghai in Fall 2023. We thank its Institute of Mathematical Sciences for hospitality.

\section{Overview of the strategy}
\label{sec:overview}

Before we give heuristic derivation of our ansatz for \eqref{e:gPAM_formal}, let us first recall the para-controlled ansatz to the constant coefficient two dimensional gPAM model, formally written as
\begin{equation} \label{e:gPAM_constant_formal}
    \d_t u_0 = \lL_0 u_0 + g(u_0) \big( \xi - \infty \, g'(u_0) \big)\;,
\end{equation}
where $\lL_0 = \div (\bar{a} \nabla)$ with a constant positive definite $2 \times 2$ matrix $\bar{a}$. Since the regularity of $\xi$ is below $-1$, one expects $u$, and hence $g(u_0)$, to have regularity at most $1-$. This suggests the product $g(u_0) \cdot \xi$ is problematic. 

We explain how this problem is tackled in \cite{para_control_SPDE}. Let us neglect the initial data as well as the renormalisation term $-\infty \, g'(u_0)$ for the moment, and perform all the operations at a formal level. The mild form of the equation is
\begin{equation} \label{e:mild_standard_PAM}
    \begin{split}
    u_0 &= \iI_0 \big( g(u_0) \prec \xi \big) + \iI_0 \big( g(u_0) \circ \xi \big) + \iI_0 \big( g(u_0) \succ \xi \big)\\
    &= g(u_0) \prec X_0 + \underbrace{\iI_0 \big( g(u_0) \succeq \xi \big) + [\iI_0, \prec] \big( g(u_0), \xi \big) + g(u_0) \prec \big( \iI_0(\xi) - X_0 \big) }_{u_0^{\#}}\;,
    \end{split}
\end{equation}
where $[\iI_0, \prec](A,B) := \iI_0 (A \prec B) - A \prec \iI_0(B)$ is the commutator between $\iI_0$ and $\prec$. One hopes that the rough part of the solution $u$ is given by the term $g(u_0) \prec X_0$, and that the remainder $u_0^{\#}$ has $\cC^{1+}$ regularity to close the fixed point argument. This suggests the ansatz
\begin{equation} \label{e:ansatz_standard_PAM}
    u_0 = g(u_0) \prec X_0 + \underbrace{u_0^\#}_{\cC^{1+}}\;.
\end{equation}
To check it is indeed a right ansatz for \eqref{e:gPAM_constant_formal}, we first note the terms $\iI_0 (g(u_0) \succ \xi)$ and $[\iI_0, \prec](g(u_0), \xi)$ do have the desired regularity a priori (in fact almost $\cC^{2-}$). For the resonance product, a para-linearisation lemma, the ansatz \eqref{e:ansatz_standard_PAM} together with commutator estimates reduce it to
\begin{equation} \label{e:standard_gPAM_resonance}
    g(u_0) \circ \xi - \infty \, g'(u_0) g(u_0) = g'(u_0) g(u_0) \big( X_0 \circ \xi - \infty \big) + \cC^{0+}\;. 
\end{equation}
Here, the remainder being in $\cC^{0+}$ can be seen from commutator estimates. The quantity $X_0 \circ \xi - \infty$ is explicit and can be made sense by hand in $\cC^{0-}$. The action of $\iI_0$ then brings \eqref{e:standard_gPAM_resonance} back to $\cC^{2-}$ and closes the loop. A rigorous treatment of the above heuristics justifies \eqref{e:gPAM_constant_formal} as the $\delta \rightarrow 0$ limit of the regularised model \eqref{e:gPAM_constant}. 

Now let us turn to the homogenisation problem \eqref{e:gPAM_formal}. Similarly, we write
\begin{equation} \label{e:ansatz_first}
    u_\eps = g(u_\eps) \prec \P_\m^\perp X_\eps + u_\eps^{\#}\;,
\end{equation}
where $X_\eps = -\lL_\eps^{-1} \Pi_0^\perp \xi$ is the stationary linear part, and $\P_\m^\perp$ removes the Littlewood-Paley blocks up to $\m-2$. Here, $\m \in \N$ is a fixed large integer to be chosen later (random but independent of $\eps$). The introduction of $\P_\m^\perp$ is to obtain smallness for contraction in the fixed point map later (see also \cite{2d_singular_ns,2d_gpam_global_szz} for similar tricks). This does not affect regularities; the solution obtained in the end also does not depend on the choice of $\m$. At this moment, \eqref{e:ansatz_first} can be viewed as the definition of $u_\eps^{\#}$. But unlike the $\eps=0$ case, both terms in \eqref{e:ansatz_first} cause problems for defining $g(u_\eps) \circ \xi$ uniformly in $\eps$. 

First, unlike $u^{\#} \in \cC^{1+}$ in \eqref{e:ansatz_standard_PAM}, it is \textit{not true} that the remainder $u_\eps^{\#} \in \cC^{1+}$ uniformly in $\eps$. In fact, by replacing $\iI_0$ with $\iI_\eps$ in \eqref{e:mild_standard_PAM}, one can see that $\iI_\eps \big( g(u_\eps) \succeq \xi \big)$ is at most Lipschitz. Furthermore, the usual commutator estimates for $[\iI_0, \prec]$ also fails for $[\iI_\eps, \prec]$ uniformly in $\eps$, and the corresponding commutator term $[\iI_\eps, \prec] \big( g(u_\eps), \xi \big)$ is at most Lipschitz but not better. Hence, $u_\eps^{\#}$ is at most $W^{1,\infty}$ uniformly in $\eps$, preventing the product $u_\eps^{\#} \circ \xi$ being a priori well defined uniformly in $\eps$. Hence, we need to explore finer structures of $u_\eps^{\#}$. 

For the semi-explicit rough term $g(u_\eps) \prec \P_\m^\perp X_\eps$, one can again use commutator estimates to reduce its resonance product with $\xi$ to making sense of the object $\P_\m^\perp X_\eps \circ \xi$. But now $X_\eps$ is obtained from the oscillatory variable coefficient operator $\lL_\eps$, whose Green's function is not conveniently compatible with the Fourier operation $\circ$. Hence, it is not easy to compute $\P_\m^\perp X_\eps \circ {\xi}$ directly. 

We now briefly explain our strategy -- how one uses integration by parts to identify suitable structure of $u_\eps^{\#}$, and how an additional ``completing the product" trick enables one to do explicit computations with pure stochastic objects. Back to the right hand side of \eqref{e:gPAM_formal}, recalling $\xi = - \lL_\eps X_\eps + \Pi_0 \xi$ and using integration by parts, we have
\begin{equation} \label{e:gPAM_formal_RHS_0}
    \begin{split}
    g(u_\eps) \, \xi &= \Pi_0 \xi \cdot g(u_\eps) - g(u_\eps) \cdot \div \big( a_\eps \nabla X_\eps \big)\\
    &= \Pi_0 \xi \cdot g(u_\eps) - \div \big( g(u_\eps) \,\fF_\eps \big) + \nabla^{\sT} \big( g(u_\eps) \big) \cdot \fF_\eps\;.
    \end{split}
\end{equation}
The term $\fF_\eps = a_\eps \nabla X_\eps$ is purely stochastic, and can be made sense by hand in $\cC^{0-}$. Using \eqref{e:ansatz_first}, we have
\begin{equation} \label{e:gradient_gu_expression}
    \nabla \big( g(u_\eps) \big) = g'(u_\eps) \, \nabla u_\eps = g'(u_\eps) \big( \Lambda_\eps (u_\eps) + g(u_\eps) \nabla  X_\eps + \nabla u_\eps^\# \big)\;,
\end{equation}
where we recall from \eqref{e:functional_Lambda} that
\begin{equation*}
    \Lambda_\eps (u) = \nabla \big( g(u) \big) \prec \P_\m^\perp X_\eps - g(u) \succeq \nabla \P_\m^\perp X_\eps - g(u)\nabla\P_{\m}X_{\eps}\;,
\end{equation*}
and we completed the paraproduct $g(u_\eps) \prec \nabla \P_\m^\perp X_\eps$ into the real product $g(u_\eps) \nabla X_\eps$. Combining \eqref{e:gPAM_formal_RHS_0} and \eqref{e:gradient_gu_expression}, we see the right hand side of \eqref{e:gPAM_formal} can be expressed as
\begin{equation} \label{e:gPAM_formal_RHS}
    \begin{split}
    g&(u_\eps) \big( \xi - \infty_\eps \, g'(u_\eps) \big) = \Pi_0 \xi \cdot g(u_\eps) - \div \big( g(u_\eps) \cdot \fF_\eps \big)\\
    &+ g'(u_\eps) \Big[ \big( \Lambda_\eps^\sT (u_\eps) + \nabla^\sT u_\eps^{\#} \big) \fF_\eps +g(u_\eps) \big( \nabla^\sT X_\eps \diamond \fF_\eps \big) \Big]\;,
    \end{split}
\end{equation}
where $\diamond$ denotes the Wick product between two Gaussians. Since one expects $\nabla X_\eps \in \cC^{0-}$ and $u_\eps \in \cC^{1-}$, $\Lambda_\eps (u_\eps)$ should inherit almost the same regularity from $u_\eps$, and hence its product with $\fF_\eps$ is well-defined. The Wick product $\nabla^\sT X_\eps \diamond \fF_\eps$ can now be made sense in $\cC^{0-}$ explicitly by hand, and is sufficient to multiply with $g(u_\eps)$. It remains to understand the product between $\nabla u_\eps^\#$ and $\fF_\eps$, which requires exploration of finer structures of $\nabla u_\eps^{\#}$. 

To proceed, applying $\d_t - \lL_\eps$ to both sides of \eqref{e:ansatz_first} and noting that $u_\eps$ satisfies \eqref{e:gPAM_formal}, we get
\begin{equation*}
    (\d_t - \lL_\eps) u_\eps^{\#} = g(u_\eps) \big( \xi - \infty_\eps \cdot g'(u_\eps) \big) - (\d_t - \lL_\eps) \big( g(u_\eps) \prec \P_{\m}^{\perp} X_\eps \big)\;.
\end{equation*}
The second term on the right hand side above can be computed explicitly as
\begin{equation} \label{e:ansatz_rough_differentiate}
    \begin{split}
    (\d_t - \lL_\eps) \big( g(u_\eps) \prec \P_{\m}^{\perp} X_\eps \big) = &\phantom{1}\d_t \big( g(u_\eps) \big)\prec \P_{\m}^{\perp}X_\eps\\
    &- \div \big( a_\eps \Lambda_\eps (u_\eps) + g(u_\eps) \, \fF_\eps \big)\;.
    \end{split}
\end{equation}
Combining \eqref{e:gPAM_formal_RHS} and \eqref{e:ansatz_rough_differentiate}, we get
\begin{equation} \label{e:equation_u_sharp}
    \begin{split}
    \d_t &u_\eps^{\#} = \lL_\eps u_\eps^{\#} + \div \big( a_\eps \Lambda_\eps (u_\eps) \big) - \underbrace{ \d_t \big( g(u_\eps) \big) \prec \P_{\m}^{\perp}X_\eps}_{\tT_\eps}\\
    &\underbrace{\Pi_{0}\xi\cdot g(u_{\eps})+  g'(u_\eps) \Big[ \big( \Lambda_\eps^\sT (u_\eps) + \nabla^\sT u_\eps^{\#} \big) \fF_\eps +g(u_\eps) \big( \nabla^\sT X_\eps \diamond \fF_\eps \big) \Big] }_{\sS_\eps}\;.
    \end{split}
\end{equation}
Assuming for the moment that $\tT_\eps$ and $\sS_\eps$ on the right hand side above belong to $\cC^{0-}$ uniformly in $\eps$, then from the asymptotic expansion of the Green's function $Q_\eps$, one expects
\begin{equation} \label{e:ansatz_gradient_u_sharp_pre}
    \begin{split}
    \nabla u_\eps^{\#} = &\nabla e^{t \lL_\eps} u_\eps^{\#}(0) + \underbrace{\Phi_\eps}_{L^\infty} \underbrace{\nabla \iI_0 ( \sS_\eps - \tT_\eps)}_{\cC^{1-}} + \underbrace{\rR_\eps (\sS_\eps - \tT_\eps)}_{\text{small} \, \cC^{0+}}\\
    &+ \nabla \iI_\eps \Big( \div \big(a_\eps \Lambda_\eps(u_\eps) \big) \Big)\;,
    \end{split}
\end{equation}
where we recall the definition of $\Phi_\eps$ from \eqref{e:Phi}, and that $\rR_\eps$ denotes the spacetime convolution by the kernel $R_\eps$ given in \eqref{e:difference_first_order}. 

Recall that our goal is to understand the product between $\nabla u_\eps^{\#}$ and $\fF_\eps$. The third term $\rR_\eps (\sS_\eps - \tT_\eps)$ on the right hand side above is regular enough to multiply with $\fF_\eps \in \cC^{0-}$. For the second term, $\nabla \iI_0 (\sS_\eps-\tT_\eps)$ is sufficiently regular and the rough part is the explicit deterministic function $\Phi_\eps$, so its product with $\fF_{\eps}$ can also be made by hand. The initial data term can be written as
\begin{equation*}
    \nabla e^{t \lL_\eps} u_\eps^{\#}(0) = \underbrace{\Phi_\eps}_{L^\infty} \underbrace{\nabla e^{t \lL_0} u_\eps^{\#}(0)}_{\cC^{1-}} + \underbrace{\big( \nabla e^{t \lL_\eps} - \Phi_\eps \nabla e^{t \lL_0} \big) u_\eps^{\#}(0)}_{\text{small} \, \cC^{0+}}\;,
\end{equation*}
which is also of the above structure. 

The subtle term is $\nabla \iI_\eps \big( \div (a_\eps \Lambda_\eps(u_\eps) ) \big)$. It has neither sufficient regularity (even when $\nabla \iI_\eps$ is replaced by $\nabla \iI_0$) nor explicit structures ($u_\eps$ is correlated with $\nabla X_\eps$ in a complicated and inexplicit way). A closer look at this term reveals that its roughness is due to $a_\eps$, while $\Lambda_\eps (u_\eps)$ is sufficiently regular. This motivates us to re-write $\nabla \iI_\eps \big( \div (a_\eps \Lambda_\eps(u_\eps) ) \big)$ as
\begin{equation} \label{e:split_subtle_term}
    \underbrace{\Big( \nabla \iI_\eps \big( \div (a_\eps \Lambda_\eps (u_\eps)) \big) - \nabla \iI_\eps (\div a_\eps) \cdot \Lambda_\eps (u_\eps) \Big)}_{\Phi_\eps \cdot \, \cC^{1-} \, + \; \text{small} \; \cC^{0+}} + \underbrace{\nabla \iI_\eps (\div a_\eps)}_{L_{T,x}^{\infty}} \cdot \underbrace{\Lambda_\eps (u_\eps)}_{\cC^{1-}}\;.
\end{equation}
The first term above can be further decomposed into the sum $\Phi_\eps \tilde{v}_\eps + \tilde{w}_\eps$ for some $\tilde{v}_\eps \in \cC^{1-}$ and $\tilde{w}_\eps$ with a small-in-$\eps$ positive H\"older norm. $\tilde{w}_\eps$ has sufficient regularity to multiply with $\fF_\eps$. The rough part of the first component is the explicit deterministic function $\Phi_\eps$, and its product with $\fF_\eps$ can be handled by hand. The same is true for the second term above, whose rough part $\nabla \iI_\eps (\div a_\eps)$ is also explicit. This allows us to close the loop and solve the fixed point equation uniformly in $\eps$. The convergence to the right homogenised limit requires further exploration on subtle cancellations from the oscillations and the divergence-free structure of $a_\eps \Phi_\eps$. 

The above discussions suggest that $\nabla u_\eps^{\#}$ should have the structure
\begin{equation} \label{e:ansatz_gradient_u_sharp}
    \nabla u_\eps^{\#} = \Phi_\eps \underbrace{v_\eps}_{\cC^{1-}} + \underbrace{w_\eps}_{\text{small} \; \cC^{0+}} + \underbrace{\nabla \iI_\eps (\div a_\eps)}_{\text{deterministic} \, L_{T,x}^{\infty}} \Lambda_\eps (u_\eps)\;,
\end{equation}
where one part of $v_\eps$ and $w_\eps$ come from applications of $\nabla e^{t \lL_0}$ and $\nabla e^{t \lL_\eps} - \Phi_\eps\nabla e^{t \lL_0}$ to the initial data $u_\eps^{\#}(0)$, one part from convolution of $\nabla Q_0$ and $R_\eps$ with a $\cC^{0-}$ function, and the final parts coming from the difference $\nabla \iI_\eps \big( \div (a_\eps \Lambda_\eps (u_\eps)) \big) - \nabla \iI_\eps (\div a_\eps) \cdot \Lambda_\eps (u_\eps)$. In contrast to $u_0^\# \in \cC^{1+}$ in the constant coefficient case, \eqref{e:ansatz_gradient_u_sharp} will be the core part of our ansatz for the homogenisation problem \eqref{e:gPAM_formal}. We will set up a fixed point problem in Section~\ref{sec:solution_convergence} based on slight modifications of \eqref{e:ansatz_gradient_u_sharp}, and prove uniform boundedness of the solutions $u_\eps$ and convergence to the homogenised limit $u_0$.

\begin{rmk}
The above ansatz circumvents the use of the commutator estimates for, with an abuse of notation of $g$,
\begin{equation*}
    \begin{split}
    [\iI_\eps, \prec] (f,g) &:= \iI_\eps (f \prec g) - f \prec \iI_\eps (g)\;,\\
    \text{Com}(f,g,h) &:= (f \prec g) \circ h - f \cdot (g \circ h)\;.
    \end{split}
\end{equation*}
While the usual estimates for the second commutator above is still true (this is purely paraproduct operation and no oscillatory operator is involved), the necessary estimates needed for the first commutator $[\iI_0, \prec]$ fails to hold for $[\iI_\eps, \prec]$ uniformly in $\eps$. 

The integration by parts trick completely circumvents the use of both commutators. Even in the case of constant coefficient case ($\eps=0$), this gives another construction of $2$D gPAM without using commutators. See also Remark~\ref{rmk:commutator_confirm}. 
\end{rmk}

The convergence of the flux $a_\eps \nabla u_\eps$ to $\bar{a} \nabla u_0$ carries another interesting subtlety, namely that two of the terms consisting of the flux converge to the ``wrong" limits individually, but their sum converges to the right limit. 

To see this, according to \eqref{e:ansatz_first} and \eqref{e:ansatz_gradient_u_sharp}, we have
\begin{equation*}
    \nabla u_\eps =  \nabla \big( g(u_\eps) \prec \P_{\m}^{\perp}X_\eps \big)  + \nabla \iI_\eps (\div a_\eps) \cdot\, \Lambda_\eps(u_\eps) + \Phi_\eps v_\eps + w_\eps\;.
\end{equation*}
The first two terms on the right hand side above converge to $\nabla \big( g(u_0) \prec \P_{\m}^{\perp}X_0 \big)$ and $0$ respectively. But when multiplied with $a_\eps$, it turns out that
\begin{equation*}
    \begin{split}
    a_\eps \, \nabla \big( g(u_\eps) \prec  \P_{\m}^{\perp}X_\eps \big) &\rightarrow \bar{a} \, \nabla \big( g(u_0) \prec  \P_{\m}^{\perp}X_0 \big) + (\Pi_{0}a -\bar{a}) \Lambda_0 (u_0)\\
    &\neq \bar{a} \, \nabla \big( g(u_0) \prec \P_{\m}^{\perp}X_0 \big)\;,
    \end{split}
\end{equation*}
where we recall $\Pi_{0} a$ is the average of $a$ on $\T^2$. The interesting point is that $\nabla \iI_\eps ( \div a_\eps ) \rightarrow 0$ weakly in $L^\infty$, and we actually have
\begin{equation*}
    a_\eps \nabla \iI_\eps (\div a_\eps) \rightarrow\bar{a} - \Pi_0 a
\end{equation*}
as $\eps \rightarrow 0$, thus compensating the twisted limit of the first term. They together give the right limit of the flux. We will carry out a detailed analysis in Section~\ref{sec:flux_convergence}.

\begin{rmk}
    In the linear case $g(u) = u$, the transform in \cite{Hairer-Labbe} can significantly simplify the problem, though it is still not as straightforward as the constant coefficient case. To see this, we note that $\widetilde{u}_\eps := u_\eps \, e^{-X_\eps}$ satisfies the equation
    \begin{equation*}
        \d_t \widetilde{u}_\eps = \lL_\eps \widetilde{u}_\eps + \nabla^{\sT} X_\eps \big( a_\eps + a_\eps^\sT \big) \nabla \widetilde{u}_\eps + \big( \nabla^\sT X_\eps \cdot a_\eps \nabla X_\eps - \infty_\eps  \big) \, \widetilde{u}_\eps\;.
    \end{equation*}
    Even though both $\nabla X_\eps$ and $\nabla^{\sT} X_\eps \cdot a_\eps \nabla X_\eps - \infty_\eps$ are well defined (with a proper choice of $\infty_\eps$) and converge in $L_T^\infty \cC_x^{-\kappa}$ for every $\kappa>0$, $\widetilde{u}_\eps$ is not uniformly in $\cC^{1+}$, and the product between $\nabla X_\eps$ and $\nabla \widetilde{u}_\eps$ is not a priori well-defined. In this case, 
    \begin{equation*}
        \nabla \widetilde{u}_\eps = \Phi_\eps \underbrace{\widetilde{v}_\eps}_{\cC^{1-}} + \underbrace{\widetilde{w}_\eps}_{\text{small} \, \cC^{0+}}
    \end{equation*}
    is a suitable ansatz, and one still needs to define the additional processes $\Phi_\eps^\sT a_\eps \nabla X_\eps$ and $\Phi_\eps^\sT a^{\sT}_\eps \nabla X_\eps$ by hand. 
\end{rmk}

\section{Uniform boundedness and convergence of the solution}
\label{sec:solution_convergence}

In this section, we first introduce the collection of enhanced stochastic objects, and then set up a fixed point problem that is suitable to get uniform-in-$\eps$ bounds and convergence of the solution $u_\eps$ to \eqref{e:gPAM_formal}. We prove these facts in Theorem~\ref{thm:fixed_pt_thm}, assuming the convergence of the stochastic objects. 

As indicated in Section~\ref{sec:overview} (see in particular \eqref{e:ansatz_first} and \eqref{e:ansatz_gradient_u_sharp}), the fixed point problem will be a system for a triple $(u_\eps, v_\eps, w_\eps)$. Throughout this section, we fix $\alpha \in (0, 1)$ and $\eta \in \big( 1-\frac{3 \alpha}{4}, 1 - \frac{\alpha}{2} \big)$, and also fix $\kappa>0$ arbitrarily small. For $T>0$, write
\begin{equation}\label{e:space_X}
    \xX_T^1 = \cC_\fs^{\alpha}([0,T] \times \T^2)\;, \quad \xX_T^2 = L_T^{\infty, \eta} \cC_x^{\alpha/4}\;, \quad \xX_T^3 = L_T^{\infty, \eta} \cC_x^{8\kappa}\;,
\end{equation}
and
\begin{equation*}
    \xX_T := \xX_T^1 \times \xX_T^2 \times \xX_T^3\;.
\end{equation*}
We will seek solutions to the system (see \eqref{e:fixed_pt_system} below) in the space $\xX_T$. To simplify notations, we write
\begin{equation*}
    \vec{u}_\eps = (u_\eps, v_\eps, w_\eps)\;, \qquad \vec{u}_0 = (u_0, v_0, 0)\;,
\end{equation*}
and write $\vec{u} = (u,v,w)$ for a generic element in $\xX_T$.

\subsection{Setting up the fixed point equations}
\label{sec:fixed_pt_problem_setup}

We are now ready to set up the fixed point equations. Recall the functional $\Lambda_\eps$ introduced in \eqref{e:functional_Lambda}. Also, $\nabla u_\eps^{\#}$ appears only from part of $\nabla u_\eps$. The latter can be re-written as
\begin{equation*}
    \nabla u_\eps = \nabla \big( g(u_\eps) \prec \P_{\m}^{\perp}X_\eps \big) + \nabla u_\eps^\# = g(u_\eps) \nabla X_\eps + \nabla u_\eps^\# + \Lambda_\eps (u_\eps)\;.
\end{equation*}
Hence, it is natural to define a functional that ``represents" $\nabla u_\eps^{\#}+\Lambda_\eps(u_\eps)$, which we denote by $\aA_\eps$. Precisely, returning to \eqref{e:ansatz_gradient_u_sharp} and adding $\Lambda_\eps (u_\eps)$ to it, we write
\begin{equation} \label{e:functional_A}
    \aA_\eps (\vec{u}) = \aA_\eps(u,v,w) := \Phi_\eps v + w + \big(\id + \nabla \iI_\eps (\div a_\eps) \big) \, \Lambda_\eps (u)\;.
\end{equation}
For $\eps = 0$, define
\begin{equation*}
    \aA_0 (\vec{u}) = \aA_0 (u,v,w) = \aA_0 (u,v) := v + \Lambda_0( u )\;.
\end{equation*}
According to \eqref{e:equation_u_sharp} and \eqref{e:functional_A} (recalling that $\aA_\eps$ ``represents" $\Lambda_\eps (u_\eps) + \nabla u_\eps^{\#}$), we define the functional $\sS_\eps$ by
\begin{equation}\label{e:functional_S_eps}
    \begin{split}
    \sS_\eps (\vec{u}) := \Pi_0 \xi \cdot g(u) + g'(u) \, \aA_\eps^\sT(\vec{u})  \fF_\eps+ g'(u) g(u) \cdot  \nabla^\sT X_\eps \diamond \fF_\eps\;,
    \end{split}
\end{equation}
where $\nabla^\sT X_\eps \diamond \fF_\eps$ is the Wick product between the two Gaussians $\nabla^\sT X_\eps$ and $\fF_\eps = a_\eps \nabla X_\eps$. $\sS_0$ is defined in the same way except replacing $\aA_\eps^{\sT} \fF_\eps$ by $\aA_0^{\sT} \fF_0$ and $\nabla^\sT X_\eps \diamond \fF_\eps$ by $\nabla^\sT X_0 \diamond \fF_0$. 

For $\eps \in \N^{-1}\cup \{0\}$, define the functional $\tT_\eps$ by
\begin{equation} \label{e:functional_T_eps}
    \tT_\eps (\vec{u}) := \widetilde{\tT}_\eps (u,v,w) \prec \P_\m^\perp X_\eps\;,
\end{equation}
where
\begin{equation} \label{e:functional_T_tilde_eps}
    \widetilde{\tT}_\eps (\vec{u}) =  \div \big( g'(u) a_\eps \aA_\eps \big) + g'(u) \sS_\eps - g''(u) \big( \aA_\eps^\sT a_\eps \aA_\eps + g(u) \, \aA_\eps^{\sT} a_\eps^{\sT} \nabla X_\eps \big)\;,
\end{equation}
and $\tT_0$ is defined in the same way except that one replaces all appearances of $X_\eps$, $\aA_\eps$ and $\sS_\eps$ by $X_0$, $\aA_0$ and $\sS_0$, and replaces $a_\eps$ by $\bar{a}$. Here, we also omitted the arguments $\vec{u}=(u,v,w)$ in the functionals $\aA_\eps$ and $\sS_\eps$. The reason to have such an expression is that $\widetilde{\tT}_\eps$ should on one hand be a faithful representation of $\d_t \big( g(u_\eps) \big)$, and on the other hand be an object in $\cC_x^{0-}$ for fixed times. This is not true in general but is true if $u_\eps$ is a ``solution" to \eqref{e:gPAM_formal}. A detailed derivation of \eqref{e:functional_T_tilde_eps} is given in Section~\ref{sec:T}. 

We now turn to writing down the precise fixed point system. Comparing \eqref{e:ansatz_gradient_u_sharp_pre} and \eqref{e:ansatz_gradient_u_sharp} suggests us to further explore the structure of the first term in \eqref{e:split_subtle_term}. Integrating by parts, we get
\begin{equation*}
    \begin{split}
    &\phantom{111} \Big( \nabla \iI_\eps \big( \div (a_\eps F_\eps) \big) \Big)(t,x) - \big( \nabla \iI_\eps (\div a_\eps) \big)(t,x) \cdot F_\eps(t,x)\\
    &= - \int_0^t \int_{\T^2} (\nabla_{x,y}^2 Q_\eps)(t-r,x,y) \,  a_\eps(y) \, \big( F_\eps(r,y) - F_\eps(t,x) \big) {\rm d}y {\rm d}r\;,
    \end{split}
\end{equation*}
where we wrote $F_\eps = \Lambda_\eps (u_\eps)$ for simplicity. With the expansion for $\nabla_{x,y}^2 Q_\eps$ in \eqref{e:difference_second_order}, we can see the contributions to $v_\eps$ and $w_\eps$ in the ansatz \eqref{e:ansatz_gradient_u_sharp} from the term $\nabla \iI_\eps \big( \div (a_\eps F_\eps) \big) - \nabla \iI_\eps (\div a_\eps) F_\eps$. 

Hence, this computation together with \eqref{e:ansatz_first}, \eqref{e:equation_u_sharp}, \eqref{e:ansatz_gradient_u_sharp_pre} and \eqref{e:ansatz_gradient_u_sharp} suggest that we shall consider the system of equations
\begin{equation}\label{e:fixed_pt_system}
    \begin{split}
    u_\eps(t) &= e^{t \lL_\eps} u_\eps^{\#}(0) + g(u_\eps) \prec \P_{\m}^\perp X_\eps + \iI_\eps \Big( \div \big( a_\eps \Lambda_\eps(u_\eps) \big) + \sS_\eps - \tT_\eps \Big)\;,\\
    v_\eps(t) &= \nabla e^{t \lL_0} u_\eps^{\#}(0) + \nabla \iI_0 \big( \sS_\eps - \tT_\eps \big) - \int_{0}^{t} \int_{\T^2} (\nabla_{x,y}^2 Q_0)(t-r,x,y) \cdot\\
    &\phantom{1}\big(\Phi_\eps^*(y)\big)^{\sT} a_\eps(y) \Big( (\Lambda_\eps u_\eps)(r,y) - (\Lambda_\eps u_\eps)(t,x) \Big) {\rm d}y {\rm d}r\\
    w_\eps(t) &= \big(\nabla e^{t \lL_\eps} - \Phi_\eps \nabla e^{t \lL_0}\big) u_\eps^{\#}(0) + \rR_\eps \big( \sS_\eps - \tT_\eps \big)\\
    &\phantom{1}- \int_{0}^{t} \int_{\T^2} \widetilde{\rR}_\eps(t-r,x,y) a_\eps(y) \Big( (\Lambda_\eps u_\eps)(r,y) - (\Lambda_\eps u_\eps)(t,x) \Big) {\rm d}y {\rm d}r\;.
    \end{split}
\end{equation}
Here, $u_\eps^{\#}$ is \textit{defined} as
\begin{equation} \label{e:u_sharp}
    u_{\eps}^{\#} := u_{\eps} - g(u_{\eps}) \prec \P_{\m}^{\perp}X_{\eps}\;,
\end{equation}
and we omitted the inputs $\vec{u}_\eps = (u_\eps, v_\eps, w_\eps)$ in $\sS_\eps$ and $\tT_\eps$. Also, $e^{\lL_\eps}$ and $\nabla_x e^{t \lL_\eps}$ denote space integration with the kernel $Q_\eps(t,x,\cdot)$ and $\nabla Q_\eps(t,x,\cdot)$, and that $\widetilde{\rR}_\eps$ denotes integration with the kernel $\widetilde{R}_\eps$ defined in \eqref{e:difference_second_order}. One can check that if $\xi$ is smooth, then for every $\eps\in \N^{-1}$, the triple $(u_\eps, v_\eps, w_\eps)$ solves the fixed point system \eqref{e:fixed_pt_system} if and only if its first component $u_\eps$ solves the original equation \eqref{e:gPAM_formal}. The reason to set up the fixed point system \eqref{e:fixed_pt_system} is to ensure all the operations are valid and uniform-in-$\eps$ (in the case when $\xi$ is the spatial white noise), and to obtain the convergence as $\eps \rightarrow 0$. 

For $\eps = 0$, the fixed point system for $\vec{u}_0 = (u_0, v_0)$ is
\begin{equation} \label{e:fixed_pt_system_limit}
    \begin{split}
    u_0 &= e^{t \lL_0} u_0^{\#}(0) + g(u_0) \prec \P_{\m}^{\perp}X_0 + \iI_0 \Big( \div \big(\bar{a} \, \Lambda_0 (u_0) \big) + \sS_0 (\vec{u}_0) - \tT_0 (\vec{u}_0) \Big)\\
    v_0 &= \nabla e^{t\lL_0} u_0^{\#}(0) + \nabla \iI_0 \big( \sS_{0} (\vec{u}_0) - \tT_0 (\vec{u}_0) \big) + \nabla \iI_0 \Big( \div \big(\bar{a} \, \Lambda_0 (u_0) \big) \Big)\;,
    \end{split}
\end{equation}
where $\tT_0$ and $\sS_0$ are functionals on $\vec{u} = (u,v)$ whose expressions are the same with $\tT_\eps$ and $\sS_\eps$ by replacing all appearances of $\eps$ by $0$ and $a_\eps$ by $\bar{a}$. One can check directly that the system \eqref{e:fixed_pt_system_limit} admits a unique solution $(u_0, v_0)$, which is precisely the solution to the standard constant coefficient 2D gPAM. 

Recall that $\alpha \in (0, 1)$, $\eta \in ( 1-\frac{3\alpha}{4}, 1-\frac{\alpha}{2} )$, and $\kappa>0$ is arbitrarily small. Recall also the definition of $\xX_T = \xX_T^1 \times \xX_T^2 \times \xX_T^3$ from \eqref{e:space_X}. We will show that the system \eqref{e:fixed_pt_system} has a unique solution $\vec{u}_\eps = (u_\eps, v_\eps, w_\eps) \in \xX_T$ up to a random time $T>0$. Moreover, if the initial data converges in $\cC^\alpha$ as $\eps \rightarrow 0$, then the triple $\vec{u}_\eps$ converges to a limit $\vec{u}_0 = (u_0, v_0, 0)$, which solves the limiting homogenised equation \eqref{e:fixed_pt_system_limit}. 

The first problem is that the flux $\fF_\eps$ is uniformly bounded in $\cC^{-\kappa}$ for every $\kappa>0$ but not in $L^p$ for any $p$, while $\aA_\eps$ given in \eqref{e:functional_A} is only uniformly bounded in $L^\infty$ in space. Hence, the product $\aA_\eps^\sT \fF_\eps$ is not a priori well defined. On the other hand, using the expression \eqref{e:functional_A}, we have
\begin{equation} \label{e:expression_A_flux}
    \aA_\eps^{\sT}(\vec{u}) \fF_\eps = v^\sT \Phi_\eps^\sT \fF_\eps + w^\sT \fF_\eps + \Lambda_\eps^\sT (u) \cdot \big( \id + \nabla \iI_\eps (\div a_\eps) \big)^\sT \fF_\eps \;. 
\end{equation}
We will take $\fF_\eps$ and $\Phi_\eps^{\sT} \fF_\eps$ as pre-defined noises, and view the right hand side of \eqref{e:expression_A_flux} as the expression or definition of $\aA_\eps^{\sT}(\vec{u}) \fF_\eps$. The object $\big( \id + \nabla \iI_\eps (\div a_\eps) \big)^\sT \fF_\eps$ can be obtained from $\fF_\eps$ and $\Phi_\eps^{\sT} \fF_\eps$ (see Lemma~\ref{lem:extra_0} below). 

Similarly, the term $\aA_\eps^{\sT} a_\eps^\sT \nabla X_\eps$ from $\widetilde{\tT}_\eps$ has the expression
\begin{equation} \label{e:expression_Aa_ff}
    \begin{split}
    \aA_\eps^{\sT}(\vec{u}) a_\eps^{\sT} \nabla X_\eps = &\phantom{1}v^{\sT} \Phi_\eps^{\sT} a_\eps^{\sT} \nabla X_\eps + w_\eps^{\sT} a_\eps^{\sT} \nabla X_\eps\\
    &+ \Lambda_\eps^{\sT}(\vec{u}) \big( \id + \nabla \iI_\eps (\div a_\eps) \big)^\sT a_\eps^{\sT} \nabla X_\eps\;.
    \end{split}
\end{equation}
This time, we take $a_\eps^{\sT} \nabla X_\eps$ and $\Phi_\eps^\sT a_\eps^\sT \nabla X_\eps$ as pre-defined noises, and take the right hand side of \eqref{e:expression_Aa_ff} as the expression or definition of $\aA_\eps^{\sT}(\vec{u}) a_\eps^{\sT} \nabla X_\eps$. The term $\big( \id + \nabla \iI_\eps (\div a_\eps) \big)^\sT a_\eps^{\sT} \nabla X_\eps$ is obtained from $a_\eps^{\sT} \nabla X_\eps$ and $\Phi_\eps^\sT a_\eps^\sT \nabla X_\eps$. 

From the above discussions, it is natural to introduce the seven dimensional vectors of stochastic objects
\begin{equation} \label{e:noises_eps}
    \begin{split}
    \Upsilon_\eps &:= \big(\Pi_0 \xi, \; X_\eps,\; \fF_\eps,\; \Phi_\eps^\sT \fF_\eps,\; a_\eps^{\sT} \nabla X_\eps, \; \Phi_\eps^{\sT} a_\eps^{\sT} \nabla X_\eps, \; \nabla^\sT X_\eps \diamond \fF_\eps \big)\;,
    \end{split}
\end{equation}
and
\begin{equation} \label{e:noises_limit}
    \begin{split}
    \Upsilon_0 := \big(\Pi_0 \xi, \; X_0,\; \fF_0,\; \fF_0,\; \Pi_0 (a^{\sT} \Phi) \nabla X_0, \; \bar{a}^{\sT} \nabla X_0, \; \nabla^\sT X_0 \diamond \fF_0 \big)\;.
    \end{split}
\end{equation}
For $1 \leq j \leq 7$, we denote the corresponding components in order by $\Upsilon_\eps^{(j)}$ and $\Upsilon_0^{(j)}$. We then define the norm $\yY$ by
\begin{equation} \label{e:noises_norm}
    \|\Upsilon_\eps\|_{\yY} := |\Pi_0 \xi| + \|X_\eps\|_{\cC^{1-\kappa}} + \sum_{j = 3}^{7} \|\Upsilon_\eps^{(j)}\|_{\cC^{-\kappa}} \;.
\end{equation}
The well-posedness of these ``enhanced" noises and convergence to the right limits will be shown in Section~\ref{sec:stochastic}. 

To summarise, we list below the collection of the stochastic objects, their corresponding limits, and spaces of convergences.
\begin{center} 
\begin{tabular}{|c c c c c c c c|} 
 \hline
 Process: &  $\Pi_0 \xi$ & $X_\eps$ & $\fF_\eps$ & $\Phi_\eps^\sT \fF_\eps$  &  $a_\eps^{\sT} \nabla X_\eps$  &  $\Phi_\eps^{\sT} a_\eps^{\sT} \nabla X_\eps$  &  $\nabla^{\sT} X_\eps \diamond \fF_\eps$  \\ 
 \hline
 Limit: &  $\Pi_0 \xi$ & $X_0$ & $\fF_0$ & $\fF_0$  &  $\Pi_0 (a^{\sT} \Phi) \nabla X_0$  &  $\bar{a}^{\sT} \nabla X_0$ & $\nabla^{\sT} X_0 \diamond \fF_0$  \\ 
 \hline
 Space: &  $\R$ & $\cC^{1-\kappa}$ & $\cC^{-\kappa}$ & $\cC^{-\kappa}$ &  $\cC^{-\kappa}$  &  $\cC^{-\kappa}$ &  $\cC^{-\kappa}$ \\
 \hline
\end{tabular}
\end{center}

\begin{rmk}
    The standard gPAM with a Laplacian (that is, $\bar{a} = \id$) has two (enhanced) stochastic objects: $\nabla X_0$ and $\nabla^\sT X_0 \diamond \fF_0$  (\cite{para_control_SPDE, rs_theory, Hairer-Labbe}). At first glance, there seem to be many more stochastic objects in our setting. Even in the constant-coefficient case $\eps = 0$, there are three more: $\fF_0$, $\bar{a}^{\sT}\nabla X_0$ and $\Pi_0(a^{\sT}\Phi)\nabla X_0$. Actually, we note that since $\bar{a}$, $\bar{a}^\sT$ and $\Pi_0(a^{\sT}\Phi)$ are constant matrices, all these three objects are well-defined and completely determined by $\nabla X_0$. Hence, the effective ones are just the two in the standard situation.
    
    For $\eps \in (0,1)$, the linear flux $\fF_\eps = a_\eps \nabla X_\eps$ is necessary, as well as its variants $\Phi_\eps^\sT \fF_\eps$, $a_\eps^{\sT} \nabla X_\eps$ and $\Phi_\eps^{\sT} a_\eps^{\sT} \nabla X_\eps$, even though their limits are a priori well defined. This is a reflection that approximations in homogenisation problems typically have worse behaviours than their limits. 
\end{rmk}

\begin{rmk}
The limit of $a_\eps^{\sT} \nabla X_\eps$ is $\Pi_0 (a^{\sT} \Phi) \nabla X_0$, where
\begin{equation*}
    \Pi_0 (a^{\sT} \Phi) = \int_{\T^2} a^{\sT}(y) \Phi(y) {\rm d}y = \int_{\T^2} a^{\sT}(y) \big( \id + \nabla \chi(y) \big) {\rm d}y\;.
\end{equation*}
It may appear strange at first since $\Pi_0(a^{\sT} \Phi)$ is neither $\bar{a}$ nor $\bar{a}^{\sT}$, and one may wonder what is its role in the homogenised limit. What really happens is that all appearances of the object $a_\eps^{\sT} \nabla X_\eps$ in the fixed point system always come with multiplication by a function with a positive H\"older norm vanishing as $\eps \rightarrow 0$. Thus, even though the object $a_\eps^{\sT} \nabla X_\eps$ is needed in the solution for fixed $\eps>0$, its limit $\Pi_0 (a^{\sT} \Phi) \nabla X_0$ actually never appears in the homogenised equation. 
\end{rmk}

We are now ready to state our main theorem.

\begin{thm} \label{thm:fixed_pt_thm}
    Let $\alpha \in (0, 1)$ and $\eta \in \big( 1-\frac{3\alpha}{4}, 1-\frac{\alpha}{2} \big)$. Let $\{u_\eps(0)\}_{\eps\in \N^{-1}}$ be a sequence of functions in $\cC^\alpha(\T^2)$. For every $\eps = \eps_N = \frac{1}{N}$ $($including $\eps = 0)$, there exists $($a random$)$ $T_\eps>0$ such that the system \eqref{e:fixed_pt_system} has a unique solution $\vec{u}_\eps = (u_\eps, v_\eps, w_\eps) \in \xX_{T_\eps}$ almost surely. For $\eps = 0$, one takes the system \eqref{e:fixed_pt_system_limit} and solution $\vec{u}_0 = (u_0, v_0, 0)$. Moreover, the local existence time $T_\eps > 0$ and $\|\vec{u}_\eps\|_{\xX_{T_\eps}}$ depend on $\eps$ via $\|\Upsilon_\eps\|_{\yY}$ and $\|u_\eps(0)\|_{\cC^\alpha}$ only. In particular, if $\|\Upsilon_\eps\|_{\yY}$ and $\|u_\eps(0)\|_{\cC^\alpha}$ are uniformly bounded in $\eps$, then so is the local existence time $T_\eps$ and the norm $\|\vec{u}_\eps\|_{\xX_{T_\eps}}$. 

    Suppose $u_0(0) \in \cC^{\alpha'}$ for some $\alpha'>\alpha$ and that
    \begin{equation*}
        \|u_\eps(0) - u_0(0)\|_{\cC^\alpha} \rightarrow 0\;,  \qquad \|\Upsilon_\eps - \Upsilon_0\|_{\yY} \rightarrow 0
    \end{equation*}
    as $\eps \rightarrow 0$. Suppose further the solution $\vec{u}_0 = (u_0, v_0)$ with initial data $u(0)$ exists up to (a random) time $T_0 > 0$ with
    \begin{equation*}
        \|\vec{u}_0\|_{\xX_{T_0}} + \|\vec{u}_0\|_{L_{T_0}^\infty \cC_x^{\alpha'}} < +\infty\;.
    \end{equation*}
    Then for every $T<T_0$, there exists $\eps_0 > 0$ such that for every $\eps < \eps_0$, the solution $\vec{u}_\eps = (u_\eps, v_\eps, w_\eps)$ to \eqref{e:fixed_pt_system} also exists on $[0,T]$. Furthermore, there exists $\theta>0$ such that
    \begin{equation*}
        \|\vec{u}_\eps - \vec{u}_0\|_{\xX_T} \lesssim \eps^{\theta} + \|u_\eps(0) - u(0)\|_{\cC^\alpha} + \|\Upsilon_\eps - \Upsilon_0\|_{\yY}\;.
    \end{equation*}
   This in particular implies $u_\eps \rightarrow u_0$ in $\cC_{\fs}^{\alpha}([0,T] \times \T^2)$. 
\end{thm}

\begin{rmk}
One can take any $\alpha \in (0,1)$ as the regularity exponent for $u_\eps$, which is the main part of the solution. The reason $\alpha$ can be close to $0$ is that all the stochastic objects appearing in the system have regularities just below $0$, and does not depend on $\alpha$. The choices of regularity and weight exponents for the other two components of the solution ($v_\eps$ and $w_\eps$) are also flexible. The only constraint is that they need to have positive regularities to multiply with enhanced stochastic objects, and that the regularities and weights should be consistent so that one can close the fixed point problem. The choices $\alpha/4$ and $8\kappa$ made here are somewhat arbitrary. 
\end{rmk}

\subsection{Some preliminary lemmas}

We give bounds on various terms on the right hand sides of \eqref{e:fixed_pt_system} that are necessary in establishing the fixed point theorem. All the proportionality constants below depend on $g$ via the $L^\infty$-norm of its first three derivatives. 

In what follows, we fix the exponents $\alpha$ and $\eta$ as
\begin{equation} \label{eq:range_parameters}
    \alpha \in (0,1)\;, \qquad \eta \in \Big( 1 - \frac{3\alpha}{4}, 1 - \frac{\alpha}{2} \Big)\;,
\end{equation}
which are the same as in the statement of Theorem~\ref{thm:fixed_pt_thm}. We also fix an arbitrary $T^* > 0$, and the time interval we consider is always $[0,T]$ for $T \leq T^*$. All the proportionality constants are allowed to depend on $T^*$ (and $\alpha$ and $\eta$ as well), but are uniform in $T \in [0,T^*]$. 

We recall the functional $\Lambda_\eps$ from \eqref{e:functional_Lambda}, which essentially all quantities will depend on. We also recall the functionals $\aA_\eps$ from \eqref{e:functional_A}, $\sS_\eps$ from \eqref{e:functional_S_eps} and $\tT_\eps$ from \eqref{e:functional_T_eps}. Finally, recall the definition of the weighted space $L_T^{\infty,\sigma} \zZ$ from \eqref{e:space_weight}, and the parabolic spacetime norm $\cC_\fs^\gamma = \cC_\fs^\gamma([0,T] \times \T^2)$ from \eqref{e:spacetime}.

\begin{lem} \label{lem:g_locally_Lip}
    Let $\alpha \in (0,1)$. We have
    \begin{equation*}
        \|g(u)\|_{\cC_\fs^\alpha} \lesssim 1 + \|u\|_{\cC_\fs^{\alpha}}\;,
    \end{equation*}
    and
    \begin{equation*}
        \|g(u_1) - g(u_2)\|_{\cC_{\fs}^{\alpha}} \lesssim \big(1 + \|u_1\|_{\cC_\fs^{\alpha}} \big) \cdot \|u_1 - u_2\|_{\cC_{\fs}^{\alpha}}\;.
    \end{equation*}
\end{lem}
\begin{proof}
    The first bound is straightforward. For the second one, we need to show that the difference satisfies
    \begin{equation*}
        \begin{split}
        &\phantom{11} \left| \Big( g \big(u_1(z) \big) - g \big(u_2(z) \big) \Big) - \Big( g \big( u_1 (z') \big) - g \big( u_2(z') \big) \Big) \right|\\
        &\lesssim \big(1 + \|u_1\|_{\cC_\fs^{\alpha}} \big) \|u_1 - u_2\|_{\cC_{\fs}^{\alpha}} \cdot |z-z'|_{\fs}^\alpha
        \end{split}
    \end{equation*}
    uniformly over spacetime points $z,z' \in \R^+ \times \T^2$. Here, $|z-z'|_{\fs}$ denotes the parabolic distance between $z = (t,x)$ and $z' = (t',x')$ in the sense that
    \begin{equation*}
        |z-z'|_{\fs} := \sqrt{|t-t'|} + |x-x'|\;.
    \end{equation*}
    We distinguish the cases $\|u_1\|_{\cC_\fs^\alpha} |z-z'|_\fs^\alpha \lesssim \|u_1 - u_2\|_{\cC_{\fs}^{\alpha}}$ and $\|u_1\|_{\cC_\fs^\alpha} |z-z'|_\fs^\alpha \gtrsim \|u_1 - u_2\|_{\cC_{\fs}^{\alpha}}$. The desired bound then follows. 
\end{proof}

\begin{lem} \label{lem:para_term}
    Let $\alpha \in (0,1)$. We have the pointwise and Lipschitz estimates
    \begin{equation*}
        \|g(u) \prec \P_\m^\perp X_\eps\|_{\cC_\fs^{\alpha}} \lesssim \m^{-(1-\alpha-\kappa)} \big( 1 + \|u\|_{\cC_\fs^\alpha} \big) \|X_\eps\|_{ {\cC^{1-\kappa} }}
    \end{equation*}
    and
    \begin{equation*}
        \begin{split}
        &\phantom{111} \big\| \big( g(u_1) - g(u_2) \big) \prec \P_\m^\perp X_\eps \big\|_{\cC_{\fs}^{\alpha}}\\
        &\lesssim \m^{-(1-\alpha-\kappa)} \|X_\eps\|_{\cC^{1-\kappa}} \big( 1 + \|u_1\|_{\cC_\fs^\alpha} \big) \|u_1 - u_2\|_{\cC_\fs^\alpha}\;.
        \end{split}
    \end{equation*}
    In addition, we have the error estimate
    \begin{equation*}
        \begin{split}
        &\phantom{111}\big\| g(u_\eps) \prec \P_\m^{\perp} X_\eps - g(u_0) \prec \P_\m^{\perp} X_0  \big\|_{\cC_\fs^\alpha}\\
        &\lesssim \m^{-(1-\alpha-\kappa)} \, \big( 1 + \|u_0\|_{\cC_\fs^\alpha} \big) \cdot \big( \|X_\eps\|_{\cC^{1-\kappa}} \|u_\eps - u_0\|_{\cC_\fs^\alpha} + \|X_\eps - X_0\|_{\cC^{1-\kappa}} \big)\;.
        \end{split}
    \end{equation*}
\end{lem}
\begin{proof}
    By the para-product estimate in Lemma~\ref{lem:Bony} and the truncation estimate in Lemma~\ref{lem:Besov_cutoff}, we have
    \begin{equation*}
        \|g(u) \prec \P_\m^\perp X_\eps\|_{\cC_\fs^\alpha} \lesssim \|g(u)\|_{\cC_\fs^\alpha} \|\P_\m^\perp X_\eps\|_{\cC^{\alpha}} \lesssim \m^{-(1-\alpha-\kappa)} \|g(u)\|_{\cC_\fs^\alpha} \|X_\eps\|_{{\cC^{1-\kappa}}}\;.
    \end{equation*}
    The first claimed bound then follows immediately from Lemma~\ref{lem:g_locally_Lip}. The second and third one follow from the linearity in $g(u)$ and $\P_\m^{\perp} X_\eps$ and Lemma~\ref{lem:g_locally_Lip}. 
\end{proof}

\begin{lem} \label{lem:Lambda}
    Let $\alpha \in (0,1)$. For every sufficiently small $\kappa>0$, we have
    \begin{equation} \label{e:Lambda_uniform}
        \|\Lambda_\eps (u)\|_{\cC_{\fs}^{\alpha-2\kappa}} \lesssim (1+\m^{\alpha}) \, (1+\|u\|_{\cC_{\fs}^\alpha}) \, \|X_\eps\|_{\cC^{1-\kappa}}\;,
    \end{equation}
    and
    \begin{equation} \label{e:Lambda_Lipschitz}
        \begin{split}
        &\phantom{111}\|\Lambda_\eps(u_1) - \Lambda_\eps(u_2)\|_{\cC_{\fs}^{\alpha-2\kappa}}\\
        &\lesssim (1+\m^{\alpha}) \big(1 + \|u_1\|_{\cC_\fs^{\alpha}} \big) \, \|X_\eps\|_{\cC^{1-\kappa}} \, \|u_1 - u_2\|_{\cC_{\fs}^\alpha}\;.
        \end{split}
    \end{equation}
    As a consequence, we have
    \begin{equation} \label{e:Lambda_error_bound}
        \begin{split}
        \|\Lambda_\eps (u_\eps) - \Lambda_0 (u_0) &\|_{\cC_\fs^{\alpha-2\kappa}} \lesssim (1+\m^{\alpha}) \, \big( 1 + \|u_0\|_{\cC_\fs^\alpha} \big)\\
        &\cdot \big( \|X_\eps\|_{\cC^{1-\kappa}} \cdot \|u_\eps - u_0\|_{\cC_\fs^\alpha} + \|X_\eps - X_0\|_{\cC^{1-\kappa}} \big)\;.
        \end{split}
    \end{equation}
\end{lem}
\begin{proof}
    Recall the expression \eqref{e:functional_Lambda} for $\Lambda_\eps$. It suffices to show the bounds
    \begin{equation} \label{e:Lambda_bound}
        \begin{split}
        \|\nabla \big( g(u) \big) \prec  \P_{\m}^{\perp}X_\eps\|_{\cC_\fs^{\alpha-\kappa}} &\lesssim \|g(u)\|_{\cC_\fs^\alpha} \|X_\eps\|_{\cC^{1-\kappa}}\;,\\
        \|g(u) \succeq \nabla \P^{\perp}_{\m} X_\eps\|_{\cC_\fs^{\alpha-\kappa}} &\lesssim \|g(u)\|_{\cC_\fs^\alpha} \|X_\eps\|_{\cC^{1-\kappa}}\;,\\
        \|g(u)  \nabla \P_{\m} X_\eps\|_{\cC_\fs^{\alpha-\kappa}} &\lesssim \m^{\alpha} \|g(u)\|_{\cC_\fs^\alpha} \|X_\eps\|_{\cC^{1-\kappa}} \;.
        \end{split}
    \end{equation}
    The uniform bound \eqref{e:Lambda_uniform} will follow directly from the three bounds in \eqref{e:Lambda_bound} and Lemma~\ref{lem:g_locally_Lip}. The Lipschitz and error estimates \eqref{e:Lambda_Lipschitz} and \eqref{e:Lambda_error_bound} will follow in addition from that $\Lambda_\eps (u)$ is bilinear in $g(u)$ and $X_\eps$. 

    For all bounds in \eqref{e:Lambda_bound}, the $L_T^{\infty} \cC_x^{\alpha-2\kappa}$ parts follow from Lemmas~\ref{lem:Bony} and~\ref{lem:Besov_cutoff} directly, so we only consider the $\cC_T^{\frac{\alpha}{2}-\kappa} L_x^\infty$ parts of the norms. We give details for the first one in \eqref{e:Lambda_bound} only. We have
    \begin{equation*}
        \begin{split}
        &\phantom{111}\big\| \big( \nabla (g \circ u)(t) - \nabla (g \circ u)(s) \big) \prec \P_\m^\perp X_\eps \|_{L_x^\infty}\\
        &\lesssim \big\| \nabla (g \circ u)(t) - \nabla (g \circ u)(s) \|_{\cC_x^{-1+2\kappa}} \|\P_\m^\perp X_\eps\|_{\cC^{1-\kappa}}\\
        &\lesssim \|g\big(u(t)\big) - g\big( u(s) \big)\|_{\cC_x^{2\kappa}} \|X_\eps\|_{\cC^{1-\kappa}}\;.
        \end{split}
    \end{equation*}
    By interpolation, we have
    \begin{equation*}
        \begin{split}
        \|g\big(u(t)\big) - g\big( u(s) \big)\|_{\cC_x^{2\kappa}} &\lesssim \|g\big(u(t)\big) - g\big( u(s) \big)\|_{L_x^\infty}^{1-\frac{2\kappa}{\alpha}} \|g\big(u(t)\big) - g\big( u(s) \big)\|_{\cC_x^\alpha}^{\frac{2\kappa}{\alpha}}\\
        &\lesssim (t-s)^{\frac{\alpha}{2}-\kappa} \|g(u)\|_{\cC_\fs^\alpha}\;.
        \end{split}
    \end{equation*}
    This gives the $\cC_T^{\frac{\alpha}{2}-\kappa}L_x^\infty$ part of the first one in \eqref{e:Lambda_bound}. The other two bounds can be obtained in the same way. This completes the proof of the lemma. 
\end{proof}

The next lemma shows that $\id + \nabla \iI_\eps (\div a_\eps)$ behaves like $\Phi_\eps$. As a consequence, the stochastic object $\big( \id + \nabla \iI_\eps (\div a_\eps) \big)^{\sT} \fF_\eps$ can be obtained from $\fF_\eps$ and $\Phi_\eps^\sT \fF_\eps$ without referring to any extra information. 

\begin{lem} \label{lem:extra_0}
    We have the uniform bound
    \begin{equation} \label{e:extra_0_uniform}
        \|\nabla \iI_\eps (\div a_\eps)\|_{L_{T,x}^\infty} \lesssim 1\;.
    \end{equation}
    Furthermore, for every $\sigma>0$, there exist $\theta, \beta > 0$ such that
    \begin{equation} \label{e:extra_0_error}
        \big\| \id + \nabla \iI_\eps (\div a_\eps) - \Phi_\eps + \Phi_\eps  \nabla e^{t \lL_0} \big( \eps \chi(\cdot/\eps) \big) \big\|_{\cC_x^\beta} \lesssim \eps^\theta \cdot t^{-\frac{\sigma}{2}}\;.
    \end{equation}
    As a consequence, for every $\sigma>0$, there exists (a possibly different) $\theta>0$ such that for all sufficiently small $\kappa>0$, we have the bounds
    \begin{equation} \label{e:extra_0_stochastic}
        \begin{split}
        &\phantom{111}\big\| \big( \id + \nabla \iI_\eps (\div a_\eps) \big)^{\sT} \fF_\eps - \fF_0  \big\|_{\cC_x^{-\kappa}}\\
        &\lesssim \eps^{\theta} t^{-\frac{\sigma}{2}} \big( \|\fF_\eps\|_{\cC^{-\kappa}}  + \|\Phi_\eps^\sT \fF_\eps\|_{\cC^{-\kappa}} \big) + \|\Phi_\eps^{\sT} \fF_\eps - \fF_0\|_{\cC^{-\kappa}}\;,
        \end{split}
    \end{equation}
    and
    \begin{equation} \label{e:extra_00_stochastic}
        \begin{split}
        &\phantom{111}\big\| \big( \id + \nabla \iI_\eps (\div a_\eps) \big)^{\sT} a_\eps^\sT \nabla X_\eps - \Phi_\eps^\sT a_\eps^\sT \nabla X_\eps \big\|_{\cC_x^{-\kappa}}\\
        &\lesssim \eps^{\theta} t^{-\frac{\sigma}{2}} \big( \|a_\eps^\sT \nabla X_\eps\|_{\cC^{-\kappa}}  + \|\Phi_\eps^\sT a_\eps^\sT \nabla X_\eps\|_{\cC^{-\kappa}} \big)\;.
        \end{split}
    \end{equation}
\end{lem}
\begin{proof}
By definition of the corrector \eqref{eq: homo_corrector}, we have
\begin{equation*}
    \div (a_\eps) = - \div \Big( a_\eps (\nabla \chi)(\cdot / \eps) \Big) = - \lL_\eps \big( \eps \chi(\cdot / \eps) \big)\;.
\end{equation*}
Applying the operation $\nabla \iI_\eps$ on both sides above, we get
\begin{equation} \label{e:I_eps_div_a}
    \nabla \iI_\eps (\div a_\eps) = - \nabla \iI_\eps \lL_\eps \big( \eps \chi(\cdot / \eps) \big) = (\nabla \chi)(\cdot/\eps) - \nabla e^{t \lL_\eps} \big( \eps \chi(\cdot/\eps) \big)\;.
\end{equation}
The uniform bound \eqref{e:extra_0_uniform} then follows from \eqref{e:initial_space} (with $\beta=\gamma=1$) and that $\eps \chi (\cdot/\eps)$ is uniformly Lipschitz. 

Now, add the identity $\id$ to both sides of \eqref{e:I_eps_div_a}. By definition of $\Phi_\eps$, we have
\begin{equation*}
    \id + \nabla \iI_\eps (\div a_\eps) - \Phi_\eps + \Phi_\eps \nabla e^{t \lL_0} \big( \eps \chi(\cdot/\eps) \big) = - \big( \nabla e^{t \lL_\eps} - \Phi_\eps \nabla e^{t \lL_0} \big) \big( \eps \chi(\cdot/\eps) \big)\;.
\end{equation*}
The bound \eqref{e:extra_0_error} then follows from Lemma~\ref{lem:initial_error} (with $\gamma=1$) and that $\eps \chi (\cdot / \eps)$ is uniformly Lipschitz. 

Note that $\kappa$ is chosen arbitrarily small (in particular smaller than $\frac{\sigma}{2}$ and $\beta$ for the $\sigma$ and $\beta$ in \eqref{e:extra_0_error}). As a consequence, we have
\begin{equation*}
    \begin{split}
    &\phantom{111}\big\| \big( \id + \nabla \iI_\eps (\div a_\eps) \big)^{\sT} \fF_\eps - \Phi_\eps^{\sT} \fF_\eps \big\|_{\cC_x^{-\kappa}}\\
    &\lesssim \big\| \nabla e^{t \lL_0} \big( \eps \chi(\cdot/\eps) \big) \big\|_{\cC_x^{2\kappa}} \|\Phi_\eps^{\sT} \fF_\eps\|_{\cC^{-\kappa}} + \eps^\theta t^{-\frac{\sigma}{2}} \|\fF_\eps\|_{\cC^{-\kappa}}\\
    &\lesssim \eps^\sigma t^{-\frac{2\kappa+\sigma}{2}} \|\Phi_\eps^{\sT} \fF_\eps\|_{\cC^{-\kappa}} + \eps^\theta t^{-\frac{\sigma}{2}} \|\fF_\eps\|_{\cC^{-\kappa}}\;.
    \end{split}
\end{equation*}
This implies \eqref{e:extra_0_stochastic} with a different but still arbitrarily small $\sigma$ (since $\kappa$ can be arbitrarily small). Replacing $\fF_\eps$ by $a_\eps^{\sT} \nabla X_\eps$ gives \eqref{e:extra_00_stochastic}. We have thus completed the proof of the lemma. 
\end{proof}

\begin{rmk} \label{rmk:Linfty}
    One shall compare the bound \eqref{e:extra_0_uniform} with Lemma~\ref{lem:Lp_critical}. The claim in Lemma~\ref{lem:Lp_critical} is in general false for $p=+\infty$. The reason one has an $L_{T,x}^\infty$ bound here is that the oscillation in the operator $\iI_\eps$ is the same as that of $a_\eps$ which is in the divergence operator on the right hand side. 
\end{rmk}

In what follows, all quantities in consideration are linear in $v$, $w$ and $\Lambda(u)$, and at most quadratic in $g(u)$ and $g'(u)$. Hence, we only provide their uniform-in-$\eps$ bounds and error estimates. The local Lipschitzness (that is, difference of the quantities in terms of $(u_1,v_1,w_1) - (u_2,v_2,w_2)$) follows from the uniform bounds, the linearities in $v$ and $w$, and the local Lipschitzness of $g$ and $\Lambda_\eps$.


\begin{lem} \label{lem:A_F}
    Let $\alpha, \eta$ be as in \eqref{eq:range_parameters}. Recall the expression $\aA_\eps^\sT \fF_\eps$ given in \eqref{e:expression_A_flux}. We have the uniform bound
    \begin{equation*}
        \|\aA_\eps^\sT \big(\vec{u}_\eps \big) \fF_\eps\|_{L_T^{\infty,\eta} \cC_x^{-\kappa}} \lesssim \big( \|\vec{u}_\eps\|_{\xX_T} + \|\Lambda_\eps (u_\eps)\|_{\cC_\fs^{\alpha-2\kappa}} \big) \|\Upsilon_\eps\|_{\yY}\;.
    \end{equation*}
    Furthermore, there exists $\theta>0$ such that for all sufficiently small $\kappa$, we have
    \begin{equation*}
        \begin{split}
        &\phantom{111}\|\aA_\eps^{\sT}(\vec{u}_\eps) \fF_\eps - \aA_0^\sT (\vec{u}_0) \fF_0\|_{L_T^{\infty,\eta} \cC_x^{-\kappa}}\\
        &\lesssim \|\Upsilon_\eps\|_{\yY} \cdot \big( \|\vec{u}_\eps - \vec{u}_0\|_{\xX_T} + \|\Lambda_\eps (u_\eps) - \Lambda_0 (u_0)\|_{\cC_\fs^{\alpha-2\kappa}} \big) \\
        &\phantom{11}+ \big( \|\vec{u}_0\|_{\xX_T} + \|\Lambda_0 (u_0)\|_{\cC_\fs^{\alpha-2\kappa}} \big) \cdot \big( \eps^\theta \|\Upsilon_\eps\|_{\yY} + \|\Upsilon_\eps - \Upsilon_0\|_{\yY} \big)\;.
        \end{split}
    \end{equation*}
\end{lem}
\begin{proof}
    We give details for the uniform-in-$\eps$ bound. The error estimates follow in the same way. For the first two terms on the right hand side of \eqref{e:expression_A_flux}, it is straightforward to show that
    \begin{equation*}
        \|v_\eps^{\sT} \Phi_\eps^{\sT} \fF_\eps\|_{L_T^{\infty,\eta} \cC_x^{-\kappa}} \lesssim \|v_\eps\|_{\xX_T^2} \|\Upsilon_\eps\|_{\yY}\;, \quad \|w_\eps^\sT \fF_\eps\|_{L_T^{\infty,\eta} \cC_x^{-\kappa}} \lesssim \|w_\eps\|_{\xX_T^3} \|\Upsilon_\eps\|_{\yY}\;.
    \end{equation*}
    The proof in Lemma~\ref{lem:extra_0} also shows that for every $\sigma>0$, there exists $\theta>0$ such that
    \begin{equation*}
        \big\| \big( \id + \nabla \iI_\eps (\div a_\eps) \big)^\sT \fF_\eps \big\|_{\cC_x^{-\kappa}} \lesssim \big(1 + \eps^\theta \cdot t^{-\frac{\sigma}{2}} \big) \cdot \big( \|\Phi_\eps^{\sT} \fF_\eps\|_{\cC^{-\kappa}} + \|\fF_\eps\|_{\cC^{-\kappa}} \big)\;.
    \end{equation*}
    Choosing $\sigma<\eta$, the uniform bound then follows. The claim for the difference $\aA_\eps^\sT \fF_\eps - \aA_0^\sT \fF_0$ also follows in a similar way. 
\end{proof}

\begin{prop} \label{prop:S}
    Let $\alpha, \eta$ be as in \eqref{eq:range_parameters}. Recall the functional $\sS_\eps$ from \eqref{e:functional_S_eps}. We have
    \begin{equation*}
        \|\sS_\eps (\vec{u}_\eps)\|_{L_T^{\infty,\eta} \cC_x^{-\kappa}} \lesssim (1 + \m^\alpha) \big( 1 + \|\Upsilon_\eps\|_{\yY}^2 \big) \big( 1 + \|\vec{u}_\eps\|_{\xX_T}^2 \big)\;,
    \end{equation*}
    and
    \begin{equation*}
        \begin{split}
        \|\sS_\eps (\vec{u}_\eps) - &\sS_0 (\vec{u}_0)\|_{L_T^{\infty,\eta} \cC_x^{-\kappa}} \lesssim (1 + \m^\alpha) \big(1 + \|\Upsilon_0\|_\yY^2 + \|\Upsilon_\eps\|_\yY^2 \big) \cdot\\
        &\big( 1 + \|\vec{u}_0\|_{\xX_T}^2 + \|\vec{u}_\eps\|_{\xX_T}^2 \big) \, \Big[ \eps^\theta +  \|\vec{u}_\eps - \vec{u}_0\|_{\xX_T} + \|\Upsilon_\eps - \Upsilon_0\|_\yY \Big]\;.
        \end{split}
    \end{equation*}
\end{prop}
\begin{proof}
    This is straightforward from the expression of $\sS_\eps$ in \eqref{e:functional_S_eps} and Lemma~\ref{lem:A_F}, and then using Lemmas~\ref{lem:g_locally_Lip} and~\ref{lem:Lambda} to bound relevant quantities involving $g$ and $\Lambda_\eps$. 
\end{proof}

\begin{prop} \label{prop:T}
Let $\alpha, \eta$ be as in \eqref{eq:range_parameters}. Recall the definition of $\tT_\eps$ and $\widetilde{\tT}_\eps$ from \eqref{e:functional_T_eps} and \eqref{e:functional_T_tilde_eps}. We have the uniform bound
\begin{equation*}
    \| \tT_\eps (\vec{u}_\eps) \|_{L_T^{\infty,2\eta} \cC_{x}^{-\kappa}} \lesssim (1 + \m^{2\alpha}) \big( 1 + \|\Upsilon_\eps\|_\yY^3 \big) \big( 1 + \|\vec{u}_\eps\|_{\xX_T}^3 \big)\;.
\end{equation*}
\end{prop}
\begin{proof}
By definition of $\tT_\eps$ and Lemmas~\ref{lem:Bony} and~\ref{lem:Besov_cutoff}, we have
\begin{equation*}
    \begin{split}
    \|\tT_\eps (\vec{u}_\eps)\|_{L_T^{\infty, 2\eta} \cC_x^{-\kappa}} &\lesssim \|\widetilde{\tT}_\eps (\vec{u}_\eps)\|_{L_T^{\infty, 2\eta} \cC_x^{-1}} \|\P_\m^\perp X_\eps\|_{\cC^{1-2\kappa}}\\
    &\lesssim \|\widetilde{\tT}_\eps (\vec{u}_\eps)\|_{L_T^{\infty, 2\eta} \cC_x^{-1}} \|X_\eps\|_{\cC^{1-\kappa}}\;.
    \end{split}
\end{equation*}
We now control $\|\widetilde{\tT}_\eps (\vec{u}_\eps)\|_{L_T^{\infty, 2\eta} \cC_x^{-1}}$. 
By the expression \eqref{e:functional_T_tilde_eps}, we have
\begin{equation*}
    \begin{split}
    \| \widetilde{\tT}_\eps (\vec{u}_\eps) &\|_{L_T^{\infty, 2\eta} \cC_x^{-1}} \lesssim \|\aA_\eps(\vec{u}_\eps)\|_{L_T^{\infty,\eta} L_x^\infty} + \|g'(u_\eps)\|_{\cC_\fs^\alpha} \|\sS_\eps(\vec{u}_\eps)\|_{L_T^{\infty,\eta} \cC_x^{-\kappa}}\\
    &\phantom{11}+ \|\aA_\eps(\vec{u}_\eps)\|_{L_T^{\infty,\eta} L_x^\infty}^2 + \|g''(u_\eps) \, g(u_\eps)\|_{\cC_\fs^\alpha} \|\aA_\eps^\sT(\vec{u}_\eps) a_\eps^\sT \nabla X_\eps\|_{L_T^{\infty,\eta} \cC_x^{-\kappa}}\;.
    \end{split}
\end{equation*}
Here, the weight $2\eta$ comes from the fact that there is a product between two $\aA_\eps$ on the right hand side, each carrying a weight $\eta$. By \eqref{e:extra_0_uniform} and Lemma~\ref{lem:Lambda}, we have
\begin{equation*}
    \begin{split}
    \|\aA_\eps (\vec{u}_\eps)\|_{L_T^{\infty,\eta} L_x^\infty} &\lesssim \|\vec{u}_\eps\|_{\xX_T} + \|\Lambda_\eps (u_\eps)\|_{\cC_\fs^{\alpha-2\kappa}}\\
    &\lesssim (1 + \m^\alpha) \big( 1 + \|\Upsilon_\eps\|_\yY \big) \big( 1 + \|\vec{u}_\eps\|_{\cC_\fs^\alpha} \big)\;.
    \end{split}
\end{equation*}
By Lemma~\ref{lem:A_F} (replacing $\fF_\eps$ by $a_\eps^{\sT} \nabla X_\eps$) and also Lemma~\ref{lem:Lambda}, we have
\begin{equation*}
    \begin{split}
    \|\aA_\eps^{\sT}(\vec{u}_\eps) a_\eps^{\sT} \nabla X_\eps\|_{L_T^{\infty,\eta} \cC_x^{-\kappa}} &\lesssim \|\Upsilon_\eps\|_\yY \cdot \big( \|\vec{u}_\eps\|_{\xX_T} + \|\Lambda_\eps (u_\eps)\|_{\cC_\fs^{\alpha-2\kappa}} \big)\\
    &\lesssim (1 + \m^\alpha) \big( 1 + \|\Upsilon_\eps\|_\yY^2 \big) \cdot \big( 1 + \|\vec{u}_\eps\|_{\cC_\fs^\alpha} \big).
    \end{split}
\end{equation*}
The desired claim then follows by using Lemma~\ref{lem:g_locally_Lip} to control expressions involving $g$ and Proposition~\ref{prop:S} to bound $\sS_\eps$ and then combining all the above. 
\end{proof}

\begin{prop} \label{prop:extra_1}
    Let $\alpha \in (0,1)$. For every $\beta \in (\alpha,1)$, there exists $\theta>0$ such that
    \begin{equation} \label{e:extra_1_bounded}
        \left\| \iI_\eps \Big( \div \big( a_\eps \Lambda_\eps(u_\eps) \big) \Big) \right\|_{\cC_{\fs}^{\beta}([0,T]\times \T^{2})} \lesssim T^{\theta} (1 + \m^\alpha) \|\Upsilon_\eps\|_\yY \big( 1 + \|u_\eps\|_{\cC_\fs^\alpha} \big)\;.
    \end{equation}
    Furthermore, there exists a (possibly different) $\theta>0$ such that
    \begin{equation} \label{e:extra_1_convergence}
        \begin{split}
        \Big\| \iI_\eps \Big( \div &\big( a_\eps \Lambda_\eps(u_\eps) \big) \Big) - \iI_0 \Big( \div \big( \bar{a} \Lambda_0(u_0) \big) \Big)  \Big\|_{\cC_\fs^{\alpha}([0,T]\times \T^{2})} \lesssim T^\theta (1 + \m^\alpha)\\
        &\cdot \big(1 + \|\Upsilon_\eps\|_\yY \big) \big( 1 + \|u_\eps\|_{\cC_\fs^\alpha} \big) \big( \eps^{\theta} + \|\Upsilon_\eps - \Upsilon_0\|_\yY + \|u_\eps - u_0\|_{\cC_\fs^\alpha} \big)
        \end{split}
    \end{equation}
    uniformly over $\eps \in [0,1]$. 
\end{prop}
\begin{proof}
For simplicity, we write $F_\eps = \Lambda_\eps (u_\eps)$ for $\eps \in \N^{-1}$. By Lemma~\ref{lem:I_eps_div_spacetime_norm}, we have
\begin{equation*}
    \left\| \iI_\eps \big( \div ( a_\eps F_\eps ) \big) \right\|_{\cC_{\fs}^{\beta}([0,T]\times \T^{2})} \lesssim T^\theta \|F_\eps\|_{L_{T,x}^\infty}\;.
\end{equation*}
The uniform bound \eqref{e:extra_1_bounded} then follows from the bound \eqref{e:Lambda_uniform} for $F_\eps = \Lambda_\eps (u_\eps)$. We now turn to the difference \eqref{e:extra_1_convergence}. Using integration by parts and the two-scale expansion of $\nabla_y^\sT Q_\eps$ in \eqref{e:difference_first_order}, we have
\begin{equation} \label{e:extra_1_expansion}
    \begin{split}
    \Big( \iI_\eps \big( \div (a_\eps &F_\eps) \big) \Big)(t,x) = - \int_{0}^{t} \int_{\T^2} (\nabla_y^{\sT} Q_\eps)(t-r,x,y) \, a_\eps(y) \, F_\eps (r,y) \, {\rm d}y {\rm d}r\\
    &= - \int_0^t \int_{\T^2} (\nabla_y^\sT Q_0)(t-r,x,y) \big(\Phi_\eps^*(y)\big)^\sT a_\eps (y) \, F_\eps (r,y) \, {\rm d}y {\rm d}r\\
    &\phantom{111}- \int_0^t \int_{\T^2} (R_\eps^*)^{\sT}(t-r,x,y) a_\eps (y) F_\eps (r,y) \, {\rm d}y {\rm d}r\;,
    \end{split}
\end{equation}
where $(R_\eps^*)^{\sT}$ is the transpose of the kernel $R_\eps^{*}$ given in \eqref{e:difference_first_order}. Integrating by parts again, we can write the difference on the left hand side of \eqref{e:extra_1_convergence} as
\begin{equation*}
    \iI_\eps \big( \div (a_\eps F_\eps) \big) - \iI_0 \big( \div (\bar{a} F_0) \big) = \iI_0 \Big( \div \big( (\Phi_\eps^*)^{\sT} a_\eps F_\eps - \bar{a} F_0 \big) \Big) - (\rR_\eps^*)^{\sT} \big( a_\eps F_\eps \big)\;,
\end{equation*}
where $(\rR_\eps^{*})^{\sT}$ means multiplying the kernel $(R_\eps^{*})^{\sT}$ from the left and then taking spacetime integration. 

For the first term on the right hand side above, by the Schauder estimate for $\iI_0$, we have
\begin{equation*}
    \begin{split}
    &\phantom{111}\left\| \iI_0 \Big( \div \big( (\Phi_\eps^*)^{\sT} a_\eps F_\eps - \bar{a} F_0 \big) \Big) \right\|_{\cC_\fs^\alpha}\\
    &\lesssim T^\theta \Big( \| (\Phi_\eps^*)^{\sT} a_\eps - \, \bar{a} \|_{\cC_x^{-\kappa}} \|F_\eps\|_{L_T^\infty \cC_x^{2\kappa}} + \|F_\eps - F_0\|_{L_T^\infty \cC_{x}^{2\kappa}} \Big)\;.
    \end{split}
\end{equation*}
For the remainder $(\rR_\eps^*)^{\sT} (a_\eps F_\eps)$, by Proposition~\ref{prop:Green_remainder_point}, we can control its $L_{T,x}^{\infty}$-norm by
\begin{equation*}
    \|(\rR_\eps^*)^{\sT} (a_\eps F_\eps)\|_{L_{T,x}^\infty} \lesssim (\eps T)^{\theta} \|F_{\eps}\|_{L_{T,x}^{\infty}}
\end{equation*}
for some $\theta>0$. Also, by the expression \eqref{e:extra_1_expansion} and Lemma~\ref{lem:I_eps_div_spacetime_norm}, we have
\begin{equation*}
    \begin{split}
    \|(\rR_\eps^*)^{\sT} (a_\eps F_\eps)\|_{\cC_\fs^\beta} &\leq \big\|\iI_\eps \big( \div (a_\eps F_\eps) \big) \big\|_{\cC_\fs^\beta} + \big\| \iI_0 \big( \div ((\Phi_\eps^*)^{\sT} a_\eps F_\eps) \big) \big\|_{\cC_\fs^\beta}\\
    &\lesssim T^\theta \|F_\eps\|_{L_{T,x}^\infty}\;.
    \end{split}
\end{equation*}
Interpolating the above two bounds with $\beta>\alpha$, we get
\begin{equation*}
    \big\| (\rR_\eps^*)^{\sT} (a_\eps F_\eps) \big\|_{\cC_\fs^\alpha} \lesssim (\eps T)^\theta \|F_\eps\|_{L_{T,x}^\infty}
\end{equation*}
for some possibly different $\theta > 0$. The proof is complete by plugging the bounds for $F_\eps = \Lambda_\eps (u_\eps)$ and $F_\eps - F_0 = \Lambda_\eps (u_\eps) - \Lambda_0 (u_0)$ from Lemma~\ref{lem:Lambda} and $\|(\Phi_\eps^*)^\sT a_\eps - \bar{a}\|_{\cC^{-\kappa}}$ from Lemma~\ref{lem:oscillation}. 
\end{proof}

We now turn to the second term on the right hand side of the second equation in \eqref{e:fixed_pt_system}. Write $F_\eps = \Lambda_\eps (u_\eps)$ and $H_\eps = (\Phi_\eps^*)^\sT a_\eps$. Then $H_\eps = H(\cdot / \eps)$ with $H = (\Phi^*)^{\sT} a$ and $\Pi_0 H = \bar{a}$. Let
\begin{equation*}
    \begin{split}
    \vV_\eps(t,x) &:= \int_{0}^{t} \int_{\T^2} (\nabla^2 Q_0)(t-r,x,y) H_\eps(y) \big( F_\eps (r,y) - F_\eps (t,x) \big) {\rm d}y {\rm d}r\;,\\
    \vV_0(t,x) &:= \int_{0}^{t} \int_{\T^2} (\nabla^2 Q_0)(t-r,x,y) \, \Pi_0 H \, \big( F_0 (r,y) - F_0 (t,x) \big) {\rm d}y {\rm d}r\;.
    \end{split}
\end{equation*}
It is standard that for every $0 < \beta < \gamma$, there exists $\theta>0$ such that
\begin{equation*}
    \|\vV_0\|_{L_T^\infty \cC_x^\beta} \lesssim T^\theta \|F_0\|_{\cC_\fs^\gamma}\;.
\end{equation*}
But controlling $\|\vV_\eps\|_{L_T^\infty \cC_x^\beta}$-norm uniformly in $\eps$ as well as the difference $\vV_\eps - \vV_0$ is nontrivial since $H_\eps$ has only $L^\infty$-regularity uniformly in $\eps$. We focus on controlling the difference $\vV_\eps - \vV_0$, and the uniform bound for $\vV_\eps$ will follow as a consequence. 

Since $H_\eps = H(\cdot / \eps)$, we have $\Pi_0 H = \Pi_0 H_\eps$ and hence $H_\eps - \Pi_0 H = \Pi_0^\perp H_\eps$. Also, since $\Pi_0 H$ is constant and $\nabla^2 Q$ integrates to $0$ on the torus, we have $\nabla^2 \iI_0 (\Pi_0 H) = 0$. Hence, we have
\begin{equation*}
    \begin{split}
    \vV_\eps &= \nabla^2 \iI_0 \big( H_\eps F_\eps \big) - \nabla^2 \iI_0 (H_\eps) F_\eps = \nabla^2 \iI_0 \big( H_\eps F_\eps \big) - \nabla^2 \iI_0 \big( \Pi_0^\perp H_\eps \big) F_\eps\\
    \vV_0 &= \nabla^2 \iI_0 \big( \Pi_0 H \cdot F_0 \big) - \nabla^2 \iI_0 \big( \Pi_0 H \big) F_0 = \nabla^2 \iI_0 \big( \Pi_0 H \cdot F_0 \big)\;.
    \end{split}
\end{equation*}
Thus, the difference $\vV_\eps - \vV_0$ can be decomposed into
\begin{equation} \label{e:difference_V_expression}
    \vV_\eps - \vV_0 = \Big( \nabla^2 \iI_0 \big( \Pi_0^\perp H_\eps \cdot F_\eps \big) - \nabla^2 \iI_0 \big( \Pi_0^\perp H_\eps \big) F_\eps \Big) + \nabla^2 \iI_0 \big( \Pi_0 H \, (F_\eps - F_0) \big)\;.
\end{equation}
We have the following bound regarding the first component in $\vV_\eps - \vV_0$.

\begin{lem} \label{lem:second_main_ingredient}
Let $H: \T^2 \rightarrow \R^{2 \times 2}$ and write $H_\eps := H(\cdot/\eps)$. Then for every $0 < \beta < \gamma < 1$, there exists $\theta>0$ such that
\begin{equation*}
    \begin{split}
   \left\| \nabla^2 \iI_0 \big( \Pi_0^\perp H_\eps \cdot F_\eps \big) - \nabla^2 \iI_0 \big( \Pi_0^\perp H_\eps \big) F_\eps \right\|_{L_T^\infty \cC_x^\beta}\lesssim (\eps T)^\theta \big( 1+\|H\|_{L^\infty} \big) \|F_\eps\|_{\cC_\fs^\gamma}
    \end{split}
\end{equation*}
uniformly over $\eps \in \N^{-1}$, $H \in L^\infty(\T^2; \R^{2 \times 2})$ and $F \in \cC_\fs^\gamma \big([0,T] \times \T^2; \R^2\big)$. 
\end{lem}
\begin{proof}
We can assume without loss of generality that $\Pi_0 H = 0$. We have
\begin{equation*}
    \big( \nabla^2 \iI_0 \big( \Pi_0^\perp H_\eps \cdot F_\eps \big) - \nabla^2 \iI_0 \big( \Pi_0^\perp H_\eps \big) F_\eps \big)(t,x) = \int_{0}^{t} \jJ_\eps (t,r,x) {\rm d}r\;,
\end{equation*}
where
\begin{equation*}
    \jJ_\eps (t,r,x) = \int_{\T^2} (\nabla^2 Q_0)(t-r,x,y) H_\eps (y) \big(  F_\eps(r,y) - F_\eps(t,x) \big) {\rm d}y\;.
\end{equation*}
First, we have the straightforward pointwise bound
\begin{equation} \label{e:rho_pointwise_crude}
    \begin{split}
    |\jJ_\eps(t,r,x)| &\lesssim \Big( \int_{\T^2} \big( \sqrt{t-r} + |x-y| \big)^{-4+\gamma} {\rm d}y \Big) \|H\|_{L^\infty} \|F_\eps\|_{\cC_\fs^\gamma}\\
    &\lesssim (t-r)^{-1+\frac{\gamma}{2}} \|H\|_{L^\infty} \|F_\eps\|_{\cC_\fs^\gamma}\;.
    \end{split}
\end{equation}
Second, by Lemma~\ref{lem:oscillation}, we also have the pointwise bound
\begin{equation} \label{e:rho_pointwise_oscillation}
    \begin{split}
    |\jJ_\eps(t,r,x)| &\lesssim \eps^\theta \big( 1 + \|H\|_{L^\infty} \big) \left\| \nabla^2 Q_0(t-r,x,\cdot) \big( F_\eps(r,\cdot) - F_\eps(t,x) \big) \right\|_{\cC_y^\gamma}\\
    &\lesssim \eps^\theta (t-r)^{-2-\frac{\gamma}{2}} \big( 1 + \|H\|_{L^\infty} \big) \|F_\eps\|_{\cC_\fs^\gamma}\;,
    \end{split}
\end{equation}
uniformly over all $\eps \in \N^{-1}$, $x \in \T^2$ and $0<r<t<T$. Finally, we have
\begin{equation} \label{e:rho_Holder}
    \begin{split}
    \|\jJ_\eps(t,r,\cdot)\|_{\cC_x^\gamma} &\lesssim \Big(\sup_{y \in \T^2} \|(\nabla^2 Q_0)(t-r,\cdot,y)\|_{\cC_x^\gamma} \Big) \, \|H\|_{L^\infty} \|F_\eps\|_{\cC_\fs^\gamma}\\
    &\lesssim (t-r)^{-2-\frac{\gamma}{2}} \|H\|_{L^\infty} \|F_\eps\|_{\cC_\fs^\gamma}\;.
    \end{split}
\end{equation}
Combining \eqref{e:rho_pointwise_crude} and \eqref{e:rho_pointwise_oscillation} and interpolating with \eqref{e:rho_Holder}, we get
\begin{equation*}
    \|\jJ_{\eps}(t,r,\cdot)\|_{\cC_x^\beta} \lesssim \eps^\theta (t-r)^{-1+\theta} \big(1 + \|H\|_{L^\infty}\big) \|F_\eps\|_{\cC_x^\gamma}
\end{equation*}
for a possibly different $\theta$ (depending on $\beta$ and $\gamma$). The claim then follows by integrating out $r \in [0,t]$. 
\end{proof}

We have the following proposition.

\begin{prop} \label{prop:second_main}
    There exists $\theta > 0$ such that for sufficiently small $\kappa>0$, we have
    \begin{equation*}
        \begin{split}
        \big\|  \vV_\eps - \vV_0 \big\|_{L_T^\infty \cC_{x}^{\alpha - 8\kappa}} \lesssim &\phantom{1}T^\theta (1 + \m^\alpha) \big( 1 + \|X_\eps\|_{\cC^{1-\kappa}} \big) \big( 1 + \|u_\eps\|_{\cC_\fs^\alpha} \big)\\
        &\cdot \Big( \eps^\theta + \|u_\eps - u_0\|_{\cC_\fs^\alpha} + \|X_\eps - X_0\|_{\cC^{1-\kappa}} \Big)\;.
        \end{split}
    \end{equation*}
\end{prop}
\begin{proof}
    By the Schauder estimate, there exists $\theta>0$ such that
    \begin{equation*}
        \|\nabla^2 \iI_0 \big( \bar{a} (F_\eps - F_0) \big)\|_{L_T^\infty \cC_x^{\alpha-8\kappa}} \lesssim T^\theta \|F_\eps - F_0\|_{\cC_\fs^{\alpha-2\kappa}}\;.
    \end{equation*}
    Combining this bound with the expression \eqref{e:difference_V_expression} as well as Lemma~\ref{lem:second_main_ingredient} with $\beta = \alpha-8\kappa$ and $\gamma = \alpha-2\kappa$, and then applying Lemma~\ref{lem:Lambda} to control $F_\eps$ and $F_\eps - F_0$, the desired claim follows. 
\end{proof}

\begin{prop}\label{prop: widetilde R}
    There exist $\beta, \theta>0$ such that
     \begin{equation*}
        \begin{split}
        &\phantom{111}\bigg\| \int_{0}^{t} \int_{\T^2} \widetilde{R}_\eps (t-r,x,y) a_\eps(y) \big( (\Lambda_\eps u_\eps)(r,y) - (\Lambda_\eps u_\eps)(t,x) \big) \mathrm{d}y \mathrm{d} r\bigg\|_{L_{T}^{\infty} \cC_{x}^{\beta}}\\
        &\lesssim (\eps T)^{\theta} (1 + \m^\alpha) \|X_\eps\|_{\cC^{1-\kappa}} \big( 1 + \|u_\eps\|_{\cC_\fs^\alpha} \big)
        \end{split}
    \end{equation*}
    uniformly over $\eps \in \N^{-1}$. 
\end{prop}
\begin{proof}
We again write $F_\eps = \Lambda_\eps u_\eps$. We first consider the integration in $y \in \T^2$. By the bound for $\widetilde{R}_\eps$ in Proposition~\ref{prop:Green_remainder_point}, there exist $\theta' > 0$ and $\sigma' < 1$ such that
\begin{equation*}
    \Big| \int_{\T^2} \widetilde{R}_\eps(t-r,x,y) a_\eps (y) \big( F_\eps(r,y) - F_\eps(t,x) \big) {\rm d}y\Big| \lesssim \eps^{\theta'} (t-r)^{-\sigma'} \|F_\eps\|_{\cC_\fs^{\alpha-2\kappa}}
\end{equation*}
uniformly over $r \in [0,t]$ and $x \in \T^2$. On the other hand, by Proposition~\ref{prop:Green_remainder_derivative}, for $\gamma < \alpha$, there exist $\lambda_1, \lambda_2 > 0$ such that
\begin{equation*}
    \Big\| \int_{\T^2} \widetilde{R}_\eps(t-r,x,y) a_\eps (y) \big( F_\eps(r,y) - F_\eps(t,x) \big) {\rm d}y\Big\|_{\cC_x^\gamma} \lesssim \eps^{-\lambda_1} (t-r)^{-\lambda_2} \|F_\eps\|_{\cC_\fs^{\alpha-2\kappa}}\;.
\end{equation*}
Since $\sigma' < 1$, if $\beta>0$ is sufficiently small, we can interpolate the above two bounds to get
\begin{equation*}
    \Big\| \int_{\T^2} \widetilde{R}_\eps(t-r,x,y) a_\eps (y) \big( F_\eps(r,y) - F_\eps(t,x) \big) {\rm d}y\Big\|_{\cC_x^\beta} \lesssim \eps^{\theta} (t-r)^{-\sigma} \|F_\eps\|_{\cC_\fs^{\alpha-2\kappa}}\;,
\end{equation*}
for some $\theta>0$ and $\sigma<1$. Integrating $r \in [0,t]$ and applying Lemma~\ref{lem:Lambda} to control $\|F_\eps\|_{\cC_\fs^{\alpha-2\kappa}}$, we can thus conclude the proof of the proposition. 
\end{proof}

\subsection{Proof of Theorem~\ref{thm:fixed_pt_thm}}

We are now ready to prove Theorem~\ref{thm:fixed_pt_thm}. For $\eps \in \N^{-1}$, let
\begin{equation*}
    \mM_\eps := 1 + \|\Upsilon_\eps\|_{\yY} + \|u_\eps(0)\|_{\cC^\alpha}\;.
\end{equation*}
Recall the definition of $u_\eps^\#$ from \eqref{e:u_sharp}. We have
\begin{equation*}
    \|u_\eps^{\#}(0)\|_{\cC^\alpha} \lesssim \|u_\eps(0)\|_{\cC^\alpha} + \|g\|_{L^\infty} \|X_\eps\|_{\cC^\alpha} \lesssim \mM_\eps\;.
\end{equation*}
We first show that for each $\eps \in \N^{-1}$, there is a unique $\vec{u}_\eps = (u_\eps, v_\eps, w_\eps) \in \xX_{T_\eps}$ satisfying the system \eqref{e:fixed_pt_system} up to a maximal existence time $T_\eps$ (and a unique $\vec{u}_0 = (u_0, v_0)$ satisfying \eqref{e:fixed_pt_system_limit} for $\eps=0$). Furthermore, the maximal existence time $T_\eps$ depends on $\eps$ via $\mM_\eps$ only. 

Define the map $\Gamma_{\eps} := (\Gamma_{\eps}^{(1)}, \Gamma_{\eps}^{(2)}, \Gamma_{\eps}^{(3)})$ on triples of spacetime functions $\vec{u} = (u,v,w)$ by
\begin{equation*}
    \begin{split}
    \Gamma_\eps^{(1)} \vec{u} &= e^{t \lL_\eps} u_\eps^{\#}(0) + g(u) \prec \P_{\m}^\perp X_\eps + \iI_\eps \Big( \div \big( a_\eps \Lambda_\eps(u) \big) + \sS_\eps(\vec{u}) - \tT_\eps(\vec{u}) \Big)\;,\\
    \Gamma_\eps^{(2)} \vec{u} &= \nabla e^{t \lL_0} u_\eps^{\#}(0) + \nabla \iI_0 \big( \sS_\eps(\vec{u}) - \tT_\eps(\vec{u}) \big) - \int_{0}^{t} \int_{\T^2} (\nabla_{x,y}^2 Q_0)(t-r,x,y) \cdot\\
    &\phantom{1}\big(\Phi_\eps^*(y)\big)^{\sT} a_\eps(y) \Big( (\Lambda_\eps u)(r,y) - (\Lambda_\eps u)(t,x) \Big) {\rm d}y {\rm d}r\;,\\
    \Gamma_\eps^{(3)} \vec{u} &= \big(\nabla e^{t \lL_\eps} - \Phi_\eps \nabla e^{t \lL_0}\big) u_\eps^{\#}(0) + \rR_\eps \big( \sS_\eps(\vec{u}) - \tT_\eps(\vec{u}) \big)\\
    &\phantom{1}- \int_{0}^{t} \int_{\T^2} \widetilde{\rR}_\eps(t-r,x,y) a_\eps(y) \Big( (\Lambda_\eps u)(r,y) - (\Lambda_\eps u)(t,x) \Big) {\rm d}y {\rm d}r\;,
    \end{split}
\end{equation*}
where $\Gamma \vec{u}$ is evaluated at the spacetime point $(t,x)$. We now show that there exists $T_\eps>0$ such that $\Gamma_{\eps}$ is a contraction from a bounded ball of radius $\kK$ in $\xX_{T_\eps}$ into itself. Again, we fix a large $T^*>0$ and consider time intervals $T_\eps < T^*$. All the proportionality constants below are allowed to depend on $T^*$, but are uniform in $T \in [0,T^*]$. 

By Lemmas~\ref{lem:initial},~\ref{lem:para_term} and~\ref{lem:I_eps_spacetime_norm} and Propositions~\ref{prop:S},~\ref{prop:T} and~\ref{prop:extra_1}, we see there exist $C_1, \theta_1>0$ such that
\begin{equation} \label{e:Gamma_1}
    \|\Gamma_{\eps}^{(1)} \vec{u}\|_{\cC_\fs^\alpha} \leq C_1 \Big( \mM_\eps + \m^{-(1-\alpha-\kappa)} \mM_\eps (1+\kK) + T^{\theta_1} (1 + \m^{2\alpha}) \big( \mM_\eps^6 + \kK^6 \big) \Big)\;,
\end{equation}
where we take $\cC_\fs^\alpha([0,T] \times \T^2)$-norm for $\Gamma_\eps^{(1)} \vec{u}$ on the left hand side. We now turn to $\Gamma_\eps^{(2)}$. For the initial data term, since $\eta > 1-\frac{3\alpha}{4}$, we have
\begin{equation*}
    \|\nabla e^{t \lL_0} u_\eps^{\#}\|_{L_T^{\infty,\eta} \cC_x^{\alpha/4}} \lesssim \sup_{t \in [0,T]} \Big( t^{\frac{\eta}{2}} \cdot t^{-\frac{1}{2}(1-\frac{3\alpha}{4})} \|u_\eps^{\#}(0)\|_{\cC^\alpha} \Big) \lesssim T^\theta \mM_\eps\;.
\end{equation*}
Using Propositions~\ref{prop:S} and~\ref{prop:T} for the term $\nabla \iI_0 (\sS_\eps - \tT_\eps)$ and Proposition~\ref{prop:second_main} for the last term on the right hand side of the equation for $\Gamma_\eps^{(2)} \vec{u}$, we have
\begin{equation} \label{e:Gamma_2}
    \|\Gamma_\eps^{(2)} \vec{u}\|_{L_T^{\infty,\eta} \cC_x^{1-\alpha}} \leq C_2 T^{\theta_2} (1 +\m^{2\alpha}) \big( \mM_\eps + \kK \big)^6
\end{equation}
for some $C_2, \theta_2 > 0$. Finally for $\Gamma_\eps^{(3)}$, this time using Lemmas~\ref{lem:initial_error},~\ref{lem:remainder_convolution_negative} and Proposition~\ref{prop: widetilde R} together with Propositions~\ref{prop:S} and~\ref{prop:T}, we get
\begin{equation} \label{e:Gamma_3}
    \|\Gamma_\eps^{(3)} \vec{u}\|_{L_T^{\infty,\eta} \cC_x^{8\kappa}} \leq C_3 (\eps T)^{\theta_3} (1 + \m^{2\alpha}) (\mM_\eps +\kK)^6
\end{equation}
for some $C_3, \theta_3 > 0$. Combining \eqref{e:Gamma_1}, \eqref{e:Gamma_2} and~\eqref{e:Gamma_3}, we deduce that there exist $C_0, \theta > 0$ such that
\begin{equation*}
    \|\Gamma_\eps \vec{u}\|_{\xX_T} \leq C_0 \Big( \mM_\eps + \m^{-(1-\alpha-\kappa)} \mM_\eps (1 + \kK) + T^\theta (1 + \m^{2\alpha}) (\mM_\eps + \kK)^6 \Big)\;.
\end{equation*}
Hence, if we take $\kK > 10 C_0 \mM_\eps$, and $\m$ sufficiently large so that
\begin{equation*}
    C_0 \cdot \m^{-(1-\alpha-\kappa)} \mM_\eps (1 +\kK) \leq \frac{\kK}{10}\;,
\end{equation*}
and $T>0$ sufficiently small depending on $\mM_\eps$ and the previously chosen $\kK$ and $\m$ (and hence depending on $\mM_\eps$ only) so that
\begin{equation*}
    T^\theta (1 + \m^{2\alpha}) (\mM_\eps + \kK)^6 \leq \frac{\kK}{10}\;,
\end{equation*}
then $\Gamma_\eps$ maps $\kK \subset \xX_T$ into itself. That $\Gamma_{\eps} $ is a contraction from the above space into itself follows from exactly the same arguments and the local Lipschitzness of $g$ and $\Lambda_{\eps} $. Hence, for each $\eps \in \N^{-1}$, there exists a unique triple $\vec{u}_\eps = (u_{\eps} , v_{\eps} , w_{\eps} ) \in \xX_{T_{\eps} }$ satisfying the system \eqref{e:fixed_pt_system} (and \eqref{e:fixed_pt_system_limit} for $\eps=0$). It also transpires from the arguments that $T_{\eps} $ depends on $\mM_{\eps}$ only. In particular, the local existence time $T_\eps$ is independent of $\eps$ if $\mM_{\eps}$ is uniformly bounded in $\eps$. 

We now turn to the second part of the theorem, namely the convergence of $\vec{u}_\eps$ up to time $T_0$ for which the limiting solution $\vec{u}_0$ is defined. It essentially follows from the same arguments as above, with an additional ingredient on dealing with terms involving the difference $\tT_\eps (\vec{u}_\eps) - \tT_0 (\vec{u}_0)$. The assumption that the limiting solution $u_0 \in L_T^\infty \cC_x^{\alpha'}$ for some $\alpha'>\alpha$ is just to ensure that convergence takes place in the same space as the existence (without weakening the regularity or adding weights), thus simplifying notations and technical details. Let
\begin{equation*}
    \mM:= 1 + \|\vec{u}_0\|_{\xX_{T_0}} + \|u_0\|_{L_{T_0}^{\infty} \cC_x^{\alpha'}} + \|\Upsilon_0\|_{\yY}.
\end{equation*}
We need to control the difference between $\vec{u}_\eps = (u_\eps, v_\eps, w_\eps)$ and $\vec{u}_0 = (u_0, v_0, 0)$. By assumption, we have
\begin{equation*}
    \|u_\eps(0)\|_{\cC^\alpha} + \|\Upsilon_\eps\|_\yY \leq \mM
\end{equation*}
for all sufficiently small $\eps>0$. Hence, there exists $\tau'>0$ independent of $\eps$ such that $\vec{u}_\eps \in \xX_{\tau'}$ for all $\eps$ and there exists $C = C(\mM)$ also independent of $\eps$ such that
\begin{equation*}
    \|\vec{u}_\eps\|_{\xX_{\tau'}} \leq C(\mM)\;.
\end{equation*}
We now give details to control the $u_\eps - u_0$ component of $\vec{u}_\eps - \vec{u}_0$. Since $\vec{u}_\eps$ and $\vec{u}_0$ are solutions to the system, we have
\begin{equation*}
    \begin{split}
    u_\eps - u_0 = &\phantom{1} \big(e^{\tau \lL_\eps} u_\eps^{\#}(0) - e^{\tau \lL_0} u_0^{\#}(0) \big) + \Big( g(u_\eps) \prec \P_\m^\perp X_\eps - g(u_0) \prec \P_\m^\perp X_0 \Big)\\
    &+ \iI_\eps \Big( \div \big( a_\eps \Lambda_\eps (u_\eps) \big) + \sS_\eps (\vec{u}_\eps) - \tT_\eps (\vec{u}_\eps) \Big)\\
    &- \iI_0 \Big( \div \big( \bar{a} \Lambda_0 (u_0) \big) + \sS_0(\vec{u}_0) - \tT_0(\vec{u}_0) \Big)\;,
    \end{split}
\end{equation*}
where the time is $\tau \in (0, \tau')$ to be specified later. For the initial data term, by Lemmas~\ref{lem:initial} and~\ref{lem:initial_convergence} and that $\alpha' > \alpha$, we have
\begin{equation} \label{e:converge_u_initial}
    \begin{split}
    &\phantom{111}\left\| e^{\tau \lL_\eps} u_\eps^{\#}(0) - e^{\tau \lL_0} u_0^{\#}(0) \right\|_{\cC^\alpha}\\
    &\leq \big\| e^{\tau \lL_\eps} \big( u_\eps^{\#}(0) - u_0^{\#}(0) \big) \big\|_{\cC^\alpha} + \big\| \big( e^{\tau \lL_\eps} - e^{\tau \lL_0} \big) u_0^{\#}(0) \big\|_{\cC^\alpha}\\
    &\lesssim \|u_\eps^{\#}(0) - u_0^{\#}(0)\|_{\cC^\alpha} + \eps^\theta \|u_0^{\#}(0)\|_{\cC^{\alpha'}}\;.
    \end{split}
\end{equation}
By Lemma~\ref{lem:para_term}, we have
\begin{equation} \label{e:converge_u_para}
    \begin{split}
    \big\|\big( g(u_\eps) - g(u_0) \big) \prec \P_\m^\perp X_\eps \big\|_{\cC_\fs^\alpha} &\lesssim_{\mM} \, \m^{-(1-\alpha-\kappa)} \, \|u_\eps - u_0\|_{\cC_\fs^\alpha}\;,\\
    \big\| g(u_0) \prec \P_\m^\perp \big( X_\eps - X_0 \big) \big\|_{\cC_\fs^\alpha} &\lesssim_\mM \, \m^{-(1-\alpha-\kappa)} \, \|\Upsilon_\eps - \Upsilon_0\|_\yY\;,
    \end{split}
\end{equation}
where we have omitted the explicit dependence on $\|u_\eps\|_{\cC_\fs^\alpha}$ and $\|u_0\|_{\cC_\fs^\alpha}$ but moved them into the dependence of the proportionality constant on $\mM$. By Proposition~\ref{prop:extra_1}, we have
\begin{equation} \label{e:converge_u_extra_1}
    \begin{split}
    &\phantom{111}\left\| \iI_\eps \Big( \div \big( a_\eps \Lambda_\eps (u_\eps) \big) \Big) - \iI_0 \Big( \div \big( \bar{a} \Lambda_0 (u_0) \big) \Big) \right\|_{\cC_\fs^\alpha}\\
    &\lesssim_\mM \, \tau^\theta (1 + \m^\alpha) \, \big( \eps^\theta + \|\Upsilon_\eps - \Upsilon_0\|_\yY + \|u_\eps - u_0\|_{\cC_\fs^\alpha} \big)\;.
    \end{split}
\end{equation}
By Proposition~\ref{prop:S}, we have
\begin{equation} \label{e:converge_u_S}
    \begin{split}
    &\phantom{111}\|\iI_\eps \big( \sS_\eps (\vec{u}_\eps) \big) - \iI_0 \big( \sS_0 (\vec{u}_0) \big)\|_{\cC_\fs^\alpha}\\
    &\lesssim_\mM \, \tau^\theta (1 + \m^\alpha) \big(\eps^\theta + \|\Upsilon_\eps - \Upsilon_0\|_\yY + \|\vec{u}_\eps - \vec{u}_0\|_{\xX_\tau} \big)\;.
    \end{split}
\end{equation}
We finally turn to $\iI_\eps (\tT_\eps) - \iI_0 (\tT_0)$. By Proposition~\ref{prop:T}, we have
\begin{equation} \label{e:converge_u_T_1}
    \big\| (\iI_\eps - \iI_0) \big( \tT_\eps (\vec{u}_\eps) \big) \big\|_{\cC_\fs^\alpha} \lesssim_\mM \, (\eps \tau)^\theta (1 + \m^{2\alpha})\;.
\end{equation}
As for $\iI_0 (\tT_\eps - \tT_0)$, we first note that since $\vec{u}_\eps$ and $\vec{u}_0$ are solutions to the fixed point system \eqref{e:fixed_pt_system}, we have
\begin{equation*}
    \tT (\vec{u}_\eps) = \d_t \big( g(u_\eps) \prec \P_\m^\perp X_\eps \big)\;, \qquad \tT (\vec{u}_0) = \d_t \big( g(u_0) \prec \P_\m^\perp X_0 \big)\;.
\end{equation*}
Hence, by Lemma~\ref{lem:I0_time_derivative} and interpolation, we have
\begin{equation} \label{e:converge_u_T_2}
    \big\| \iI_0 \big( \tT_\eps (\vec{u}_\eps) - \tT_\eps (\vec{u}_0) \big) \big\|_{L_{\tau}^\infty \cC_x^{\alpha}} \lesssim \tau^\theta \big\| g(u_\eps) \prec \P_\m^\perp X_\eps - g(u_0) \prec \P_\m^\perp X_0 \big\|_{\cC_\fs^\alpha}\;.
\end{equation}
Combining \eqref{e:converge_u_T_1} and \eqref{e:converge_u_T_2} and interpolating with the uniform bound, we get
\begin{equation} \label{e:converge_u_T}
    \begin{split}
    &\phantom{111}\big\| \iI_\eps \big( \tT_\eps (\vec{u}_\eps) \big) - \iI_0 \big( \tT_0 (\vec{u}_0) \big) \big\|_{\cC_\fs^\alpha}\\
    &\lesssim_\mM \, \tau^\theta (1 + \m^{2\alpha}) \cdot \big( \eps^\theta + \|\Upsilon_\eps - \Upsilon_0\|_\yY + \|\vec{u}_\eps - \vec{u}_0\|_{\xX_\tau} \big)\;.
    \end{split}
\end{equation}
Combining \eqref{e:converge_u_initial}, \eqref{e:converge_u_para}, \eqref{e:converge_u_extra_1}, \eqref{e:converge_u_S} and \eqref{e:converge_u_T}, and using that
\begin{equation*}
    \|u_\eps^{\#}(0) - u_0^{\#}(0)\|_{\cC^\alpha} \lesssim \mM \cdot \|u_\eps(0) - u_0(0)\|_{\cC^\alpha} + \|\Upsilon_\eps - \Upsilon_0\|_{\yY}\;,
\end{equation*}
we get
\begin{equation} \label{e:converge_u_bootstrap}
    \begin{split}
    \|u_\eps - &u_0\|_{\cC_\fs^\alpha([0,\tau] \times \T^2)} \lesssim_\mM  \big(\eps^\theta + \|u_\eps(0) - u_0(0)\|_{\cC^\alpha} + \|\Upsilon_\eps - \Upsilon_0\|_\yY \big)\\
    &+ \big( \tau^\theta (1 + \m^{2\alpha}) + \m^{-(1-\alpha-2\kappa)} \big) \cdot \big(\eps^\theta + \|\vec{u}_\eps - \vec{u}_0\|_{\xX_\tau} \big)\;.
    \end{split}
\end{equation}
One can get the same bounds for $\|v_\eps - v_0\|_{\xX_\tau^{(2)}}$ and $\|w_\eps\|_{\xX_\tau^{(3)}}$. Hence, we deduce there exist $C_0 = C_0(\mM)$ and $\theta>0$ such that
\begin{equation} \label{e:converge_vecu_bootstrap}
    \begin{split}
    \|\vec{u}_\eps - &\vec{u}_0\|_{\xX_\tau} \leq  C_0 \Big[ \big(\eps^\theta + \|u_\eps(0) - u_0(0)\|_{\cC^\alpha} + \|\Upsilon_\eps - \Upsilon_0\|_\yY \big)\\
    &+ \big( \tau^\theta (1 + \m^{2\alpha}) + \m^{-(1-\alpha-2\kappa)} \big) \cdot \big(\eps^\theta + \|\vec{u}_\eps - \vec{u}_0\|_{\xX_\tau} \big) \Big]\;.
    \end{split}
\end{equation}
One can choose $\m$ large enough depending on $\mM$ and then $\tau < \tau'$ sufficiently small (depending on $\mM$ and $\m$, and hence on $\mM$ only) so that one can absorb the $\|\vec{u}_\eps - \vec{u}_0\|_{\xX_\tau}$ term on the right hand side into the left, and obtain
\begin{equation*}
    \|\vec{u}_\eps - \vec{u}_0\|_{\xX_\tau} \leq C \Big[ \eps^\theta + \|u_\eps(0) - u_0(0)\|_{\cC^\alpha} + \|\Upsilon_\eps - \Upsilon_0\|_\yY \Big]
\end{equation*}
for some $C>0$ depending on $\mM$ (but independent of $\eps$) and $\theta>0$. Since such a time $\tau$ depends on $\mM$ only and that for sufficiently small $\eps$, one can keep the norm $\|\vec{u}_\eps\|_{\xX_\tau}$ not exceeding $\mM$. One can then iterate this procedure finitely many times until one reaches $T_0$. This completes the proof of Theorem~\ref{thm:fixed_pt_thm}.

\begin{rmk} \label{rmk:commutator_confirm}
    One can also check that the only properties of the para-products used are Bony's inequalities in Appendix~\ref{sec:para-products}. In particular, no commutators or estimates related to them have appeared. This is a byproduct of the integration by parts and ``completing products" operations as described in Section~\ref{sec:overview}. 
\end{rmk}

\section{Convergence of the flux}
\label{sec:flux_convergence}

We now prove that for the solution $\vec{u}_\eps$ and $\vec{u}_0$ obtained in Theorem~\ref{thm:fixed_pt_thm}, we also have the convergence of the flux $a_\eps \nabla u_\eps \rightarrow \bar{a} \nabla u_0$. We first state an additional lemma that is used in the flux convergence.

\begin{lem} \label{lem:aA}
Fix $T>0$. For every $\sigma>0$ and every sufficiently small $\kappa>0$, there exists $\theta>0$ depending on $\sigma$ and $\kappa$ such that
\begin{equation*}
    \big\| a_\eps \big( \id + \nabla \iI_\eps (\div a_\eps) \big) - \bar{a} \big\|_{L_{T}^{\infty,\sigma} \cC_x^{-\kappa}} \lesssim \eps^\theta\;.
\end{equation*}
\end{lem}
\begin{proof}
By \eqref{e:extra_0_error}, for every $\sigma>0$, there exists $\theta>0$ such that
\begin{equation*}
    \begin{split}
    \big\| a_\eps \big( \id + \nabla \iI_\eps (\div a_\eps) \big) - a_\eps \Phi_\eps \big\|_{L_x^\infty} &\lesssim \eps^\theta \cdot t^{-\frac{\sigma}{2}} + \|\nabla e^{t \lL_0} \big( \eps \chi(\cdot/\eps) \big)\|_{L_x^\infty}\\
    &\lesssim \eps^\theta \cdot t^{-\frac{\sigma}{2}}\;.
    \end{split}
\end{equation*}
Here, the second bound follows from Lemma~\ref{lem:initial} and that $\|\eps \chi(\cdot/\eps)\|_{\cC^{1-\sigma}} \lesssim \eps^\sigma \cdot t^{-\frac{\sigma}{2}}$. In addition, we have
\begin{equation*}
    \|a_\eps \Phi_\eps - \bar{a}\|_{\cC^{-\kappa}} \lesssim \eps^\theta\;.
\end{equation*}
The desired claim then follows. 
\end{proof}

Unlike in Theorem~\ref{thm:fixed_pt_thm}, it is more convenient to consider the norm weighted in time defined in \eqref{e:spacetime} to prove the convergence of the flux. We have the following theorem regarding the flux. 

\begin{thm} \label{thm:flux_convergence}
    Let $\vec{u}_\eps = (u_\eps, v_\eps, w_\eps)$ and $\vec{u}_0 = (u_0, v_0)$ be the solutions constructed in Theorem~\ref{thm:fixed_pt_thm}, and $T_0 > 0$ be the maximal existence time for $\vec{u}_0$. Then for every sufficiently small $\kappa>0$, there exists $\theta>0$ such that for $\sigma = 1-\alpha+2\kappa$ and $T<T_0$, we have
    \begin{equation} \label{e:flux_converge}
        \|a_\eps \nabla u_\eps - \bar{a} \nabla u_0\|_{L_T^{\infty,\sigma} \cC_x^{-\kappa}} \lesssim  \eps^\theta + \|u_\eps(0) - u_0(0)\|_{\cC^\alpha} + \|\Upsilon_\eps - \Upsilon_0\|_{\yY}\;.
    \end{equation}
    In other words, the flux converges in a time weighted version of the space in which the solution, the initial data and enhanced external terms in $\Upsilon_{\eps}$ converge. 
\end{thm}
\begin{proof}
    From the proof of Theorem~\ref{thm:fixed_pt_thm}, we have
    \begin{equation*}
        \nabla u_\eps = g(u_{\eps})\nabla X_{\eps} + \Phi_\eps v_\eps + w_\eps +(\id+ \nabla \iI_\eps \big( \div a_\eps \big)) \Lambda_\eps (u_\eps)\;.
    \end{equation*}
    Multiplying $a_\eps$ on both sides above, we get
    \begin{equation} \label{e:flux_eq_eps}
        a_\eps \nabla u_\eps = g(u_\eps) \fF_\eps + a_\eps \Phi_\eps v_\eps + a_\eps w_\eps + a_\eps \big( \id + \nabla \iI_\eps (\div a_\eps) \big) \Lambda_\eps (u_\eps)\;.
    \end{equation}
    Similarly, we have
    \begin{equation*}
        \bar{a} \nabla u_0 = g(u_0) \fF_0 + \bar{a} v_0 + \bar{a} \Lambda_{0}(u_0) \;.
    \end{equation*}
    The difference can then be written as
    \begin{equation*}
        \begin{split}
        a_\eps \nabla u_\eps - &\bar{a} \nabla u_0 = \big( g(u_\eps) - g(u_0) \big) \fF_\eps + g(u_0) \cdot (\fF_\eps - \fF_0) + a_\eps \Phi_\eps (v_\eps - v_0)\\
        &\big( a_\eps \Phi_\eps - \bar{a} \big) v_0 + a_\eps w_\eps + a_\eps \big( \id + \nabla \iI_\eps (\div a_\eps) \big) \big( \Lambda_\eps (u_\eps) - \Lambda_0 (u_0) \big)\\
        &+ \Big( a_\eps \big( \id + \nabla \iI_\eps (\div a_\eps) \big) - \bar{a} \Big) \Lambda_0 (u_0)\;.
        \end{split}
    \end{equation*}
    The desired bounds for the above terms follow from Theorem~\ref{thm:fixed_pt_thm} and Lemma~\ref{lem:aA}. Here, all proportionality constants depend on $\|(u_\eps, v_\eps, w_\eps)\|_{\xX_T}$ and $\|\Upsilon\|_{\yY}$, which we omit in notation for simplicity. This completes the proof of the theorem. 
\end{proof}

\section{Convergence of enhanced stochastic terms}
\label{sec:stochastic}

In this section, we will show convergences of the stochastic objects to their homogenised limits. The following is the main theorem.

\begin{thm} \label{thm:stochastic}
    Recall the stochastic objects $\Upsilon_\eps$ and $\Upsilon_0$ from \eqref{e:noises_eps} and \eqref{e:noises_limit}, and the norm $\yY$ from \eqref{e:noises_norm}. For every $\kappa>0$, there exists $\theta>0$ such that for every $p \in [1,+\infty)$, we have
    \begin{equation} \label{e:stochastic_convergence}
        \E \|\Upsilon_\eps - \Upsilon_0\|_{\yY}^p \lesssim \eps^{\theta p}\;.
    \end{equation}
    The proportionality constant depends on $\theta$, $\kappa$ and $p$, but is uniform in $\eps \in \N^{-1}$. 
\end{thm}

Theorem~\ref{thm:stochastic} verifies the assumptions on $\Upsilon_\eps$ and $\Upsilon_0$ in Theorems~\ref{thm:fixed_pt_thm} and \ref{thm:flux_convergence}. To prove Theorem~\ref{thm:stochastic}, the basic building block is the object $X_\eps = -\lL_\eps^{-1} \Pi_0^\perp \xi$. Let $G_\eps$ and $G_0$ denote the Green's functions associated to $-\lL_\eps$ and $ -\lL_0$. More precisely, for every $\eps \in [0,1]$, we have
\begin{equation} \label{e:Green_elliptic}
     ( -\lL_\eps^{-1} f ) (x) = \int_{\T^2} G_\eps (x,y) f(y) {\rm d}y
\end{equation}
for $f$ with $\int_{\T^2} f = 0$, and $\Pi_0 \big( G_\eps(x,\cdot) \big) = 0$ for every $x \in \T^2$. Let $\nabla_x G_\eps$ and $\nabla_y G_\eps$ denote the gradient of $G_\eps$ with respect to its first and second spatial variables respectively. We restrict outselves to $\T^2$ here. 

Define the vector valued kernel $R^{G}_\eps$ by
\begin{equation} \label{e:difference_first_order_G}
    R^{G}_\eps (x,y) := (\nabla_x G_\eps)(x,y) - \Phi_\eps (x) \, (\nabla_x G_0)(x,y)\;.
\end{equation}
Let $\rR^{G}_\eps$ be operation given by
\begin{equation} \label{e:operator_R_G}
    (\rR_\eps^{G} f)(x) :=  \int_{\T^2} R^{G}_\eps (x,y) f(y) {\rm d}y\;.
\end{equation}
We shall use the following bounds on $G_\eps$ and $\nabla_x G_\eps - \Phi_\eps \nabla_x G_0$.

\begin{prop}
$G_\eps$ and its derivative satisfy the pointwise bounds
\begin{equation} \label{e:Green_elliptic_bound}
    |G_\eps (x,z)| \lesssim 1+\left| \, \log |x-z| \, \right|\;, \qquad |\nabla_x G_\eps (x,z)| \lesssim \frac{1}{|x-z|}
\end{equation}
uniformly in $x,z \in \T^2$ and $\eps \in \N^{-1}$. Also, for every $\sigma \in (0,1)$, there exists $\theta>0$ such that we have the bound on the difference
\begin{equation} \label{e:Green_elliptic_difference}
    |G_\eps(x,z) - G_0(x,z)| \lesssim \frac{\eps^\theta}{|x-z|^{\sigma}}\;.
\end{equation}
Furthermore, for every $\sigma \in (0,1)$, there exists $\theta>0$ such that
\begin{equation*}
    \left| \nabla_x G_\eps (x,z) - \Phi_\eps (x) \nabla_x G_0(x,z) \right| \lesssim \frac{\eps^\theta}{|x-z|^{1+\sigma}}\;.
\end{equation*}
All proportionality constants are independent of $\eps \in \N^{-1}$. 
\end{prop}
\begin{proof}
The two bounds in \eqref{e:Green_elliptic_difference} follow from interior Lipschitz estimates and a duality method in \cite{FLinCompactness} and \cite{Shen2018}. The uniform continuity bound \eqref{e:Green_elliptic_difference} then follows from \eqref{e:Green_elliptic_bound}. The last estimate on $\nabla_x G_\eps - \Phi_\eps \nabla_x G_0$ can be obtained in the same way illustrated in Proposition~\ref{prop:Green_remainder_derivative}. 
\end{proof}

We will prove Theorem~\ref{thm:stochastic} by interpolating a uniform-in-$\eps$ bound in a slightly stronger space and quantitative convergence in a weaker space. We first state the uniform bound. 

\begin{prop} \label{prop:stochastic_uniform}
For every $\delta>0$ and $p \geq 1$, we have
\begin{equation} \label{e:stochastic_uniform}
    \E |\Pi_0 \xi|^p + \E \|X_\eps\|_{\cC^{1-\delta}}^p + \sum_{j=3}^{7} \E \|\Upsilon_\eps^{(j)}\|_{\cC^{-\delta}}^{p} \lesssim 1
\end{equation}
uniformly in $\eps \in \N^{-1}$. 
\end{prop}
\begin{proof}
Since all objects live in Wiener chaos up to order two, by hypercontractivity, it suffices to prove the statement for $p=2$. 

By the second bound in \eqref{e:Green_elliptic_bound}, we have for every $\delta>0$,
\begin{equation*}
    \left| \E \big( \nabla X_\eps (x) \, \nabla^{\sT} X_\eps(y) \big) \right| \lesssim \int_{\T^2} |(\nabla_1 G_\eps)(x,z)| \cdot |(\nabla_1 G_\eps)(y,z)| {\rm d}z \lesssim \frac{1}{|x-y|^\delta}\;,
\end{equation*}
where $\nabla_1$ denotes gradient with respect to the first variable. The bound for the $\|\Upsilon_\eps^{(j)}\|_{\cC^{-\delta}}$, for $j = 2, \dots, 6$, then follows from the above covariance inequality. 

As for $\|X_\eps\|_{\cC^{1-\delta}}$, we first interpolate the two bounds in \eqref{e:Green_elliptic_bound} to get
\begin{equation*}
    |G_\eps (x,z) - G_\eps(y,z)| \lesssim \frac{|x-y|^{1-\frac{\delta}{2}}}{|x-z|^{1-\frac{\delta}{4}}}\;.
\end{equation*}
We then have
\begin{equation*}
    \E |X_\eps(x) - X_\eps(y)|^2 \lesssim \int_{\T^2} |G_\eps(x,z) - G_\eps(y,z)|^2{\rm d}z \lesssim |x-y|^{2-\delta}\;.
\end{equation*}
The desired bound for $\E \|X_\eps\|_{\cC^{1-\delta}}$ then follows from Kolmogorov's continuity criterion. This completes the proof of the proposition. 
\end{proof}

The convergences to the desired homogenised limits are more subtle, as they require cancellations from the structured oscillations (such as the divergence-free structure of $a_\eps \Phi_\eps$, etc.). In order to simplify arguments, we make repeated use of integration by parts to utilise such structures, which will lead to convergences in weaker spaces (typically in $\cC_{x}^{-1-\kappa}$). But when interpolated with the boundedness in \eqref{e:stochastic_uniform} in stronger spaces, one obtains the convergences in the desired spaces. We will frequently use the two-scale expansion
\begin{equation} \label{e:expansion_nablaY}
    \nabla X_\eps(x) = \Phi_\eps(x) \nabla X_0(x) + \int_{\T^2} R^{G}_\eps(x,z) \, (\Pi_{0}^{\perp}\xi)(z) {\rm d}z
\end{equation}
for $\nabla X_\eps$. 

\subsection{The free field}
\label{sec:free_field}

\begin{prop} \label{prop:free_field}
For every $\kappa$, there exists $\theta>0$ such that, for every $\eps\in\N^{-1}$, $p\geq 1$,
\begin{equation*}
    \E \|X_\eps - X_0\|^{p}_{\cC_x^{1-\kappa}(\T^2)} \lesssim \eps ^{p\theta} \;
\end{equation*}
\end{prop}
\begin{proof}
We have the pointwise bound
\begin{equation*}
    \E |X_\eps(x) - X_0(x)|^2 = \int_{\T^2} \big( G_\eps(x,z) - G_0(x,z) \big)^{2} {\rm d}z \lesssim \eps^{2 \theta}
\end{equation*}
for some $\theta>0$. By Kolmogorov's continuity criterion and hypercontractivity, this gives
\begin{equation*}
    \E \|X_\eps - X_0\|_{L^\infty}^{p} \lesssim \eps^{p \theta}
\end{equation*}
for every $p \geq 1$. The claim then follows by interpolating the above $L^\infty$ bound and the uniform bound in \eqref{e:stochastic_uniform}. 
\end{proof}

\subsection{The linear flux and its variants}
\label{sec:flux}

In this section, we show the convergence of the linear flux $\fF_\eps = a_\eps \nabla X_\eps$ to its homogenised limit $\fF_0 =\bar{a} \nabla X_0$, as well as several variants of it. The main ingredient is the following lemma. 

\begin{lem} \label{lem:flux_stochastic_general}
    Let $H: \T^2 \rightarrow \R^{2 \times 2}$ be bounded, and write $H_\eps := H (\cdot / \eps)$. Then for every $\kappa>0$, there exists $\theta>0$ such that for every $p \in (1, +\infty)$ and $\eps \in \N^{-1}$, we have
    \begin{equation*}
        \E \big\| H_\eps \nabla X_\eps - \Pi_0 \big( H \Phi \big) \nabla X_0  \big\|_{\cC_x^{-\kappa}}^p \lesssim \eps^{p \theta} (1+\|H\|_{L^\infty})^{p}\;,
    \end{equation*}
    where
    \begin{equation*}
        \Pi_0 \big( H \Phi \big) = \int_{\T^2} H(x) \Phi(x) {\rm d}x = \int_{\T^2} H(x) \big( \id + (\nabla \chi)(x) \big) {\rm d}x
    \end{equation*}
    is the average of $H \Phi$ on the torus. 
\end{lem}
\begin{proof}
Since $\Pi_0 (H_\eps \Phi_\eps) = \Pi_0 (H \Phi)$ for every $\eps \in \N^{-1}$, we have
\begin{equation*}
    H_\eps \nabla X_\eps - \Pi_0 (H \Phi) \nabla X_0 = \Pi_0^\perp \big( H_\eps \Phi_\eps \big) \nabla X_0 + H_{\eps}\rR_\eps^G(\Pi_{0}^{\perp}\xi)\;,
\end{equation*}
where we recall from \eqref{e:operator_R_G} that $\rR_\eps^G$ denotes the convolution with the kernel $R_\eps^G$. The desired bound for the remainder $H_{\eps}\rR_\eps^G(\Pi_{0}^{\perp}\xi)$ follows directly from the pointwise bound for the kernel $R_\eps^G$. As for the first part on the right hand side, testing against $\phi \in \cC^\infty(\T^2; \R^2)$ and integrating by parts gives
\begin{equation*}
    \bigscal{\Pi_0^\perp \big( H_\eps \Phi_\eps \big) \nabla X_0, \; \phi} =  \Bigscal{\Pi_{0}^{\perp}\xi, \;  (\lL_{0}^*)^{-1}\div \big( \Pi_0^\perp (H_\eps \Phi_\eps)^{\sT} \phi \big)}\;.
\end{equation*}
Hence, by Lemma~\ref{lem:oscillation}, we have
\begin{equation*}
    \E \big| \bigscal{\Pi_0^\perp \big( H_\eps \Phi_\eps \big) \nabla X_0, \; \phi} \big|^2 \lesssim \big\| \Pi_0^\perp \big( H_\eps \Phi_\eps \big)^{\sT} \varphi \big\|_{H^{-1}}^{2} \lesssim \eps^{2 \theta} \big(1 + \|H \Phi\|_{L^\infty}\big)^2 \|\varphi\|_{\cC^\kappa}^2\;.
\end{equation*}
The claim then follows by interpolating the above bound with the known uniform-in-$\eps$ bound. 
\end{proof}

Applying the above lemma with $H_\eps = a_\eps$, $\Phi_\eps^{\sT} a_\eps$, $a_\eps^{\sT}$ and $\Phi_\eps^{\sT} a_\eps^{\sT}$ gives the desired bounds for the flux $\fF_\eps$ and its variants $\Phi_\eps^{\sT} \fF_\eps$, $a_\eps^{\sT} \nabla X_\eps$ and $\Phi_\eps^{\sT} a_\eps^{\sT} \nabla X_\eps$.

\begin{cor} \label{cor:fluxes}
There exists $\theta>0$ such that for $j=3,4,5,6$ and every $p \geq 1$, we have
\begin{equation*}
    \E \| \Upsilon_\eps^{(j)} - \Upsilon_0^{(j)} \|_{\cC^{-\kappa}}^{p} \lesssim \eps^{p \theta}\;.
\end{equation*}
The proportionality constant is independent of $\eps \in \N^{-1}$. 
\end{cor}

\subsection{The Wick ordered term}
\label{sec:Wick}

\begin{prop}
    For every $\kappa>0$, there exists $\theta>0$ such that for every $p\in(1,+\infty)$, 
    \begin{equation*}
        \E||\nabla^{\sT} X_\eps \diamond \fF_{\eps}-\nabla^{\sT} X_0 \diamond \fF_{0}||^p_{\cC_{x}^{-\kappa}}\lesssim \eps^{p\theta}.
    \end{equation*}
\end{prop}
\begin{proof}
    Similarly as before, thanks to the uniform-in-$\eps$ bound in Proposition~\ref{prop:stochastic_uniform}, it suffices to show that the convergence of $\nabla^{\sT} X_\eps \diamond \fF_\eps$ to $\nabla^{\sT} X_0 \diamond \fF_0$ in a weak space. Write
    \begin{equation*}
        \nabla^{\sT} X_{\eps} \diamond \fF_{\eps} = \div(X_{\eps}\diamond \fF_{\eps})+ X_{\eps}\diamond \Pi_{0}^{\perp} \xi = \div (X_\eps \fF_\eps) - \div \E(X_\eps \fF_\eps) + X_{\eps}\diamond \Pi_{0}^{\perp} \xi\;.
    \end{equation*}
    For the first term, we have the deterministic bound
    \begin{equation*}
        \begin{split}
        &\phantom{111}\|\div (X_\eps  \fF_\eps) - \div (X_0  \fF_0)\|_{\cC^{-1-\kappa}} \\ &\lesssim \|(X_\eps - X_0)  \fF_\eps\|_{\cC^{-\kappa}} + \|X_0 (\fF_\eps - \fF_0)\|_{\cC^{-\kappa}}\\
        &\lesssim \|X_\eps - X_0\|_{\cC^{2\kappa}} \|\fF_\eps\|_{\cC^{-\kappa}} + \|X_0\|_{\cC^{2\kappa}} \|\fF_\eps - \fF_0\|_{\cC^{-\kappa}}\;.
        \end{split}
    \end{equation*}
    The desired bound for this term then follows from Proposition~\ref{prop:free_field} and Corollary~\ref{cor:fluxes} for $\Upsilon_\eps^{(2)} = \fF_\eps$. For the second term with expectation, we formally have
    \begin{equation*}
        \E (X_\eps \fF_\eps) = \E (X_\eps a_\eps X_\eps) = \frac{1}{2} a_\eps \nabla \E (X_\eps^2)\;.
    \end{equation*}
    Since
    \begin{equation*}
        \E |X_\eps^2(x)| = \int_{\T^2} G_\eps^2 (x,y) {\rm d} y
    \end{equation*}
    does not depend on $x$, the second term is identically $0$. This can be justified rigorously via smooth approximation. 
    
    For the third term, we have
    \begin{equation*}
        (X_\eps - X_0) \diamond \Pi_0^\perp \xi = -\Pi_0 \xi \cdot (X_\eps - X_0) + (X_\eps - X_0) \diamond \xi\;,
    \end{equation*}
    where we used that $\E(\Pi_0 \xi \cdot X_\eps  ) = 0$ for $\eps \in \N^{-1}$. It suffices to consider the second term on the right hand side above. Testing against $\phi \in \cC^\infty (\T^2; \R)$ on both sides, we get
    \begin{equation*}
        \begin{split}
        \E &\big| \bigscal{(X_\eps - X_0) \diamond \xi, \; \phi} \big|^2 = \int_{\T^2} \E \big| X_\eps(x) - X_0(x) \big|^2 \cdot |\phi^2(x)| \, {\rm d}x\\
        &+ \int_{(\T^2)^2}  \big( G_\eps(x,y) - G_0(x,y) \big) \, \big( G_\eps(y,x) - G_0(y,x) \big) \, \phi(x)\, \phi(y) \, {\rm d}x {\rm d}y\;.
        \end{split}
    \end{equation*}
    By the Green's function estimate \eqref{e:Green_elliptic_difference} and Proposition~\ref{prop:free_field}, we get the bound
    \begin{equation*}
        \E \big| \bigscal{(X_\eps - X_0) \diamond \xi, \; \phi} \big|^2 \lesssim \eps^{2\theta} \|\phi\|_{L^2}^2
    \end{equation*}
    for some $\theta>0$. By scaling and hypercontractivity, this implies there exists $\theta>0$ such that
    \begin{equation*}
        \E \big\| \bigscal{(X_\eps - X_0) \diamond \xi, \; \phi} \big\|_{\cC^{-1-\kappa}}^p \lesssim \eps^{p \theta}
    \end{equation*}
    for every $p \geq 1$. The conclusion then follows. 
    \end{proof}

\appendix

\section{Bounds on Green's functions}

In this appendix, we give asymptotic expansions and characterise the behaviours of parabolic Green's function up to a level that is needed in the main text. The statements here mainly follow \cite{GengShen2020, GengShen2015,Geng2020}, which are based on methods in the early works \cite{FLinCompactness,FLinCompactness1989} on elliptic Green's functions. 

Let $K_\eps$ be the fundamental solution of $\partial_t-\lL_\eps$ in $\R^+ \times \R^d$. The following lemma follows from estimates for $K_{\eps}$ in \cite[Theorems~1.1, 1.2 and 1.3]{GengShen2020}.

\begin{lem}
    For the fundamental solution $K_\eps$,  we have 
    \begin{align*}
        | K_{\eps}(t,x,y)- K_{0}(t,x,y)|\lesssim \eps t^{-\frac{d+1}{2}}e^{-\frac{c|x-y|^2}{t}}\;.
    \end{align*}
    Its first order derivative satisfies the bound
    \begin{align*}
        |\nabla_{x} K_{\eps}(t,x,y)-\Phi_{\eps}(x)\nabla_{x} K_{0}(t,x,y)|\lesssim \eps t^{-\frac{d+2}{2}}e^{-\frac{c|x-y|^2}{t}}\log(\frac{\sqrt{t}}{\eps}+2)\;.
    \end{align*}
    Finally, its second order cross derivative satisfies
    \begin{align*}
        |\nabla_{x}\nabla_{y} K_{\eps}(t,x,y)-\Phi_{\eps}(x)\nabla_{x,y}^{2} K_{0}(t,x,y)(\Phi_{\eps}^{*})^{\mathsf{T}}(y)|\lesssim \eps t^{-\frac{d+3}{2}} e^{-\frac{c|x-y|^2}{t}}\log(\frac{\sqrt{t}}{\eps}+2)\;.
    \end{align*}
    All the bounds are uniform in $\eps\in (0,1)$, and $T>0$, $ t\in (0,T]$, $x,y \in \R^{d}$.
\end{lem}

Let $Q_\eps$ be the Green's function of $\partial_t-\lL_\eps$ in $\R^+ \times \T^d$, for $\eps\in \N^{-1}$. Then we have
\begin{align}\label{eq: relation between K eps and Q eps}
    Q_{\eps}(t,x,y)=\sum_{k\in\Z^d}K_{\eps}(t,x,y+k)\;.
\end{align}
By \cite[(1-6)]{GengShen2020}, essentially based on \cite{Byun2007,Cho2008,GengShen2015}, we have
\begin{align*}
    |K_\eps(t,x,y)|\lesssim t^{-\frac{d}{2}} e^{-\frac{c|x-y|^2}{t}}\;.
\end{align*}
 It follows that the sum in \eqref{eq: relation between K eps and Q eps} converges.  We first give some standard Green's function estimates for $Q_\eps$. 

\begin{prop} \label{prop:Green_point}
    For the Green's function $Q_\eps$, we have that there exists $c>0$ such that
    \begin{equation*}
        |Q_\eps (t,x,y)| \lesssim t^{-\frac{d}{2}} e^{-\frac{c |x-y|^2}{t}}\;.
    \end{equation*}
    Its first order derivatives satisfy the bounds
    \begin{equation*}
        |(\nabla_x Q_\eps)(t,x,y)| +|(\nabla_y Q_\eps)(t,x,y)| \, \lesssim t^{-\frac{d+1}{2}} e^{-\frac{c |x-y|^2}{t}}\;.
    \end{equation*}
    Finally, its second order cross derivative satisfies
    \begin{equation*}
        |(\nabla_{x,y}^2 Q_\eps)(t,x,y)| \lesssim t^{-\frac{d+2}{2}} e^{-\frac{c |x-y|^2}{t}}\;.
    \end{equation*}
    All bounds are uniform in $\eps \in \N^{-1}$, $x, y \in \T^d$ and $t > 0$. 
\end{prop}
\begin{proof}
    The estimates above follow from \cite[(1-6)]{GengShen2020}, \cite[Theorem 2.7]{GengShen2020} and that $Q_\eps$ is the periodisation of the fundamental solution $K_{\eps}$ in $\R^+ \times \R^d$. 
\end{proof}

Recall from \eqref{e:difference_first_order} that the kernels $R_\eps$ and $R_\eps^*$ are given by
\begin{equation*}
    \begin{split}
    R_\eps(t,x,y) := (\nabla_x Q_\eps)(t,x,y) - \Phi_\eps(x) \, (\nabla_x Q_0)(t,x,y)\;,\\
    R_\eps^*(t,x,y) := (\nabla_y Q_\eps)(t,x,y) - \Phi_\eps^*(y)  (\nabla_y Q_0)(t,x,y)\;, 
    \end{split}
\end{equation*}
and from \eqref{e:difference_second_order} that
\begin{equation*}
    \widetilde{R}_\eps(t,x,y) := (\nabla_{x,y}^{2} Q_\eps)(t,x,y) - \Phi_\eps(x) \, (\nabla_{x,y}^{2} Q_0)(t,x,y) \, \big( \Phi_\eps^*(y) \big)^{\sT}\;.
\end{equation*}
By the formula and estimates for $K_{\eps}$ in \cite[Theorem 1.1, 1.2 and 1.3]{GengShen2020}, we have similar estimates for $Q_{\eps}$ as in \cite[Theorem 1.1, 1.2 and 1.3]{GengShen2020}. We therefore have the following bounds.

\begin{prop} \label{prop:Green_Gaussian}
    There exist $c, C>0$ such that the following bounds are true uniformly over $x,y \in \T^d$, $t>0$ and $\eps \in \N^{-1}$:
    \begin{equation*}
        \big| Q_\eps(t,x,y) - Q_0(t,x-y) \big| \lesssim \eps \, t^{-\frac{d+1}{2}} e^{-\frac{c |x-y|^2}{t}}\;,
    \end{equation*}
    and
    \begin{equation*}
        | R_\eps(t,x,y)|+|R_\eps^*(t,x,y)| \lesssim \eps \, t^{-\frac{d+2}{2}} e^{-\frac{c |x-y|^2}{t}} \log \Big(2 + \frac{\sqrt{t}}{\eps} \Big)\;,
    \end{equation*}
    and, finally,
    \begin{equation*}
        \big| \widetilde{R}_\eps(t,x,y) \big| \lesssim \eps \, t^{-\frac{d+3}{2}} e^{-\frac{c |x-y|^2}{t}} \log \Big(2 + \frac{\sqrt{t}}{\eps} \Big)\;.
    \end{equation*}
\end{prop}
\begin{proof}
    All theses bounds follow from \cite[Theorems~1.1, 1.2 and 1.3]{GengShen2020} and that $Q_\eps$ is the periodisation of the fundamental solution in $\R^+ \times \R^d$. 
\end{proof}

\begin{prop} \label{prop:Green_remainder_point}
    For every $\sigma_2 \in (0,1)$, there exists $\sigma_1 > 0$ such that
    \begin{equation} \label{e:Green_remainder_point_1}
        \big| R_\eps(t,x,y) \big| \lesssim \frac{\eps^{\sigma_1}}{(\sqrt{t} + |x-y|)^{d+1+\sigma_2}}\;,
    \end{equation}
    and
    \begin{equation} \label{e:Green_remainder_point_2}
        \big| \widetilde{R}_\eps(t,x,y) \big| \lesssim \frac{\eps^{\sigma_1}}{(\sqrt{t} + |x-y|)^{d+2+\sigma_2}}\;,
    \end{equation}
    uniformly over $\eps \in \N^{-1}$, $t>0$ and $x,y \in \T^d$. The kernel $R_\eps^*$ satisfies the same bound as $R_\eps$ in \eqref{e:Green_remainder_point_1}. 
\end{prop}
\begin{proof}
    The bounds~\eqref{e:Green_remainder_point_1} and~\eqref{e:Green_remainder_point_2} follow from interpolating Proposition~\ref{prop:Green_Gaussian} and the pointwise bounds
    \begin{equation*}
    |R_\eps (t,x,y)| \lesssim \frac{1}{(\sqrt{t} + |x-y|)^{d+1}}\;, \quad |\widetilde{R}_\eps(t,x,y)| \lesssim \frac{1}{(\sqrt{t} + |x-y|)^{d+2}}
    \end{equation*}
    respectively. The bound for $R_\eps^*$ is exactly the same as $R_\eps$. 
\end{proof}

In the final proposition of this section, we use the following notation.  For $p\in [1,\infty)$, a bounded domain $\dD\subset \R^d$ and an integrable function $f\in L^{p}(\dD)$, we define 
\begin{equation*}
    \|f\|_{\underline{L}^{p}(\dD)}: = \Big(\frac{1}{|\dD|} \int_{ \dD } |f( x )|^p {\rm d}x\Big)^{\frac{1}{p}}\;,
\end{equation*}
where $|\dD|$ is the Lebesgue measure of $\dD$. For $z = (t,x)$ and $\rho>0$, define $\mathcal{D}_\rho(z) := (t-\rho^2,t) \times B(x,\rho)$. We have the following pointwise bounds on derivatives of $R_\eps$ and $\widetilde{R}_\eps$.

\begin{prop} \label{prop:Green_remainder_derivative}
    For $a\in C^{3}(\T^d)$, there exist $\beta_1, \beta_2, C> 0$ such that for every $\eps\in \N^{-1}$, $t\in (0,T)$, $x,y\in \T^{d}$, we have
    \begin{equation*}
        |\nabla_x R_\eps (t,x,y)| + |\nabla_y R_\eps(t,x,y)| \leq  \frac{C \eps^{- \beta_1}}{(\sqrt{t} + |x-y|)^{d+2+\beta_2}}\;,
    \end{equation*}
    and
    \begin{equation*}
        |\nabla_x \widetilde{R}_\eps (t,x,y)| + |\nabla_y \widetilde{R}_\eps(t,x,y)| \leq  \frac{ C \eps^{-\beta_1}}{(\sqrt{t} + |x-y|)^{d+3+\beta_2}}\;.
    \end{equation*}
    Moreover, similar estimates also hold for $R_{\eps}^{*}$ and $\tilde{R}^{*}_{\eps}$. 
\end{prop}
\begin{proof}
The bound for $\nabla_y R_\eps$ follows directly from the bound on cross derivatives $\nabla_{x,y}^2 Q_\eps$ in Proposition~\ref{prop:Green_point}. So we focus on $\nabla_x R_\eps$. By expression of $R_\eps$, for $\ell \in \{1, \dots, d\}$, we have
\begin{equation} \label{e:R_derivative_expression}
    \d_{x_\ell} R_\eps = \d_{x_\ell} \nabla_x Q_\eps - \d_{x_\ell} \Phi_\eps \cdot \nabla_x Q_0 - \Phi_\eps \cdot \d_{x_\ell} \nabla_x Q_0\;.
\end{equation}
It is standard that
\begin{equation*}
    \|\d_{x_\ell} \Phi_\eps\|_{L^\infty} \lesssim \frac{1}{\eps}\;, \quad \big|(\d_{x_\ell} \nabla_x Q_0)(t,x,y)| \lesssim \big( \sqrt{t} + |x-y| \big)^{-d-2}\;.
\end{equation*}
It then remains to control $\d_{x_\ell} \nabla_x Q_\eps$. Fix $z_0=(t_0,x_0)$ and $w_0=(\tau_0,y_0)$ in $(0,T) \times \T^d$. Define
\begin{equation*}
    \rho := \frac{1}{8} |z_0 - w_0|_\fs = \frac{1}{8} \big( \sqrt{|t_0 - \tau_0|} + |x_0 - y_0| \big)\;.
\end{equation*}
For convenience, we write $Q_\eps$ with parabolic coordinates in the sense that
\begin{equation*}
    Q_{\eps}(t,x;\tau,y) := Q_{\eps}(t-\tau,x,y)\;.
\end{equation*}
By the Lipschitz estimates in \cite[Theorem~1.2]{GengShen2015}, for $p>d+2$, we have
\begin{equation} \label{e:Q_derivative_C1}
    \begin{split}
     \|\nabla_x Q_{\eps}(\,\cdot\,;\,w_0\,)\|_{\cC^{1}_{\fs}(\mathcal{D}_{\rho/{2}}(z_0))} \lesssim &\phantom{1}\rho^{-1}\|\nabla_{x}Q_{\eps}(\,\cdot\,;\,w_0\,)\|_{\underline{L}^2(\dD_\rho(z_0))}\\
     &+\| (\partial_{t}-\mathcal{L}_{\eps,x})(\nabla_{x}Q_{\eps})(\,\cdot\,;\, w_0\,)\|_{\underline{L}^{p}(\mathcal{D}_\rho(z_0))}\;.
     \end{split}
\end{equation}
By Proposition \ref{prop:Green_point}, the first term on the right hand side above satisfied the desired bound. We focus on the second term. Since $Q_\eps$ satisfies
\begin{equation*}
    \d_{t}Q_{\eps}(\cdot ;\,w_0\,) = \lL_{\eps,x}Q_{\eps}(\cdot ;\,w_0\,)
\end{equation*}
in $\dD_{2 \rho}(z_0)$, differentiating $x_\ell$ variable on both sides above gives
\begin{equation*}\begin{split}
    (\partial_{t}-\mathcal{L}_{\eps,x})\d_{x_\ell} Q_{\eps}(\,\cdot\,;\,w_0\,)\,=\,\div_x\big((\d_{\ell}a_{\eps})\nabla_{x} Q_{\eps}(\,\cdot\,;\,w_0\,)\,\big)\;.
    \end{split}
\end{equation*}
By \cite[Theorem~1.4]{GengShen2015} in $\mathcal{D}_{2 \rho}(z_0)$ and cut-off functions, we have
\begin{equation} \label{e:Q_Hessian_Lp}
    \begin{split}
    \|\nabla_x \d_{x_\ell} Q_{\eps}(\,\cdot\,;\,w_0\,)\|_{\underline{L}^{p}(\mathcal{D}_\rho(z_0))} \lesssim &\phantom{1}\rho^{-1}\|\nabla_{x} Q_{\eps}(\,\cdot\,;\,w_0\,)\|_{\underline{L}^{p}(\mathcal{D}_{2 \rho}(z_0))}\\
    &+\| (\d_{\ell} a_{\eps}) \nabla_{x} Q_{\eps}(\,\cdot\,;\,w_0\,)\|_{\underline{L}^{p}(\mathcal{D}_{2 \rho}(z_0))}\;.
    \end{split}
\end{equation}
Combining \eqref{e:R_derivative_expression}, \eqref{e:Q_derivative_C1} and \eqref{e:Q_Hessian_Lp}, we deduce there exist $\beta_3,\beta_4>0$ such that
\begin{equation*}
    \|\nabla^{2}_{x} Q_{\eps}(\,\cdot\,;\,w_0\,)\|_{L^{\infty}(\mathcal{D}_{\rho/2}(z_0))}\lesssim \eps^{-\beta_3} \rho^{-\beta_4},
\end{equation*}
where the proportionality constant does not depend on $z_0,w_0 \in (0,T)\times\T^d$. This implies the bound for $\nabla_x R_\eps$. The bounds for $\nabla_x \widetilde{R}_\eps$ and $\nabla_y \widetilde{R}_\eps$ follow by similar arguments. 
\end{proof}

\section{Preliminaries on oscillations and convolutions}

\begin{lem} \label{lem:oscillation}
    For every $\sigma>0$, we have the bound
    \begin{equation*}
        \|f_\eps\|_{\cC^{-\sigma}(\T^d)} \lesssim \eps^\sigma \|f\|_{L^\infty(\T^d)}\;
    \end{equation*}
    uniformly over $f \in L^\infty(\T^d)$ with $\Pi_0 f =0$ and $\eps = \eps_N = \frac{1}{N}$, where $f_\eps (x) = f(x/\eps)$. 
\end{lem}
\begin{proof}
    We recall the notion of Littlewood-Paley block $\Delta_j$ from Appendix~\ref{sec:para-products} and the definition of $\cC^\gamma$-norm from Definition~\ref{defn:Holder_Besov}. Fix $\eps = \frac{1}{N}$, since $\Pi_0 f = 0$, we know $\widehat{f_\eps}$ is supported on $\{Nk\}_{k \in \Z^d \setminus \{0\}}$. 
    
    Hence, for $j$ such that $2^j \leq \frac{N}{8}$, we have $\Delta_j f_\eps = 0$. For $j$ such that $2^j \geq \frac{N}{8}$, we have
    \begin{equation*}
        \|\Delta_j f_\eps\|_{L^\infty} \lesssim \|f_\eps\|_{L^\infty} \lesssim N^{-\sigma} 2^{\sigma j} \|f\|_{L^\infty\;.}
    \end{equation*}
    This implies
    \begin{equation*}
        \sup_{j \geq -1} \Big( 2^{ -\sigma j} \|\Delta_j f_\eps\|_{L^\infty} \Big) \lesssim N^{-\sigma} \lesssim \eps^\sigma\;.
    \end{equation*}
    The proof is thus complete. 
\end{proof}

\begin{lem} \label{lem:initial}
    Let $0 \leq \gamma \leq \beta \leq 1$. We have
    \begin{equation} \label{e:initial_space}
        \|e^{t \lL_\eps} f\|_{\cC_x^\beta} \lesssim t^{-\frac{\beta-\gamma}{2}} \|f\|_{\cC^\gamma}\;,
    \end{equation}
    and
    \begin{equation} \label{e:initial_time}
        \| \big( e^{t \lL_\eps} - e^{s \lL_\eps} \big) f\|_{L_x^\infty} \lesssim (t-s)^{\frac{\beta}{2}} s^{-\frac{\beta-\gamma}{2}} \|f\|_{\cC^\gamma}\;,
    \end{equation}
    uniformly over $\eps \in \N^{-1}$, $0 \leq s \leq t \leq 1$ and $f \in \cC^\gamma$. In particular, choosing $\beta=\gamma$, we have
    \begin{equation*}
        \|e^{t \lL_\eps} f\|_{\cC_\fs^\gamma} \lesssim \|f\|_{\cC^\gamma}\;.
    \end{equation*}
\end{lem}
\begin{proof}
    We first deal with \eqref{e:initial_space}. For $x, x' \in \T^d$, we have the expression
    \begin{equation*}
        (e^{t \lL_\eps} f)(x) - (e^{t \lL_\eps}f)(x') = \int_{\T^d} \big( Q_\eps(t,x,y) - Q_\eps(t,x',y) \big) \, \big( f(y) - f(x) \big) {\rm d}y\;.
    \end{equation*}
    Using the assumption on $f$, we get
    \begin{equation} \label{e:initial_difference_bound}
        \begin{split}
        &\phantom{11}\big| (e^{t \lL_\eps} f)(x) - (e^{t \lL_\eps}f)(x') \big|\\
        &\lesssim \|f\|_{\cC^\gamma} \int_{\T^d} |y-x|^\gamma \cdot | Q_\eps(t,x,y) - Q_\eps(t,x',y) | \; {\rm d}y\;.
        \end{split}
    \end{equation}
    We split the domain of integration for the right hand side of \eqref{e:initial_difference_bound} into three disjoint regions:
    \begin{equation*}
        \begin{split}
        &\dD_1 = \left\{y: |y-x| \leq \frac{|x-x'|}{2} \right\}\;, \quad \dD_2 = \left\{y: \frac{|x-x'|}{2} < |y-x| \leq 2 |x-x'| \right\}\;,\\
        &\dD_3 = \big\{y: |y-x| > 2 |x-x'| \big\}\;.
        \end{split}
    \end{equation*}
    On $\dD_1$, we have $|y-x| \lesssim |x-x'|$ and $|y-x| \lesssim |y-x'|$. Hence we bound the two factors in the integrand respectively by
    \begin{equation*}
        \begin{split}
        |y-x|^\gamma &\lesssim |x-x'|^\beta \cdot |y-x|^{\gamma-\beta}\;,\\
        |Q_\eps(t,x,y)| + |Q_\eps(t,x',y)| &\lesssim t^{-\frac{d}{2}} \big( e^{-\frac{c|y-x|^2}{t}} + e^{-\frac{c |y-x'|^2}{t}} \big) \lesssim t^{-\frac{d}{2}} e^{-\frac{c |y-x|^2}{t}}\;.
        \end{split}
    \end{equation*}
    This gives
    \begin{equation} \label{e:integral_D1}
        \begin{split}
        &\phantom{111}\int_{\dD_1} |y-x|^\gamma \cdot | Q_\eps(t,x,y) - Q_\eps(t,x',y) | \; {\rm d}y\\
        &\lesssim |x-x'|^\beta \int_{\T^d} t^{-\frac{d}{2}} |y-x|^{-(\beta-\gamma)} e^{-\frac{c|y-x|^2}{t}} {\rm d}t \; \lesssim \; t^{-\frac{\beta-\gamma}{2}} |x-x'|^\beta\;.
        \end{split}
    \end{equation}
    On $\dD_2$, we have $|y-x'| \lesssim |y-x| \lesssim |x-x'|$. Hence we have the pointwise bounds
    \begin{equation*}
        \begin{split}
        |y-x|^\gamma &\lesssim |x-x'|^\beta \cdot |y-x|^{\gamma-\beta} \lesssim |x-x'|^\beta \cdot |y-x'|^{\gamma-\beta}\;,\\
        |Q_\eps(t,x,y)| + |Q_\eps(t,x',y)| &\lesssim t^{-\frac{d}{2}} \big( e^{-\frac{c|y-x|^2}{t}} + e^{-\frac{c |y-x'|^2}{t}} \big) \lesssim t^{-\frac{d}{2}} e^{-\frac{c |y-x'|^2}{t}}
        \end{split}
    \end{equation*}
    for the integrand, which in turn give the integral bound (same as \eqref{e:integral_D1})
    \begin{equation} \label{e:integral_D2}
        \int_{\dD_2} |y-x|^\gamma \cdot | Q_\eps(t,x,y) - Q_\eps(t,x',y) | \; {\rm d}y \; \lesssim \; t^{-\frac{\beta-\gamma}{2}} |x-x'|^\beta\;.
    \end{equation}
    We now turn to the domain $\dD_3$. For the difference between the two heat kernels at different spatial locations, by intermediate value theorem and the gradient estimate in Proposition~\ref{prop:Green_point}, we have
    \begin{equation*}
        \begin{split}
        |Q_\eps(t,x,y) - Q_\eps(t,x',y)| &\leq |x-x'| \cdot \sup_{\theta \in [0,1]} |(\nabla Q_\eps)(t, \theta x + (1-\theta) x', y)|\\
        &\lesssim |x-x'| \cdot t^{-\frac{d+1}{2}} \sup_{\theta \in [0,1]} e^{-\frac{c |y-\theta x - (1-\theta) x'|^2}{t}}\;,
        \end{split}
    \end{equation*}
    where $\nabla Q_\eps$ is the gradient with respect to the first spatial variable of $Q_\eps$. Note that we have
    \begin{equation*}
        |x-x'| \lesssim |y-x| \sim \inf_{\theta \in [0,1]} |y - \theta x - (1-\theta) x'| 
    \end{equation*}
    uniformly over $x,x' \in \T^d$ and $y \in \dD_3$. Hence, we have the pointwise bound
    \begin{equation*}
        \begin{split}
        |Q_\eps(t,x,y) - Q_\eps(t,x',y)| &\lesssim \frac{|x-x'|}{\sqrt{t}} \cdot e^{-\frac{c |x-x'|^2}{t}} \cdot t^{-\frac{d}{2}} \cdot e^{-\frac{c |y-x|^2}{t}}\\
        &\lesssim \Big( \frac{|x-x'|}{\sqrt{t}} \Big)^\beta \cdot t^{-\frac{d}{2}} \cdot e^{-\frac{c |y-x|^2}{t}}\;,
        \end{split}
    \end{equation*}
    where the constant $c$ in the two lines are possibly different. This then gives the integral bound
    \begin{equation} \label{e:integral_D3}
        \begin{split}
        &\phantom{11}\int_{\dD_3} |y-x|^\gamma \cdot | Q_\eps(t,x,y) - Q_\eps(t,x',y) | \; {\rm d}y\\
        &\lesssim \Big( \frac{|x-x'|}{\sqrt{t}} \Big)^\beta \cdot t^{-\frac{d}{2}} \int_{\T^d} |y-x|^\gamma \cdot e^{-\frac{c |y-x|^2}{t}} {\rm d}y \lesssim |x-x'|^\beta \cdot t^{-\frac{\beta-\gamma}{2}}\;.
        \end{split}
    \end{equation}
    Combining the bounds \eqref{e:integral_D1}, \eqref{e:integral_D2} and \eqref{e:integral_D3} and plugging them back to the right hand side of \eqref{e:initial_difference_bound}, we arrive at the desired bound \eqref{e:initial_space}. 
    
    We now turn to \eqref{e:initial_time}. We first note that using the Gaussian bound for $Q_\eps$ in Proposition~\ref{prop:Green_point}, we have
    \begin{equation*}
        \Big| \big( ( e^{t \lL_\eps} - \id) f \big)(x) \Big| \leq \int_{\T^2} |Q_\eps (t,x,y)| \cdot \big| f(y) - f(x) \big| {\rm d}y \lesssim t^{\frac{\gamma}{2}} \|f\|_{\cC^\gamma}\;.
    \end{equation*}
    Hence, combining it with \eqref{e:initial_space}, we have
    \begin{equation*}
        \begin{split}
        \big\| \big( e^{t \lL_\eps} - e^{s \lL_\eps} \big) f \big\|_{L_x^\infty}  &= \big\| \big( e^{(t-s) \lL_\eps} - \id \big) e^{s \lL_\eps} f \big\|_{L_x^\infty}\\
        &\lesssim (t-s)^{\frac{\beta}{2}} \|e^{s \lL_\eps} f\|_{\cC_x^\beta} \lesssim (t-s)^{\frac{\beta}{2}} \cdot s^{-\frac{\beta}{2}} \|f\|_{\cC^\gamma}\;.
        \end{split}
    \end{equation*}
    This proves \eqref{e:initial_time}. The final bound follows immediately by choosing $\beta=\gamma$. 
\end{proof}

\begin{lem} \label{lem:initial_convergence}
    For every $T>0$ and $0<\sigma <\gamma<1$, there exists $\theta>0$ such that
    \begin{equation*}
        \|(e^{t \lL_\eps} - e^{t \lL_0}) f\|_{\cC_\fs^{\gamma-\sigma}([0,T] \times \T^d)} \lesssim \eps^\theta \|f\|_{\cC^\gamma(\T^{d})}\;,
    \end{equation*}
    and
    \begin{equation*}
        \sup_{t \in [0,T]} \Big(t^{\frac{\sigma}{2}} \|(e^{t \lL_\eps} - e^{t \lL_0}) f\|_{\cC_x^\gamma(\T^{d})} \Big) \lesssim \eps^\theta \|f\|_{\cC^\gamma(\T^{d})}\;.
    \end{equation*}
    The proportionality constant depends on $T$. 
\end{lem}
\begin{proof}
    Both bounds follow by combining Green's functions estimates for $R_\eps$ in Propositions~\ref{prop:Green_remainder_point} and~\ref{prop:Green_remainder_derivative}, and the heat kernel bounds in Lemma~\ref{lem:initial}. 
\end{proof}

\begin{lem} \label{lem:initial_error}
Fix $\gamma \in [0,1]$. For every $\sigma>0$, there exist $\beta, \theta>0$ such that
\begin{equation*}
    \big\| (\nabla e^{t \lL_\eps} - \Phi_\eps \nabla e^{t \lL_0}) f \big\|_{\cC_x^\beta} \lesssim \eps^\theta \cdot t^{-\frac{1-\gamma+\sigma}{2}} \, \|f\|_{\cC^\gamma}\;.
\end{equation*}
\end{lem}
\begin{proof}
Fix $\sigma > 0$. Since the kernel of $\nabla e^{t \lL_\eps} - \Phi_\eps \nabla e^{t \lL_0}$ is $R_\eps$ and it integrates to $0$ on the torus with respect to the second spatial variable, we have
\begin{equation*}
    \big( (\nabla e^{t \lL_\eps} - \Phi_\eps \nabla e^{t \lL_0}) f \big)(x) = \int_{\T^d} R_\eps (t,x,y) \big (f(y) - f(x) \big) {\rm d}y\;.
\end{equation*}
Hence, by Proposition~\ref{prop:Green_remainder_point}, there exists $\sigma' \in (0,\sigma)$ and $\theta'>0$ such that
\begin{equation*}
    \begin{split}
    \big\| (\nabla e^{t \lL_\eps} - \Phi_\eps \nabla e^{t \lL_0}) f \big\|_{L_x^\infty} &\lesssim \eps^{\theta'} \sup_{x \in \T^d} \bigg( \int_{\T^d} \frac{|f(y) - f(x)|}{\big( \sqrt{t} + |x-y| \big)^{d+1+\sigma'}} {\rm d}y \bigg) \\
    &\lesssim \eps^{\theta'} \cdot t^{-\frac{1-\gamma+\sigma'}{2}} \|f\|_{\cC^\gamma}\;.
    \end{split}
\end{equation*}
On the other hand, by Proposition~\ref{prop:Green_remainder_derivative}, there exist $\beta_1, \beta_2 > 0$ such that
\begin{equation*}
    \big\| (\nabla e^{t \lL_\eps} - \Phi_\eps \nabla e^{t \lL_0}) f \big\|_{\cC_x^1} \lesssim \eps^{-\beta_1} \cdot t^{-\frac{2+\beta_2}{2}} \|f\|_{\cC^\gamma}\;.
\end{equation*}
Since $\sigma'<\sigma$, interpolating the above two bounds gives the desired claim. 
\end{proof}

\begin{lem} \label{lem:initial_negative}
Let $\beta, \gamma \in (0,1)$. For every $\delta>0$, we have the bounds
\begin{equation} \label{e:initial_negative}
    \|e^{t \lL_\eps} f\|_{\cC_x^\beta} \lesssim t^{-\frac{\beta+\gamma+\delta}{2}} \|f\|_{\cC^{-\gamma}}
\end{equation}
and
\begin{equation*}
    \|(e^{t \lL_\eps} - e^{s \lL_\eps}) f\|_{L_x^\infty} \lesssim (t-s)^{\frac{\beta}{2}} s^{-\frac{\beta+\gamma+\delta}{2}} \|f\|_{\cC^{-\gamma}}\;.
\end{equation*}
\end{lem}
\begin{proof}
We give details for the first bound, and briefly discuss how the second one follows from it. By Green's function bounds in Proposition~\ref{prop:Green_point} and suitable interpolation, we have
\begin{equation*}
    |(e^{t \lL_\eps} f)(x)| \lesssim \|Q_\eps (t,x,y)\|_{\wW_y^{\gamma+\delta,1}} \|f\|_{\cC^{-\gamma}} \lesssim t^{-\frac{\gamma+\delta}{2}} \|f\|_{\cC^{-\gamma}}\;,
\end{equation*}
and
\begin{equation*}
    |(\nabla e^{t \lL_\eps} f)(x)| \lesssim \|\nabla_x Q_\eps (t,x,y\|_{\wW_y^{\gamma+\delta,1}} \|f\|_{\cC^{-\gamma}} \lesssim t^{-\frac{1+\gamma+\delta}{2}} \|f\|_{\cC^{-\gamma}}\;,
\end{equation*}
where $\wW_{y}^{\gamma+\delta,1}$ denotes the standard Sobolev norm in the second spatial variable of $Q_\eps$ (and $\nabla_x Q_\eps$). Interpolating these two bounds gives the desired bound \eqref{e:initial_negative} for $\cC^\beta$ with $\beta \in (0,1)$. As for the time difference, by Lemma~\ref{lem:initial} (with $\beta=\gamma$) and \eqref{e:initial_negative}, we have
\begin{equation*}
    \begin{split}
    \big\| (e^{t \lL_\eps} - e^{s \lL_\eps}) f \big\|_{L_x^\infty} &= \big\| \big( e^{(t-s) \lL_\eps} - \id \big) e^{s \lL_\eps} f \big\|_{L_x^\infty}\\
    &\lesssim (t-s)^{\frac{\beta}{2}} \|e^{s \lL_\eps} f\|_{\cC_x^{\beta}} \lesssim (t-s)^{\frac{\beta}{2}} s^{-\frac{\beta+\gamma+\delta}{2}} \|f\|_{\cC^{-\gamma}}\;.
    \end{split}
\end{equation*}
This completes the proof. 
\end{proof}

\begin{lem} \label{lem:remainder_convolution_negative}
Let $\gamma \in (0,1)$. For every $\delta>0$, there exist $\beta, \theta > 0$ such that
\begin{equation*}
    \big\| \big(\nabla e^{t \lL_\eps} - \Phi_\eps \nabla e^{t \lL_0} \big) f \big\|_{\cC_x^\beta} \lesssim \eps^\theta t^{-\frac{1+\gamma+\delta}{2}} \|f\|_{\cC^{-\gamma}}\;.
\end{equation*}
\end{lem}
\begin{proof}
    We have the expression
    \begin{equation*}
        \big( (\nabla e^{t \lL_\eps} - \Phi_\eps \nabla e^{t \lL_0}) f \big)(x) = \int_{\T^d} R_\eps (t,x,y) f(y) {\rm d}y\;,
    \end{equation*}
    where we recall the kernel $R_\eps$ is given by \eqref{e:difference_first_order}. By duality and interpolation, for every $\delta'>0$, we have
    \begin{equation} \label{e:remainder_convolution}
        \begin{split}
        &\phantom{111}\big\| \big(\nabla e^{t \lL_\eps} - \Phi_\eps \nabla e^{t \lL_0} \big) f \big\|_{\cC_x^\beta}\\
        &\lesssim \Big(\sup_x \|R_\eps(t,x,y)\|_{\wW^{\gamma+\delta',1}_y}^{1-\beta} \Big) \cdot \Big( \sup_x \|\nabla_x R_\eps(t,x,y)\|_{\wW^{\gamma+\delta',1}_y}^{\beta} \Big) \cdot \|f\|_{\cC^{-\gamma}}\;.
        \end{split}
    \end{equation}
    By Propositions~\ref{prop:Green_remainder_point}, for the above $\delta'>0$, there exists $\theta'>0$ such that
    \begin{equation*}
        |R_\eps (t,x,y)| \lesssim \frac{\eps^{\theta'}}{\big( \sqrt{t} + |x-y| \big)^{d+1+\delta'}}\;.
    \end{equation*}
    By the cross derivative estimate in Proposition~\ref{prop:Green_point}, we have
    \begin{equation*}
        |\nabla_y R_\eps (t,x,y)| \lesssim \frac{1}{\big( \sqrt{t} + |x-y| \big)^{d+2}}\;.
    \end{equation*}
    Interpolating the above two bounds gives
    \begin{equation} \label{e:remainder_convolution_pointwise}
        \| R_\eps(t,x,y) \|_{\wW^{\gamma+\delta',1}_y} \lesssim \eps^{\theta'} \cdot t^{-\frac{1+\gamma+2\delta'}{2}}
    \end{equation}
    for some possibly different $\theta'>0$ (depending on $\gamma$ and $\delta'$ only). Also, by Proposition~\ref{prop:Green_remainder_derivative}, there exist $\sigma_1, \sigma_2 > 0$ such that
    \begin{equation} \label{e:remainder_convolution_derivative}
        \|\nabla_x R_\eps (t,x,y)\|_{\wW^{\gamma+\delta',1}_y} \lesssim \eps^{-\sigma_1} t^{-\frac{\sigma_2}{2}}\;.
    \end{equation}
    Now, plugging the bounds \eqref{e:remainder_convolution_pointwise} and \eqref{e:remainder_convolution_derivative} back into \eqref{e:remainder_convolution}, we get
    \begin{equation*}
        \big\| \big(\nabla e^{t \lL_\eps} - \Phi_\eps \nabla e^{t \lL_0} \big) f \big\|_{\cC_x^\beta} \lesssim \eps^{(1-\beta) \theta' - \beta \sigma_1} t^{-\frac{1}{2} \big( (1-\beta) (1+\gamma+2\delta') + \beta \sigma_2 \big)}\;.
    \end{equation*}
    We can choose $\beta$ sufficiently small so that the power in $\eps$ is still positive, and that the power in $t$ is arbitrarily close to $-\frac{1+\gamma+2\delta'}{2}$. Since $\delta'>0$ is arbitrary, this proves the claim. 
\end{proof}

The next two lemmas are the content of \cite[Theorem~1.4]{GengShen2015} and interpolation.

\begin{lem} \label{lem:Lp_critical}
For every $p\in (1,\infty)$ and $T>0$, we have the bound
\begin{equation*}
    \|\iI_{\varepsilon}(f)\|_{L^{p}([0,T];W^{1,p}(\T^{2}))}\lesssim_{T,p} \|f\|_{L^{p}([0,T];W^{-1,p}(\T^2))}\;.
\end{equation*}
\end{lem}

\begin{lem} \label{lem:interpolation}
    Fix $\beta, \gamma>0$. We have
    \begin{equation*}
        \|f\|_{L^\infty (\T^{d})} \lesssim \|f\|_{\cC^\beta(\T^{d})}^{\frac{\gamma}{\beta+\gamma}} \|f\|_{\cC^{-\gamma}(\T^{d})}^{\frac{\beta}{\beta + \gamma}}\;.
    \end{equation*}
\end{lem}

\begin{lem} \label{lem:I_eps_spacetime_norm}
Fix $\beta, \gamma, \sigma \in (0,1)$ with $\beta+\gamma+\sigma<2$. Then, there exists $\theta>0$ such that
\begin{equation*}
    \|\iI_\eps (f)\|_{\cC_\fs^{\beta}([0,T] \times \T^d)} \lesssim T^\theta \|f\|_{L_T^{\infty, \sigma} \cC_x^{-\gamma}}\;.
\end{equation*}
\end{lem}
\begin{proof}
Let $\delta>0$ be sufficiently small such that $\beta + \gamma + \sigma + \delta < 2$. By Lemma~\ref{lem:initial_negative}, we have
\begin{equation*}
    \|(\iI_\eps f)(t)\|_{\cC^\beta} \lesssim \int_{0}^{t} (t-r)^{-\frac{\beta+\gamma+\delta}{2}} \|f(r)\|_{\cC^{-\gamma}} {\rm d}r \lesssim t^{1-\frac{\beta+\gamma+\sigma+\delta}{2}} \|f\|_{L_t^{\infty,\sigma} \cC_x^{-\gamma}}\;.
\end{equation*}
As for continuity in time, we write
\begin{equation*}
    (\iI_\eps f)(t) - (\iI_\eps f)(s) = \int_{0}^{s} \big( e^{(t-r) \lL_\eps} - e^{(s-r) \lL_\eps} \big) f(r) {\rm d}r + \int_{s}^{t} e^{(t-r) \lL_\eps} f(r) {\rm d}r\;.
\end{equation*}
For the first one, again by Lemma~\ref{lem:initial_negative}, we have
\begin{equation*}
    \begin{split}
    &\phantom{111}\int_{0}^{s} \big\| \big( e^{(t-r) \lL_\eps} - e^{(s-r) \lL_\eps} \big) f(r) \big\|_{L^\infty} {\rm d}r\\
    &\lesssim (t-s)^{\frac{\beta}{2}} \int_{0}^{s} (s-r)^{-\frac{\beta+\gamma+\delta}{2}} \|f(r)\|_{\cC^{-\gamma}} {\rm d}r \lesssim (t-s)^{\frac{\beta}{2}} s^{1-\frac{\beta+\gamma+\sigma+\delta}{2}} \|f\|_{L_s^{\infty,\sigma} \cC_x^{-\gamma}}\;.
    \end{split}
\end{equation*}
The second term can be dealt with similarly. This completes the proof. 
\end{proof}

\begin{lem} \label{lem:I_eps_div_spacetime_norm}
For every $\beta \in (0,1)$, there exists $\theta>0$ such that
\begin{equation*}
    \|\iI_\eps (\div f)\|_{\cC_\fs^\beta ([0,T] \times \T^d)} \lesssim T^\theta \|f\|_{L_{T,x}^\infty}
\end{equation*}
uniformly over $f \in L^\infty ([0,T]; L^\infty (\T^d; \, \R^d))$. 
\end{lem}
\begin{proof}
Integrating by parts, we have
\begin{equation*}
    \big( \iI_\eps (\div f) \big)(t,x) = - \int_{0}^{t} \int_{\T^d} (\nabla_y^{\sT} Q_\eps)(t-r,x,y) f(r,y) {\rm d}y {\rm d}r\;.
\end{equation*}
We first consider the $L_T^\infty \cC_x^\beta$ norm. By the pointwise estimates of $\nabla_y^{\sT} Q_\eps$ and $\nabla_{x,y} Q_\eps$ in Proposition~\ref{prop:Green_point}, we have
\begin{equation*}
    \Big| \int_{\T^d} (\nabla_y^{\sT} Q_\eps)(t-r,x,y) f(r,y) {\rm d}y \Big| \lesssim (t-r)^{-\frac{1}{2}} \|f\|_{L_{t,x}^\infty}
\end{equation*}
and
\begin{equation*}
    \Big| \nabla_x \int_{\T^d} (\nabla_y^{\sT} Q_\eps)(t-r,x,y) f(r,y) {\rm d}y \Big| \lesssim (t-r)^{-1} \|f\|_{L_{t,x}^\infty}\;.
\end{equation*}
Interpolating the two bounds gives
\begin{equation*}
    \Big\| \int_{\T^d} (\nabla_y^{\sT} Q_\eps)(t-r,\cdot,y) f(r,y) {\rm d}y\Big\|_{\cC_x^\beta} \lesssim (t-r)^{-\frac{1+\beta}{2}} \|f\|_{L_{t,x}^\infty}\;.
\end{equation*}
Integrating $r \in [0,t]$ then gives the $L_T^\infty \cC_x^\beta$ bound. For the $\cC_t^{\beta/2} L_x^\infty$ part, we write
\begin{equation*}
    \begin{split}
    \big( \iI_\eps (\div f) \big)(t,x) - &\big( \iI_\eps (\div f) \big)(s,x) = - \int_{s}^{t} \int_{\T^{d}} (\nabla_y^\sT Q_\eps)(t-r,x,y) f(r,y) {\rm d}y {\rm d}r\\
    &- \int_{0}^{s} \int_{\T^d} f(r,y)  \bigg[ \int_{\T^d} Q_\eps (t-s,x,z)\\
    &\cdot \Big( (\nabla_y^{\sT} Q_\eps)(s-r,z,y) - (\nabla_y^{\sT} Q_\eps)(s-r,x,y) \Big) {\rm d}z \bigg] {\rm d}y {\rm d}r\;.
    \end{split}
\end{equation*}
The desired bounds again follow from the estimates for $Q_\eps$, $\nabla_y^\sT Q_\eps$ and $\nabla_{x,y}^2 Q_\eps$. This completes the proof. 
\end{proof}

\begin{lem} \label{lem:I0_time_derivative}
Suppose $0 < \beta < \gamma < 1$. Then there exists $\theta>0$ such that
\begin{equation*}
    \|\iI_0 (\d_t f)\|_{\cC_\fs^\beta([0,T] \times \T^d)} \lesssim T^\theta \|f\|_{\cC_\fs^\gamma}\;.
\end{equation*}
\end{lem}
\begin{proof}
Note that we have the identity
\begin{equation*}
    \big( \iI_0 (\d_t f) \big)(t) = \int_{0}^{t} e^{(t-r) \lL_0} \, \d_r f(r) {\rm d}r = f(t) - e^{t \lL_0} f(0) + (\iI_0 \lL_0 f)(t)\;.
\end{equation*}
The conclusion then follows. 
\end{proof}

\section{Bony's decomposition and para-products}
\label{sec:para-products}

Para-products were introduced by Bony in \cite{Bony1981} to deal with multiplications in nonlinear PDEs. In this appendix, we introduce Bony's decomposition and para-products needed in this article, mostly based on the exposition in \cite[Section~2]{Fourier_para}. 

Let $\widetilde{\rho}$ and $\rho$ be compactly supported smooth functions from $\R^d$ to $[0,1]$ such that
\begin{itemize}
    \item $\widetilde{\rho} + \sum_{j=0}^{\infty} \rho(2^{-j} \cdot)\equiv 1$\;; 
    \item $\supp (\widetilde{\rho}) \, \cap \, \supp \big(\rho(2^{-j} \cdot) \big) = \emptyset$ for every $j \geq 1$\;;
    \item $\supp \big(\rho(2^{-i} \cdot) \big) \, \cap \, \supp \big(\rho( 2^{-j} \cdot)\big) = \emptyset$ whenever $|i-j| \geq 2$\;. 
\end{itemize}
Write $\rho_j = \rho(2^{-j} \cdot)$ for $j \geq 0$, and $\rho_{-1} := \widetilde{\rho}$. Using the dyadic decomposition defined above, we can define Littlewood-Paley blocks $\Delta_{j}$ as 
\begin{align*}
    \widehat{\Delta_j f}(k)  :=  \rho_j(k) \widehat{f}(k)
\end{align*}
for all $j\geq -1$, $k \in \Z^d$, where $\widehat{f}(k)$ is the $k$-th Fourier coefficient of $f$. 

\begin{defn} \label{defn:Holder_Besov}
For every $\gamma \in \R$, the H\"older-Besov norm $\cC^\gamma (\T^d)$ is defined by
\begin{equation*}
    \|f\|_{\cC^\gamma} := \sup_{j \geq -1} \left( 2^{\gamma j} \cdot \big\| \Delta_j f \big\|_{L^\infty} \right)\;.
\end{equation*}
\end{defn}

It is well known that for non-integer $\gamma>0$, the above $\cC^\gamma$-norm is equivalent to classical H\"older-$\gamma$ norm, and that $\cC^1$ defined in this way is equivalent to the following H\"older-Zygmund norm (see \cite[Chapter~2]{YSSpringer}),
\begin{equation*}
    \|f\|_{HZ} :=  \|f\|_{L^{\infty}(\T^{d})}+\sup_{x\neq y\in \T^{d}}|x-y|^{-1}|\frac{f(x)+f(y)}{2}-f(\frac{x+y}{2})|\;.
\end{equation*}
For $\m \geq 1$, recall from \eqref{e:Littlewood_Paley_cutoff} the operators $\P_\m$ and $\P_\m^\perp$ that
\begin{equation*}
    \P_\m f := \sum_{j \, \leq \, \log_2 \m} \Delta_j f\;, \qquad \P_\m^\perp f := f - \P_\m f\;.
\end{equation*}
We have the following lemma.

\begin{lem} \label{lem:Besov_cutoff}
Fix $\gamma \in \R$ and $\theta \geq 0$. We have
\begin{equation} \label{e:truncate_high}
    \|\P_\m^\perp f\|_{\cC^{\gamma-\theta}} \lesssim \m^{-\theta} \|f\|_{\cC^\gamma}
\end{equation}
and
\begin{equation} \label{e:truncate_low}
    \|\P_\m f\|_{\cC^{\gamma+\theta}} \lesssim \big(1 + \m^{\theta} \big) \|f\|_{\cC^\gamma}
\end{equation}
uniformly over $f \in \cC^\gamma (\T^d)$ and $\m \in \N$. 
\end{lem}
\begin{proof}
Both bounds follow directly from the definitions of $\P_\m$, $\P_\m^\perp$ and the $\cC^\gamma$ norm, so we omit the details. 
\end{proof}

Define the para-products $\prec$, $\succ$ and $\circ$ by
\begin{equation*}
    f \prec g = g \succ f := \sum_{i \leq j-2} \Delta_i f \cdot \Delta_j g\;, \qquad f \circ g := \sum_{|i-j| \leq 1} \Delta_i f \cdot \Delta_j g\;.
\end{equation*}
Then the product $f g$ of a $\cC^{\beta}$ distribution $f$ and a $\cC^{\gamma}$ distribution $g$ with $\beta+\gamma>0$ can be written as
\begin{align*}
    f g  = f \prec g + f \circ g + f \succ g\;.
\end{align*}
The classical property of para-products is given as follows.

\begin{lem} \label{lem:Bony}
We have
\begin{equation*}
    \begin{split}
    &\|f \prec g\|_{\cC^{\gamma}} \lesssim \|f\|_{L^\infty} \|g\|_{\cC^\gamma}\;, \qquad \gamma \in \R\; ;\\
    &\|f \circ g\|_{\cC^{\beta+\gamma}} \lesssim \|f\|_{\cC^\beta} \|g\|_{\cC^\gamma}\;,\qquad \beta + \gamma > 0\;;\\
    &\|f \succ g\|_{\cC^{\beta+\gamma}} \lesssim \|f\|_{\cC^\beta} \|g\|_{\cC^\gamma}\;, \qquad \gamma<0\;.
    \end{split}
\end{equation*}
All proportionality constants are uniform over the functions $f$ and $g$ in the respective classes. 
\end{lem}

Finally, to simplify notations, we also write
\begin{equation*}
    f \preceq g = f \prec g + f \circ g\;, \qquad f \succeq g = f \circ g + f \succ g\;.
\end{equation*}

\section{Derivation of $\widetilde{\tT}_{\eps}$}
\label{sec:T}

The purpose of this section is to derive the expression \eqref{e:functional_T_tilde_eps} for $\widetilde{\tT}_\eps (u_\eps, v_\eps, w_\eps) = \d_t \big( g(u_\eps) \big)$, with $u_\eps$ being the ``solution" to \eqref{e:gPAM_formal} and whose terms are all well defined in desired regularity spaces. Such an expression necessarily involves the triple $(u_\eps, v_\eps, w_\eps)$. 

We first obtain an effective expression for the right hand side of
\begin{equation*}
    \d_t u_\eps = \div (a_\eps \nabla u_\eps) + g(u_\eps) \big( \xi - \infty_\eps \, g'(u_\eps) \big)\;.
\end{equation*}
Note that we have
\begin{equation*}
    \nabla u_\eps = g(u_\eps) \nabla X_\eps + \aA_\eps (u_\eps, v_\eps, w_\eps)\;.
\end{equation*}
Substituting it back into the first term on the right hand side above, and using \eqref{e:gPAM_formal_RHS} for the second term, we get
\begin{equation*}
    \begin{split}
    \d_t u_\eps = \div \big( a_\eps \aA_\eps \big) + \sS_\eps\;,
    \end{split}
\end{equation*}
where $\sS_\eps = \sS_\eps (u_\eps, v_\eps, w_\eps)$ is defined in \eqref{e:functional_S_eps}. To get an expression for $\d_t \big( g(u_\eps) \big)$, we need to multiply $g'(u_\eps)$ on both sides above. While the product $g'(u_\eps) \sS_\eps$ is a priori well defined, the one between $g'(u_\eps)$ and $\div (a_\eps \aA_\eps)$ is not since the latter lives in $\cC^{-1}$ while the former only has $\cC^{1-}$ regularity. But we can integrate by parts to get
\begin{equation*}
    \d_t \big( g(u_\eps) \big) = \div \big( g'(u_\eps) a_\eps \aA_\eps \big) - g''(u_\eps) \nabla^{\sT} \, u_\eps \, a_\eps \aA_\eps + g'(u_\eps) \sS_\eps\;.
\end{equation*}
Using again the above expression for $\nabla u_\eps$, we get
\begin{equation*}
    \begin{split}
    \d_t \big( g(u_\eps) \big) = &\div \big( g'(u_\eps) a_\eps \aA_\eps \big) + g'(u_\eps) \sS_\eps\\
    &- g''(u_\eps) \big( \aA_\eps^\sT a_\eps \aA_\eps + g(u_\eps) \cdot \aA_\eps^{\sT} a_\eps^{\sT} \nabla X_\eps \big)\;,
    \end{split}
\end{equation*}
which is precisely the expression \eqref{e:functional_T_tilde_eps}.

\printbibliography

\bigskip
\bigskip

\textbf{Ethics declaration:} not applicable. 
\\
\\
\\
\textsc{Beijing International Center for Mathematical Research, Peking University, 5 Yiheyuan Road, Haidian District, Beijing, 100871, China}. 
\\
Email: yilin\_chen@pku.edu.cn
\\
\\
\textsc{Department of Mathematics, Louisiana State University, Baton Rouge, LA 70802, USA}. 
\\
Email: fehrman@math.lsu.edu
\\
\\
\textsc{Institute for Theoretical Sciences, Westlake University, 600 Dun Yu Road, Westlake District, Hangzhou, Zhejiang Province, 310030, China}. 
\\
Email: xuweijun@westlake.edu.cn

\end{document}